\documentclass[11pt,a4paper]{article}

\long\def\drop#1{}

\usepackage[latin1]{inputenc}
\usepackage{amsmath}
\usepackage{amsthm}
\usepackage{amssymb}
\usepackage{latexsym}
\usepackage{graphicx,subfig}
\usepackage{mathrsfs}
\usepackage{geompsfi}
\usepackage{psfrag} 
\usepackage{longtable}

\def\star{\ensuremath{*}}
\newcommand{\gme}{\Gamma_\varepsilon}
\renewcommand{\o}{\Omega}

\newcommand{\h}{\hspace{1cm}}

\DeclareMathOperator{\Lip}{Lip}
\DeclareMathOperator{\spt}{supp}

\newcommand{\es}{\varepsilon}
\let\e\es
\newcommand{\vfi}{\varphi}
\newcommand{\g}{\gamma}
\renewcommand{\k}{\kappa}
\newcommand{\mue}{\mu_\es}
\newcommand{\ue}{u_\es}

\newcommand{\F}{\mathscr{F}}
\newcommand{\G}{\mathcal{G}}
\newcommand{\hf}{\mathscr{H}^1}

\renewcommand{\L}{\mathcal{L}}
\newcommand{\R}{\mathbb{R}}
\newcommand{\Z}{\mathbb{Z}}
\newcommand\N{\mathbb{N}}
\renewcommand{\S}{\mathscr{S}}
\newcommand{\tra}{\mbox{\textsl{\textbf{t}}}}	%Optimal Transport
\renewcommand{\mp}{\mbox{\textbf{m}}}			%Mass coordinates
\newcommand{\gp}{\hat{\gamma}_+}
\newcommand{\gm}{\hat{\gamma}_-}

\newcommand{\infn}{\infty}
\newcommand{\ds}{\displaystyle}
\newcommand{\ra}{\rightarrow}
\newcommand{\weakto}{\rightharpoonup}
\newcommand{\ras}{\stackrel{*}{\rightharpoonup}}
\renewcommand{\div}{\mathop{\mathrm{div}}}
\DeclareMathOperator\rank{rank}
\newcommand{\conv}{\! \ast \!}
\newcommand{\ext}{\! \times \!}
\def\pref#1{(\ref{#1})}
\def\qed {{%        set up
   \parfillskip=0pt        % so \par doesnt push \square to left
   \widowpenalty=10000     % so we dont break the page before \square
   \displaywidowpenalty=10000  % ditto
   \finalhyphendemerits=0  % TeXbook exercise 14.32
   \leavevmode             % \nobreak means lines not pages
   \unskip                 % remove previous space or glue
   \nobreak                % don't break lines
   \hfil                   % ragged right if we spill over
   \penalty50              % discouragement to do so
   \hskip1em              % ensure some space
   \null                   % anchor following \hfill
   \hfill                  % push \square to right
   $\square$%              % the end-of-proof mark
   \par}}                  % build paragraph

\def\Xint#1{\mathchoice
   {\XXint\displaystyle\textstyle{#1}}%
   {\XXint\textstyle\scriptstyle{#1}}%
   {\XXint\scriptstyle\scriptscriptstyle{#1}}%
   {\XXint\scriptscriptstyle\scriptscriptstyle{#1}}%
   \!\int}
\def\XXint#1#2#3{{\setbox0=\hbox{$#1{#2#3}{\int}$}
     \vcenter{\hbox{$#2#3$}}\kern-.5\wd0}}
\def\dashint{\Xint-}
\DeclareMathOperator\supp{supp}
\newtheorem{theo}{Theorem}[section]
\newtheorem{lemma}[theo]{Lemma}
\newtheorem{cor}[theo]{Corollary}
\newtheorem{prop}[theo]{Proposition}

\newtheorem{definition}[theo]{Definition}
\newenvironment{defi}{\begin{definition} \upshape}{\qed\end{definition}}

{\bfseries}{\itshape}

\newenvironment{remark}%
  {\par\medbreak\refstepcounter{theo}%
    \noindent\textbf{Remark~\thetheo. }}%
  {\qed\par\medskip}

\makeatletter
\@addtoreset{equation}{section}
\makeatother
\renewcommand{\theequation}{\arabic{section}.\arabic{equation}}

\begin{document}

\title{Stripe patterns in a model for block copolymers}
\author{Mark A. Peletier and Marco Veneroni}
\date{\today}

\maketitle

\begin{abstract}
We consider a pattern-forming system in two space dimensions defined by an energy $\mathcal G_\e$. The functional $\mathcal G_\e$ models strong phase separation in AB diblock copolymer melts, and patterns are represented by $\{0,1\}$-valued functions; the values $0$ and $1$ correspond to the A and B phases. The parameter $\e$ is the ratio between the intrinsic, material length scale and the scale of the domain $\Omega$. We show that in the limit $\e\to0$ any sequence $u_\e$ of patterns with uniformly bounded energy $\mathcal G_\e(u_\e)$ becomes stripe-like: the pattern becomes locally one-dimensional and resembles a periodic stripe pattern of periodicity $O(\e)$. In the limit the stripes become uniform in width and increasingly straight. 

Our results are formulated as a convergence theorem, which states that the functional  $\mathcal G_\e$ Gamma-converges to a limit functional $\mathcal G_0$. This limit functional is defined on fields of rank-one projections, which represent the local direction of the stripe pattern. The functional $\mathcal G_0$ is only finite if the projection field solves a version of the Eikonal equation, and in that case it is the $L^2$-norm of the divergence of the projection field, or equivalently the $L^2$-norm of the curvature of the field. 

At the level of patterns the converging objects are the jump measures $|\nabla u_\e|$ combined with  the projection fields corresponding to the tangents to the jump set. The central inequality from Peletier \& R\"oger, \emph{Archive for Rational Mechanics and Analysis}, to appear, provides the initial estimate and leads to weak measure-function-pair convergence. We obtain strong convergence by exploiting the non-intersection property of the jump set.
\end{abstract}

\noindent\textbf{AMS Cl.} 49J45, 49Q20, 82D60.
\smallskip


\tableofcontents
\section{Introduction}

%\subsection{Questions}
%\begin{itemize}
%\item Change all measures to measures on $\R^2$, and remove sets from integrals?
%\item Reduce length of section on convergence?
%\end{itemize}
\subsection{Striped patterns}

Of all the patterns that nature and science present, striped patterns are in many ways the simplest. Amenable to a one-dimensional analysis,  they are often the first to be analysed and their characterization is the most complete. In many systems stationary stripe patterns are considered to be well understood, with the research effort focusing on either pattern evolution (such as in the Newell-Whitehead-Segel equation) or on defects. 

In this paper we return to a very basic question: can we prove rigorously that `stripes are best' in the appropriate parts of parameter space? The word `best' requires specification, and let us therefore restrict ourselves to stationary points in variational systems, and take `best' to mean `globally minimizing'. Can we prove that stripes are global minimizers? Within the class of one-dimensional structures---those represented by a function of one variable---optimality of one such structure has been shown in for instance the Swift-Hohenberg equation~\cite{MizelPeletierTroy98,Peletier01a,MarcusZaslavski99,MarcusZaslavski02,Peletier99} and in a block copolymer model~\cite{Mueller93,RenWei00,FifeHilhorst01,ChoksiRen05,ChenOshita06,VanGennipPeletier08}. However, when comparing a striped pattern with arbitrary multidimensional patterns we know of no rigorous results, for any system.

\medskip
The work of this paper provides a weak version of the statement `stripes are best' for a specific two-dimensional system that arises in the modelling of block copolymers. This system is defined by an energy $\mathcal G_\e$ that admits locally minimizing stripe patterns of width $O(\e)$. As $\e\to0$, we show that any sequence $u_\e$ of patterns for which $\mathcal G_\e(u_\e)$ is bounded becomes stripe-like. In addition, the stripes become increasingly straight and uniform in width.

\subsection{Diblock Copolymers}

An \emph{AB diblock copolymer} is constructed by grafting two polymers together (called the A and B parts). Repelling forces between the two parts lead to phase separation at a scale that is no larger than the length of a single polymer. In this micro-scale separation patterns emerge, and it is exactly this pattern-forming property that makes block copolymers technologically useful~\cite{RuzetteLeibler05}. 

By modifying the derivation in~\cite[Appendix A]{PeletierRoegerTA} we find the functional 
\begin{equation}
\label{eq:functional}
\F_\e(u) = \left\{ \begin{array}{ll} \displaystyle
\e\int_\Omega |\nabla u| + \frac1\e d(u,1-u),
  \qquad&\mbox{ if $u \in K$,}\vspace{0.25cm}\\
\infty & \mbox{ otherwise.} \end{array} \right.
\end{equation}
Here $\Omega$ is an open, connected, and bounded subset of $\R^2$ with $C^2$ boundary, and
\begin{equation}\label{defk} 
		K:=\left\{u\in BV(\Omega;\{0,1\})\, :\ \dashint_\Omega u(x)\, dx= \frac{1}2 
		\ \text{ and }\ u = 0  \text{ on }  \partial \Omega\,\right\}.
\end{equation}
The interpretation of the function $u$ and the functional $\F_\e$ are as follows.

The function $u$ is a characteristic function, whose support corresponds to the region of space occupied by the A part of the diblock copolymer; %by an assumption of incompressibility %(see e.g. \cite[Sect. 2.1]{matsen02})
 the complement (the support of $1-u$) corresponds to the B part. The boundary condition $u=0$ in $K$ reflects a repelling force between the boundary of the experimental vessel and the A phase. Figure~\ref{fig:intro} shows two examples of admissible patterns. 
\begin{figure}[ht]
\centering
\psfig{figure=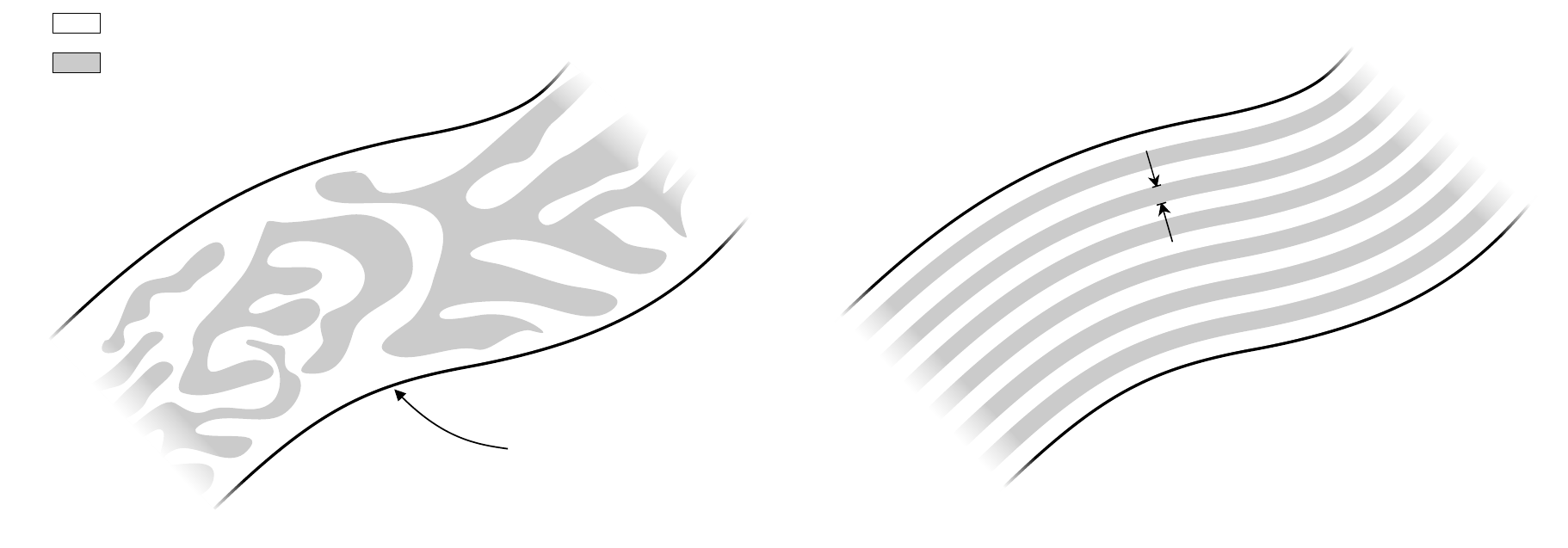,height=5cm}
\caption{A section of a domain $\Omega$ with a general admissible pattern (left) and a stripe-like pattern (right). We prove that in the limit $\e\to0$ all patterns with bounded energy $\mathcal G_\e$ resemble the right-hand picture.}
\label{fig:intro}
\end{figure}

The functional $\F_\e$ contains two terms. The first term penalizes the interface between the A and the B parts, and arises from the repelling force between the two parts; this term favours large-scale separation. In the second term the the Monge-Kantorovich distance~$d$ appears (see~\pref{eq:def-d-1} for a definition); this term is a measure of the spatial separation of the two sets $\{u=0\}$ and $\{u=1\}$, and favours rapid oscillation. The combination of the two leads to a preferred length scale, which is of order $\e$ in the scaling of~\pref{eq:functional}. 

\medskip
The competing long- and short-range penalization in the functional $\F_\e$ is present in many pattern-forming functionals, such as the Swift-Hohenberg and Extended Fisher-Kolmogorov functionals (see~\cite{PeletierTroy07} for an overview). A commonly used energy in the modelling of block copolymers was derived by Ohta and Kawasaki~\cite{OhtaKawasaki86} (see also~\cite{ChoksiRen03}); its sharp-interface limit shares the same interface term with $\F_\e$, and contains a strongly related distance penalization.

%%%%%%%%%%%%%%%%%%%%%%%%%%%%%%%%%%%%%%%%%%%%%%%%%%%%%%%%%%%%%%%%%%%%%%%%%%%%%%%%%%%%%%%%%%%%%
%%%%%%%%%%%%%%%%%%%%%%%%%%%%%%%%  PROPERTIES OF F_e  %%%%%%%%%%%%%%%%%%%%%%%%%%%%%%%%%%%%%%%%
%%%%%%%%%%%%%%%%%%%%%%%%%%%%%%%%%%%%%%%%%%%%%%%%%%%%%%%%%%%%%%%%%%%%%%%%%%%%%%%%%%%%%%%%%%%%%

\subsection{Properties of $\F_\e$}

\drop{
We now leave the derivation of $\F_\e$ aside and turn to its properties. To start with, $\F_\e$ has a distinct preferred length scale. A simple way to recognize this is to take for $\Omega$ a square $[0,L]^2$ and let $u$ be a horizontally striped pattern of periodicity $L/n$,
\[
u(x,y) := 
\begin{cases}
0 & \text{if }\lfloor\frac{2ny}L+ \frac12\rfloor \text{ is even}\\
1 & \text{if }\lfloor\frac{2ny}L+\frac12\rfloor \text{ is odd}
\end{cases}
\]
Then it is not difficult to see that $d(u,1-u)= L^3/16n$ and $\int|\nabla u| = (2n-1)L$. Therefore
\[
\int |\nabla u| + d(u,1-u) = (2n-1)L + \frac{L^3}{16n},
\]
which has a minimum in $n$ at $n^2 = L^2/16$, with period $L/n = 4$. 

\medskip
}

Many of the properties of the functional $\F_\e$ can be understood from the following lower bound. (The description that follows is embellished, and cuts some corners; full details are given in Section~\ref{sec:lowerbound}). Take a sequence $u_\e$, and let us pretend that the interface $\partial \supp u_\e$ consists of a single closed curve $\gamma_\e:[0,L_\e]\to\Omega$, parametrized by arclength $s$. 

The metric $d$ induces a partition of the domain $\Omega$ into roughly-tubular neighbourhoods of $\gamma_\e$, and defines a parametrization of $\Omega$ of the form
\[
(s,m) \mapsto \gamma_\e(s) + t_\e(m;s)\theta_\e(s)
\qquad \text{for }0\leq s\leq L_\e \text{ and } -M_\e(s)< m<M_\e(s).
\]
Here $\theta_\e:[0,L_\e]\to S^1$ is the direction of the rays along which mass is shifted by an optimal transport (see Section \ref{sec:mtp} below), and $m\mapsto t_\e(m;s)$ is an increasing function (see Figure~\ref{fig:rays}). The function $M_\e:[0,L_\e]\to[0,\infty)$ is the area density between two rays, and can be interpreted as (approximately) the width of a tubular neighbourhood. Each such tubular neighbourhood then consists of `half' of a $u_\e$-stripe ($0<m<M_\e(s)$) and half of a $(1-u_\e)$-stripe ($-M_\e(s)<m<0$). 

\begin{figure}[ht]
\psfig{figure=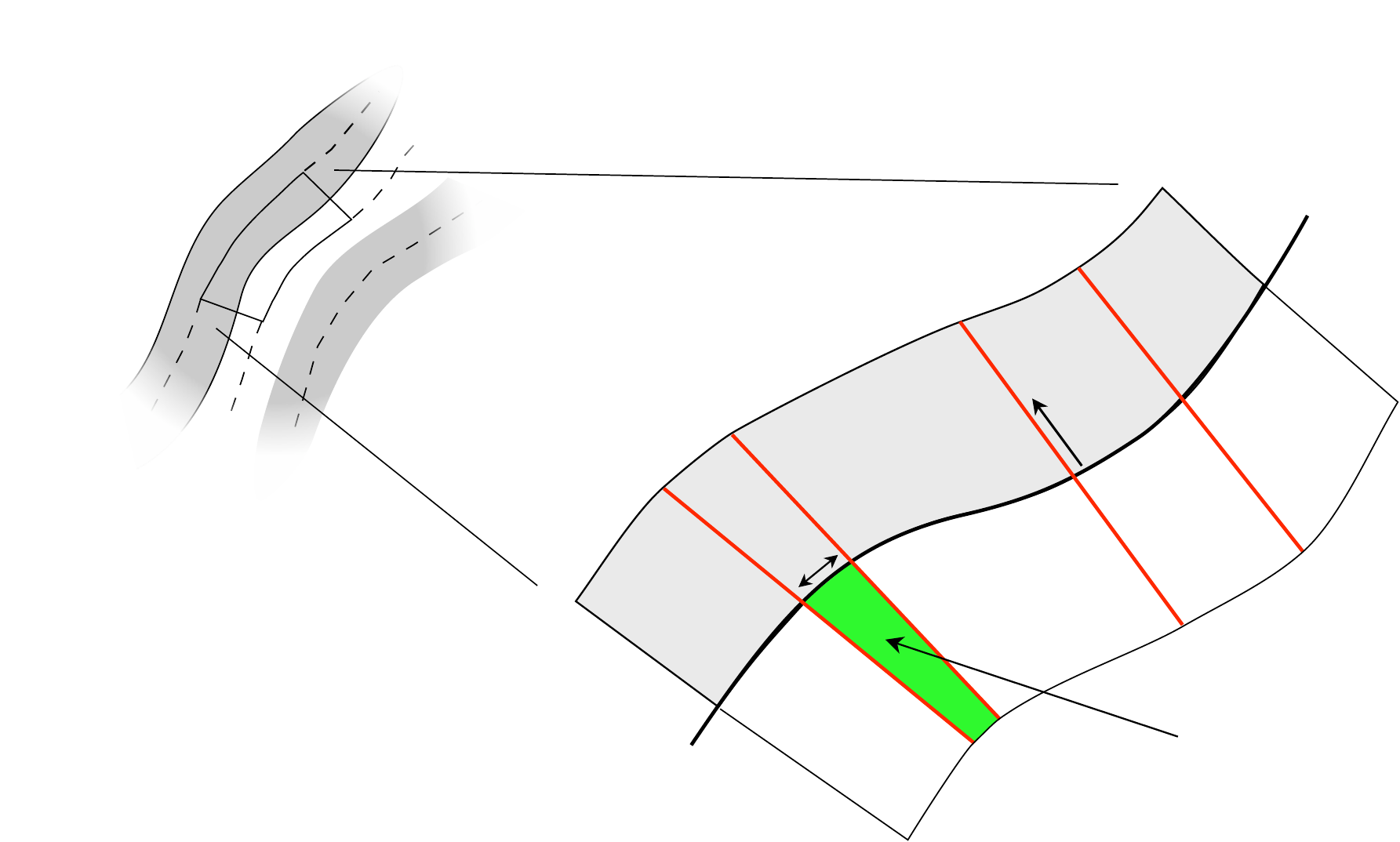,height=7cm}
\caption{The parametrization induced by the distance $d(u,1-u)$. }
\label{fig:rays}
\end{figure}

Using this parametrization we find for the functional $\F_\e$ the (simplified) estimate
\begin{equation}
\label{ineq:basic}
\F_\e(u) - |\Omega| \geq 
\int_0^{L_\e}\left[ \left( \frac{M_\e(s)}{\es} -1\right)^2 
+ \left( \frac 1{\sin \angle(\gamma_\e'(s),\theta_\e(s))}-1\right) 
+ \frac {\e^2}{4}	|\theta_\e'(s)|^2\right]\,\es\,ds.
\end{equation}
In this integral we have joined a factor $\e$ with the length element $ds$, so that the integral satisfies $\int_0^{L_\e} \e \, ds = \es \int_\Omega |\nabla u_\es|\sim 1$. 

In the inequality above, all three terms on the right-hand side are non-negative. If $\F_\e - |\Omega|$ vanishes as $\e\to0$,  then necessarily
\begin{itemize}
\item $M_\e/\e$ converges to $1$, implying that the tubular neighbourhoods become of uniform width $2M_\e\approx 2\e$;
\item $\gamma_\e'(s)$ and $\theta_\e(s)$ become orthogonal at each $s$, which means that $\theta_\e$ becomes a unit normal to $\gamma_\e$.
\end{itemize}
These two properties imply that the final term in~\pref{ineq:basic} is approximately equal to 
\[
\frac{\e^2}4 \int_0^{L_\e} |\gamma_\e''(s)|^2 \, \e ds.
\]

\medskip
With these arguments in mind we introduce a rescaled functional $\mathcal G_\e$ defined by
\begin{equation}
\label{def:rescaled_functional}
\mathcal G_\e(u) := \frac1{\e^2} \Bigl( \F_\e(u) - |\Omega|\Bigr).
\end{equation}
If for a sequence $u_\e$ the rescaled energies $\mathcal G_\e(u_\e)$ are bounded in $\e$, then from the discussion above we expect $u_\e$ to become stripe-like, with stripes that are of width approximately $2\e$; the limit value of the sequence $\mathcal G_\e(u_\e)$ will be related to the curvature of the limiting stripes.

%%%%%%%%%%%%%%%%%%%%%%%%%%%%%%%%%%%%%%%%%%%%%%%%%%%%%%%%%%%%%%%%%%%%%%%%%%%%%%%%%%%%%%%%%%%%%
%%%%%%%%%%%%%%%%%%%%%%%%%%%%%%%%  THE LIMIT PROBLEM  %%%%%%%%%%%%%%%%%%%%%%%%%%%%%%%%%%%%%%%%
%%%%%%%%%%%%%%%%%%%%%%%%%%%%%%%%%%%%%%%%%%%%%%%%%%%%%%%%%%%%%%%%%%%%%%%%%%%%%%%%%%%%%%%%%%%%%

\subsection{The limit problem}\label{subsec:p}

If, as we expect,  $u_\e$ is a sequence of patterns with an increasingly uniform stripe pattern, then the sequence $u_\e$ should converge weakly to its average on $\Omega$, that is $1/2$. This implies that  the sequence of functions $u_\e$ does not capture the directional information that we need in order to define a `straightness' or `curvature' of the limit structure.

The derivative $\nabla u_\e$ does carry information on the direction of the stripes, but it vanishes in the limit, as one can readily verify by partial integration. The interpretation of this vanishing is that interfaces that face each other carry opposite signs and therefore cancel each other. 

In order to counter this cancellation we switch from vectors to \emph{projections}. For the purposes of this paper, a projection will be a symmetric rank-one unit-norm 2-by-2 matrix, or equivalently a matrix $P$ that can be written as $P=e\otimes e$, where $e$ is a unit vector. For $u\in K$ the Radon-Nikodym derivative $d\nabla u/d|\nabla u|$ is a unit vector at $|\nabla u|$-a.e.~$x$, and this allows us to define 
\[
P(x) := \frac{\nabla u^\perp}{|\nabla u|}(x) \otimes \frac{\nabla u^\perp}{|\nabla u|} (x)
\qquad \text{for }|\nabla u|\text{-a.e.~$x$}.
\]
\begin{figure}[htbp]
  \begin{center}
    \psfrag{u0}[c][]{$u=0$}		
    \psfrag{u1}[c][]{$u=1$}
    \psfrag{p}[c][]{$P$}
    \psfrag{n}[c][]{$\ds \nabla u$}
    \psfrag{t}[c][]{$\ds \nabla u^\perp$}
		\includegraphics[scale=1.2]{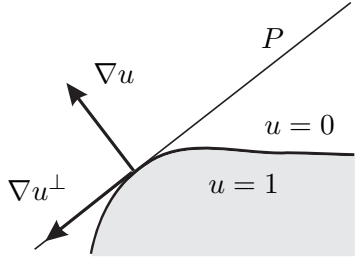}
		\caption{$P$ is the orthogonal projection onto the line normal to $\nabla u$. }
		\end{center}
\end{figure}%
Here and below we write simply $\nabla u/|\nabla u|$ instead of $d\nabla u/d|\nabla u|$, and we use the notation~$e^\perp$ for the rotation over 90 degrees anti-clockwise of the vector $e$. With 
this definition $P$ projects along the vector $\nabla u$ onto the line with direction $\nabla u^\perp$. 

The space $\mathcal P$ of projections is homeomorphic to $P^1$, the projective line, i.e. $S^1/\Z_2$ or $S^1$ with plus and minus identified with each other, something which can be directly recognized by remarking that in $P=e\otimes e$ one can replace $e$ by $-e$ without changing $P$. Since the direction of the stripes in Fig.~\ref{fig:intro} (right) is also only defined up to 180 degrees, this shows why projections are a more natural characterization of stripe directions than unit vectors. 

\medskip

In the limit $\e\to0$ the stripe boundaries become dense in $\Omega$, suggesting that the limit object is a projection $P(x)$ defined at every $x\in \Omega$. Let us assume, to fix ideas, that this $P$ arises from a smooth unit-length vector field $e$, such that 
%Exploring the connection between unit vector fields and their associated projection fields a little further, let $e(x)$ be a unit vector field and set 
$P(x) := e(x)\otimes e(x)$. We keep the interpretation of a stripe field in mind, in which $e(x)$ is the tangent direction of a stripe at $x$. The divergence\footnote{Recall that the divergence of a matrix with elements $a_{ij}$ is the vector $\sum_j \partial_j a_{ij}$.} of $P$ splits into two parts:
\[
\div P =  (\nabla e )\cdot e + e (\div e),
\]
The first of these is the derivative of $e$ in the direction of $e$, and therefore equal to the curvature of the stripe. It follows that this term is orthogonal to the stripe. The second term measures the divergence of the flow field $e$, and since $e$ is unit-length this term measures the relative divergence of nearby stripes. If the stripes are locally parallel, this term should vanish.

Summarizing, if $P$ is the limit projection field, then $\div P$ is expected to contain two terms,  one of which is parallel to the stripe and should vanish, and the other which is orthogonal to the stripe and captures curvature. This serves to motivate the following definition of the admissible set of limit projections~$P$:
\begin{defi}\label{limprob}
$\mathcal{K}_0(\Omega)$ is the set of all $P\in L^2(\Omega;\mathbb{R}^{2 \times 2})$ such that
\begin{eqnarray*}
         P^2 = P && \mbox{a.e. in } \Omega,\\
         \rank P=1 && \mbox{a.e. in } \o,\\
         P \mbox{ is symmetric} && \mbox{a.e. in } \o,\\
         \div P \in L^2(\R^2;\R^2) && (\mbox{extended to $0$ outside  }\o),\\
         P\, \div P = 0  && \mbox{a.e. in } \o.
\end{eqnarray*}
\end{defi}
The first three conditions encode the property that $P$ is a projection field. The fourth one is a combination of a regularity requirement in the interior of $\Omega$ and a boundary condition  on $\partial \Omega$ (see Remark~\ref{rem:divPL2}); %The boundary condition results from the boundary condition $u=0$ in~\pref{defk}; 
we comment on boundary conditions below.  The regularity condition implies that $\div P$ is locally a function, which ensures that the fifth condition is meaningful. That last condition, which reduces to $\div e=0$ in the case discussed above, is exactly the condition of parallel stripes. 

The regularity condition also implies that various singularities in the line fields are excluded. We comment on this issue in the Discussion below.

%%%%%%%%%%%%%%%%%%%%%%%%%%%%%%%%%%%%%%%%%%%%%%%%%%%%%%%%%%%%%%%%%%%%%%%%%%%%%%%%%%%%%%%%%%%%%
%%%%%%%%%%%%%%%%%%%%%%%%%%%%%%%%  THE EIKONAL EQUATION  %%%%%%%%%%%%%%%%%%%%%%%%%%%%%%%%%%%%%
%%%%%%%%%%%%%%%%%%%%%%%%%%%%%%%%%%%%%%%%%%%%%%%%%%%%%%%%%%%%%%%%%%%%%%%%%%%%%%%%%%%%%%%%%%%%%

\subsection{The Eikonal Equation}

As is to be expected from the parallel-stripe property, the set $\mathcal{K}_0(\o)$ can be seen as a \emph{set of solutions of the Eikonal equation}. The Eikonal equation arises in various different settings, and consequently has various formulations and interpretations. For our purposes the important features are listed below. With the stripe pattern in mind we identify at every point two orthogonal vectors, the \emph{tangent} (which would be $e$ above) and the \emph{normal}. Naturally this identification leaves room for the choice of sign, but since our application is stated in terms of projections rather than vectors this will pose no problem. 

Elements of $\mathcal K_0(\Omega)$ satisfy
\begin{itemize}
\item \emph{tangents propagate along normals:} along the straight line parallel to the normal in $x_0$, the tangents are constant and equal to the tangent in $x_0$
\item \emph{the boundary $\partial \Omega$ is tangent}: the stripes run parallel to the boundary $\partial \Omega$.
\end{itemize}

This leads to the following existence and uniqueness theorem, which we prove in a separate paper using results from~\cite{JabinOttoPerthame02}:
\begin{theo}[\cite{PeletierVeneroniTA}]
\label{th:ee}
Among domains $\Omega$ with $C^2$ boundary, $\mathcal K_0(\Omega)$ is non-empty if and only if $\Omega$ is a tubular domain. In that case $\mathcal K_0(\Omega)$ consists of a single element.
\end{theo}
A \emph{tubular domain} is a domain in $\R^2$ that can be written as
\[
\Omega = \Gamma + B(0,\delta),
\]
where $\Gamma$ is a closed curve in $\R^2$ with curvature $\kappa$ and $0<\delta< \|\kappa\|^{-1}_\infty$. In this case the \emph{width} of the domain is defined to be $2\delta$. The unique element $P\in \mathcal{K}_0(\Omega)$ in the theorem is given by
\[
P(x) = \tau(\pi x)\otimes \tau(\pi x),
\]
where $\pi:\Omega\to\Gamma$ is the orthogonal projection onto $\Gamma$ (which is well-defined by the assumption on $\delta$) and $\tau(x)$ is the unit tangent to $\Gamma$ at $x$.

The reason why Theorem~\ref{th:ee} is true can heuristically be recognized in a simple picture. Figure~\ref{fig:ee} shows two sections of $\partial \Omega$ with a normal line that connects them. By the first property above, the stripe tangents are orthogonal to this normal line; by the second, this normal line is orthogonal to the two boundary segments, implying that the two segments have the same tangent. Therefore the length of the connecting normal line is constant, and as it moves it sweeps out a full tubular neighbourhood. 
\begin{figure}[ht]
\indent\centering
\psfig{figure=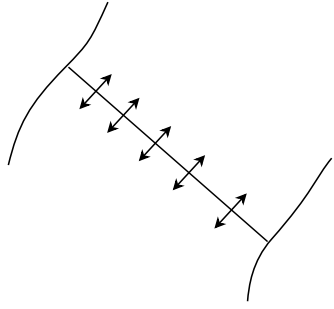,height=4cm}
\caption{If tangent directions propagate normal to themselves, and if in addition the boundary is a tangent direction, then the domain is tubular (Theorem~\ref{th:ee}).}
\label{fig:ee}
\end{figure}

In order to introduce the limit functional, define the space of bounded measure-function pairs on $\o$:
\begin{equation}
\label{def:X}
X:=\left\{ (\mu,P):\ \mu\in RM(\o),\ P \in L^{\infty}(\o,\mu;\R^{2 \times 2})\right\}.
\end{equation}
Here $RM(\Omega)$ is the space of Radon measures on $\Omega$. With the definition of $\mathcal{K}_0(\o)$ in hand we now define the limit functional $\G_0:X\ra \R,$
\begin{equation}
\label{def:G0}
\G_0(\mu,P):=\left\{\begin{array}{cl}
	\ds \frac 14 \int_\o |\div P(x)|^2d\mu(x) & \mbox{if }\mu=\frac12 \L^2\llcorner \o \mbox{ and }P\in \mathcal{K}_0(\o)\\
	+\infty & \mbox{otherwise}
\end{array} \right.
\end{equation}
%$$\G_0:\mathcal{K}_0(\o)\ra \R,$$
%$$ \G_0(P):=\frac 18 \int_\o |\div P(x)|^2dx.$$
Here $\L^2$ is two-dimensional Lebesgue measure. For the case of $\mu=\frac12 \L^2\llcorner \o$, $P=e\otimes e$, we have $\mathcal G_0(\mu,P) = 1/8 \int |(\nabla e)\cdot e|^2$: the functional $\mathcal G_0$ measures the curvature of stripes. 

%%%%%%%%%%%%%%%%%%%%%%%%%%%%%%%%%%%%%%%%%%%%%%%%%%%%%%%%%%%%%%%%%%%%%%%%%%%%%%%%%%%%%%%%%%%%%
%%%%%%%%%%%%%%%%%%%%%%%%%%%%%%%%%%%  THE MAIN RESULT   %%%%%%%%%%%%%%%%%%%%%%%%%%%%%%%%%%%%%%
%%%%%%%%%%%%%%%%%%%%%%%%%%%%%%%%%%%%%%%%%%%%%%%%%%%%%%%%%%%%%%%%%%%%%%%%%%%%%%%%%%%%%%%%%%%%%

\subsection{The main result}

The main result of this paper states that $\G_\e$ converges in the Gamma-convergence sense to the functional $\G_0$. We first give the exact statement. %(Notation: in the following, for a given $u_n \in K,$ $P_n(x)\in\R^{2\times 2}$ is the projection onto the tangent of $|\nabla u_n|$ at $x$, as in Section \ref{subsec:p}.)

\begin{theo}\label{thetheorem}
Let $\Omega$ be an open,  connected subset of $\R^2$ with $C^2$ boundary. %\marginpar{change to $C^3$?}
\begin{enumerate}
\item\label{thetheorem:part1}
	\textup{(Compactness)}	
For any sequence $\e_n\to0$, let a family $\{u_n\}\subset K$ satisfy
	$$ \limsup_{n\to\infty} \G_{\es_n}(u_n)<\infty.$$
Then there exists a subsequence, denoted again $\e_n$, such that
	\begin{eqnarray}
	    \label{conv:u-weak}
		&u_n \ras \frac 12 & \mbox{weakly-\star}\ \mbox{in}\ L^{\infty}(\o),\\
		&\mu_n:= \es_n|\nabla u_n|\ras \mu:=\frac12 \L^2\,\llcorner\,\o & \mbox{weakly-\star}\ \mbox{in}\ RM(\o).\notag
	\end{eqnarray}		
	Let $P_n(x)\in\R^{2\times 2}$ be the projection onto the tangent of $\mu_n$ at $x$. Then there exists a $P\in \mathcal{K}_0(\o)$ such that 
	\begin{equation}\label{pstrong}
		(\mu_n,P_n) \ra (\mu,P)\h \mbox{strongly in $L^2$, in the sense of Definition \ref{def:strongconv}}. 
	\end{equation}				
\item \label{thetheorem:part2}
\textup{(Lower bound)} For every measure-function pair $(\mu,P)\in X$ and for every sequence $\{u_n\}\subset K$, $\es_n \to 0$ such that
$$ (\es_n|\nabla u_n|,P_n) \weakto (\mu,P)\h \mbox{weakly in $L^2$, in the sense of Definition \ref{def:weak-coup-conv}}, $$
	it holds
	\begin{equation}
	\label{ineq:lowerboundintro}
		\liminf_{n\to\infty}\, \G_{\es_n}(u_n)\, \geq\, \G_0(\mu,P).	
	\end{equation} 					
\item\label{thetheorem:part3}
\textup{(Upper bound)}  Let $\Omega$ be a tubular neighbourhood of width $2\delta$, with boundary $\partial \Omega$ of class $C^3$. Let the sequence $\e_n\to0$ satisfy 
\begin{equation}
\label{cond:resonance}
\delta/2\e_n\in \N.
\end{equation}
If $P\in \mathcal{K}_0(\o)$, then there exists a sequence $\{u_n\} \subset K$ such that
%    		$$ \mu_n := \es_n|\nabla u_n|\ras \frac12 \mathcal{L}\, \llcorner\, \o\quad \mbox{ in }RM(\o), $$
	\begin{eqnarray*}
		&u_n \ras \frac 12 & \mbox{weakly-\star}\ \mbox{in}\ L^{\infty}(\o),\\
		&\mu_n:= \es_n|\nabla u_n|\ras  \mu:=\frac12\L^2\,\llcorner\,\o & \mbox{weakly-\star}\ \mbox{in}\ RM(\o).
	\end{eqnarray*}	
	As above, let $P_n(x)\in\R^{2\times 2}$ be the projection onto the tangent of $\mu_n$ at $x$. Then
    		\begin{equation*}
            	(\mu_n,P_n) \ra (\mu,P)\h \mbox{strongly in $L^2$, in the sense of Definition \ref{def:strongconv}},
    		\end{equation*}
				and
\begin{equation}
\label{ineq:upperboundintro}
\limsup_{n\to\infty}\, \G_{\es_n}(u_n)\leq \G_0(\mu,P).
\end{equation}
\end{enumerate}	
\end{theo}

This theorem can be summarized by the statement that $\G_{\e_n}$ Gamma-converges to $\G_0$, provided $\e_n$ satisfies~\pref{cond:resonance}. The underlying concept of convergence is given by the measure-function-pair convergence of the pair $(\mu_n, P_n)$ in combination with the condition $u_n\weakto 1/2$. 
\begin{remark}
The convergence employed in the \textit{liminf} inequality (point \ref{thetheorem:part2}) is weaker than then convergence required for the \textit{limsup} inequality (point \ref{thetheorem:part3}). This kind of asymmetric convergence is also called Mosco-convergence and was introduced in \cite{Mosco67} for bilinear forms on Hilbert spaces. In general it is not weaker than $\Gamma$-convergence in the strong topology; if a strong (asymptotic) compactness property holds, as in point \ref{thetheorem:part1}, then the two notions of Mosco- and $\Gamma$-convergence are equivalent \cite[Lemma 2.3.2]{Mosco94}.
\end{remark}
\begin{remark}
There is an asymmetry in Theorem~\ref{thetheorem} in the conditions on $\Omega$ and $\e_n$: while the lower bound states no requirements on $\Omega$ and $\e_n$, the upper bound requires (a) that $\Omega$ is tubular, and (b) that $\e_n$ is related to the width of the tube, and (c) that $\Omega$ has higher regularity ($C^3$). 

Part of this asymmetry is only appearance. The tubular nature of $\Omega$ is actually also required in the lower bound, but this requirement is implicit in the condition that $\mathcal{K}_0(\o)$ is non-empty; put differently, the sequence $\mathcal G_{\e_n}(u_n)$ can only be bounded if $\Omega$ is tubular.
We comment on this issue, as well as condition~\pref{cond:resonance}, in the next section. The regularity condition on $\Omega$, on the other hand, constitutes a real difference between the upper and lower bound results. It arises from higher derivatives in the construction of the recovery sequence, and this issue is further discussed in Remark~\ref{rem:higher-regularity}.
\end{remark}

%%%%%%%%%%%%%%%%%%%%%%%%%%%%%%%%%%%%%%%%%%%%%%%%%%%%%%%%%%%%%%%%%%%%%%%%%%%%%%%%%%%%%%%%%%%%%
%%%%%%%%%%%%%%%%%%%%%%%%%%%%%%%%%%%%  DISCUSSION  %%%%%%%%%%%%%%%%%%%%%%%%%%%%%%%%%%%%%%%%%%%
%%%%%%%%%%%%%%%%%%%%%%%%%%%%%%%%%%%%%%%%%%%%%%%%%%%%%%%%%%%%%%%%%%%%%%%%%%%%%%%%%%%%%%%%%%%%%

\subsection{Discussion}

As described above, the aim of this paper is to prove a weak version of the statement `stripes are best'. The convergence result of Theorem~\ref{thetheorem} makes this precise.

The theorem characterizes the behaviour of a sequence of structures~$u_n$ for which $\F_{\e_n}(u_n)- |\Omega| = O(\e_n^2)$, or equivalently, $\G_{\e_n}(u_n) = O(1)$. Such structures become stripe-like, in the sense that
\begin{itemize}
\item the interfaces between the sets $\{u_n=0\}$ and $\{u_n=1\}$ become increasingly parallel to each other,
\item the spacing between the interfaces becomes increasingly uniform, and
\item the limit value of the energy $\G_{\e_n}(u_n)$ along the sequence is the squared curvature of the limiting stripe pattern.
\end{itemize}
The first property corresponds to the statement~\pref{pstrong} that $(\mu_n,P_n)\to (\mu,P)$ in the strong sense, and the third %statement 
one is contained in the combination of~\pref{ineq:lowerboundintro} and~\pref{ineq:upperboundintro}. The second property appears in a weak form in the weak convergence~\pref{conv:u-weak} of $u_n$ to $1/2$, and in a stronger form in the statement $M_\e/\e\to 1$ after Proposition~\ref{prop:inequality}.

\medskip

A slightly different way of describing Theorem~\ref{thetheorem} uses a vague characterization of stripe patterns in the plane---see Figure~\ref{fig:types_of_variation}. 
\def\hht{2cm}%
\def\spacing{\hskip1cm}%
%\begin{figure}[ht]
%\centering
%[caption][\centering width variation]{\includegraphics[height=\hht,clip=on]{change_in_width-cut}\label{subfig:cw}}
%\spacing
%\subfloat[caption][\centering grain boundary]{\includegraphics[height=\hht,clip=on]{grainboundary3-cut}\label{subfig:gb}}
%\spacing
%\subfloat[caption][\centering target and U-turn patterns]{\includegraphics[height=\hht,clip=on]{target-cut}\ \ \includegraphics[height=\hht,clip=on]{u-turn-cut}\label{subfig:target}}
%\spacing
%\subfloat[caption][\centering smooth directional variation]{\includegraphics[height=\hht,clip=on]{smoothvariation-cut}}
%\caption{Canonical types of stripe variation in two dimensions. }
%\label{fig:types_of_variation}
%\end{figure}%

%%%  handmade subfloat %%%
\begin{figure}[ht]
\begin{minipage}{2,9cm}
\centering{
\includegraphics[height=\hht,clip=on]{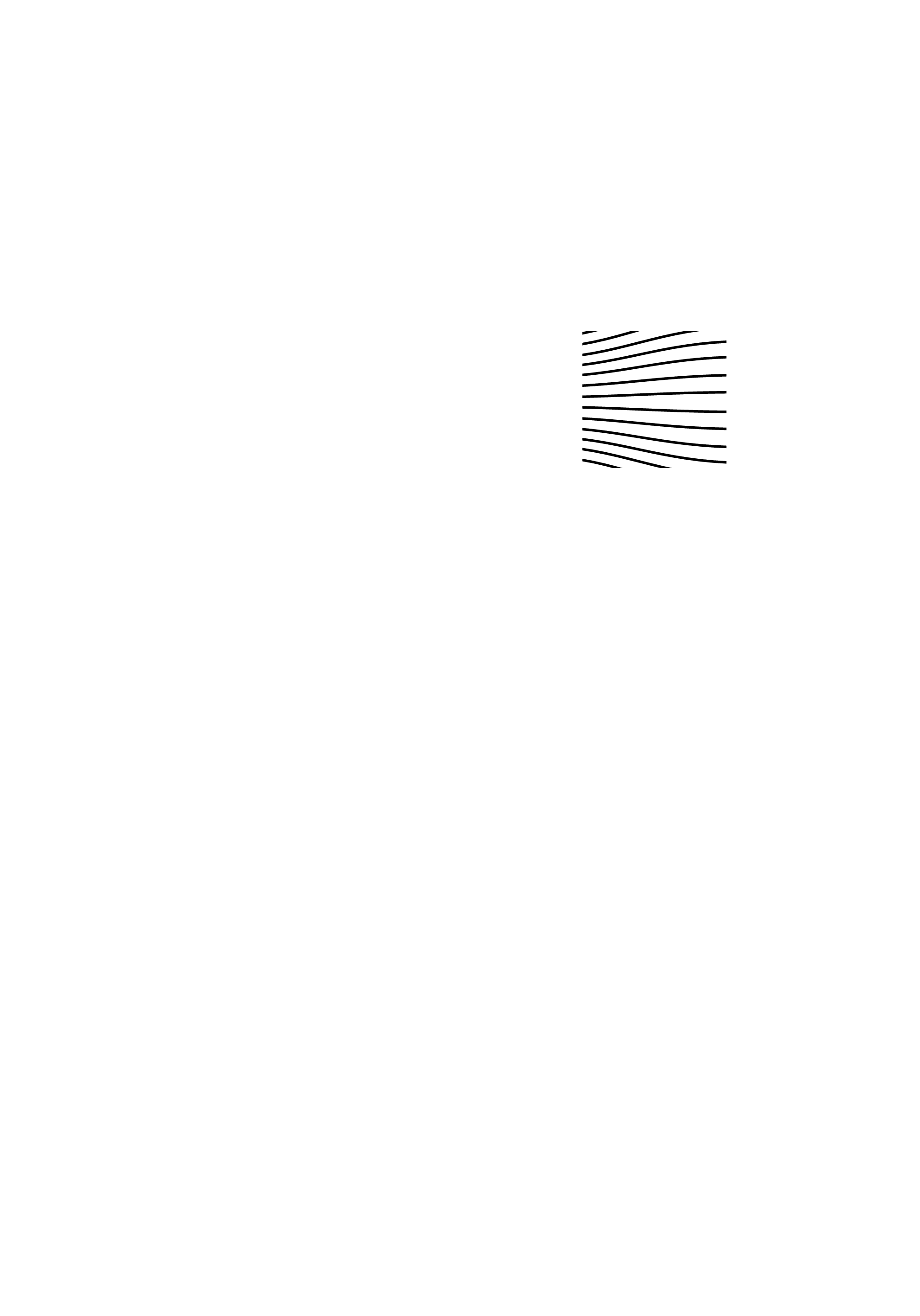}\label{subfig:cw}\\
\small a)\! width variation \phantom{indent}}
\end{minipage}
\begin{minipage}{2,9cm}
\centering{
\includegraphics[height=\hht,clip=on]{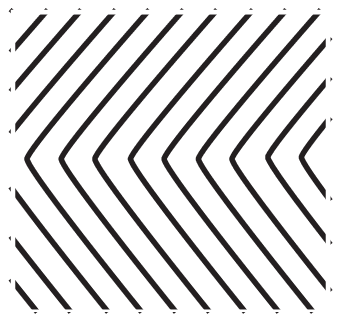}\label{subfig:gb}\\
\small b)\! grain boundary \phantom{indent}}
\end{minipage}
\begin{minipage}{5.3cm}
\centering{
\includegraphics[height=\hht,clip=on]{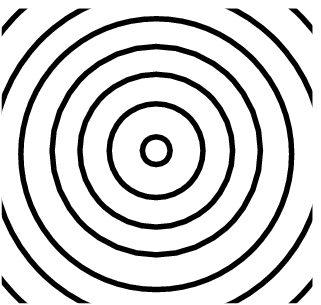}\hspace{0.3cm}\includegraphics[height=\hht,clip=on]{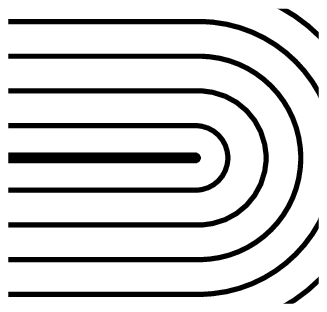}\label{subfig:target}\\
\small c)\! target and U-turn patterns \phantom{indent}}
\end{minipage}
\begin{minipage}{3,1cm}
\centering{
\includegraphics[height=\hht,clip=on]{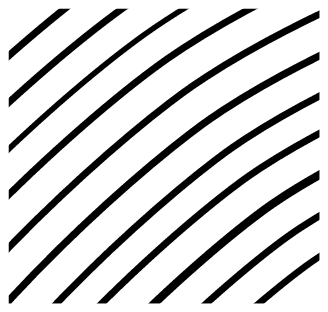}\\
\small d)\! smooth directional variation}
\end{minipage}\\
\caption{Canonical types of stripe variation in two dimensions. }
\label{fig:types_of_variation}
\end{figure}
Theorem~\ref{thetheorem} states that the  decay condition $\F_{\e_n}(u_n) - |\o| = O(\e_n^2)$  excludes all but the last type. This can also be recognized from a formal calculation based on~\pref{ineq:basic}, which shows that width variation is penalized  by $\F_\e$ at order $O(1)$, grain boundaries at order $O(\e)$, and the target and U-turn patterns at order $O(\e^2|\log\e|)$.

If one interprets the figures in Figure~\ref{fig:types_of_variation} not as discrete stripes but as a visualization of a line field $P$ that is defined everywhere, then the condition $\div P\in L^2$ similarly excludes all but the last example. This follows from  an explicit (but again formal) calculation, which shows that the width variation fails to satisfy $P\div P=0$, that a grain boundary leads to a singularity in $\div P$ comparable to a locally finite measure, and that the target and U-turn patterns satisfy $\div P\in L^q$ for all $1\leq q<2$.

From both points of view---the behaviour of the functional along the sequence and the conditions on the limiting line field---only the smooth variation is admissible. However, since the target and U-turn patterns only just fail the two tests, it would be interesting to explore different rescalings of the functionals $\F_\es$ in order to allow for limit patterns of this type. The main impediment for doing so can be recognized in the discussion following the statement of Proprosition~\ref{estimatecomp}: if $\G_\e$ is unbounded as $\e\to0$, then the estimate~\pref{limites} no longer holds; therefore the proof of strong compactness no longer follows.

\medskip

Yet another way of phrasing the result of Theorem~\ref{thetheorem} is as follows: deviation from the optimal, straight-and-uniform stripe pattern carries an energy penalty. The combination of Theorems~\ref{th:ee} and~\ref{thetheorem} shows that the same is true for a mismatch in boundary behaviour: boundedness of $\G_\e$ forces the line field to be parallel to $\partial \Omega$, resulting in the fairly rigid situation that the limit solution set is empty for any other domain than a tubular neighbourhood. 

\medskip
A corollary of Theorem~\ref{thetheorem} is the fact that both stripes and energy density become evenly distributed in the limit $\e\to0$. This is reminiscent of the uniform energy distribution result of a related functional in~\cite{AlbertiChoksiOtto09}. Note that Theorem~\ref{thetheorem} goes much further, by providing a strong characterization of the \emph{geometry} of the structure.

\medskip
One result that we do not prove is a statement that for any \emph{fixed} $\e>0$ global minimizers themselves are stripe-like, or even tubular. At the moment it is not even clear whether such a statement is true. This is related to the condition~\pref{cond:resonance}, which expresses the requirement that an integer number of optimal-width layers fit exactly into~$\Omega$. 

The role of condition~\pref{cond:resonance} is most simply described by taking $\Omega$ to be a square, two-dimensional flat torus of size $L$. If $L$ is an integer multiple of $2\e$, then there exist structures---parallel, straight stripes---with zero energy $\G_\e$. This can be recognized in~\pref{ineq:basic}, where all terms on the right-hand side vanish. If $L$ is such that no straight-stripe patterns with optimal width exist, however, then $\G_\e$ is necessarily positive. In this case we can not exclude that a wavy-stripe structure (reminiscent of the wriggled stripes of of~\cite{RenWei05}) has lower energy, since by slightly modulating the stripes the average width (given by $M_\e$ in~\pref{ineq:basic}) may be closer to $\e$, at the expense of introducing a curvature term $\int |\theta'_\e|^2$.

\medskip

The introduction of projections, or line fields, for the representation of stripe patterns seems to be novel, even though they are commonly used in the modelling of liquid crystals (going back to De Gennes~\cite{deGennes71}). Ercolani \emph{et al.}~\cite{ErcolaniIndikNewellPassot03}, for instance, discuss the sign mismatch that happens at a U-turn pattern, and approach this mismatch by replacing the domain by a two-leaf Riemann surface. Using line fields appears to have the advantage of avoiding such mathematical contraptions, and staying closer to the physical reality.

\subsection{Plan of the paper}

In Section \ref{definitions} we recall the basic definitions and properties concerning Mass Transport, and we introduce line fields and measure-function pairs with the related notions of convergence. In Section \ref{sec:lowerbound} we prove that sequences with bounded energy~$\G_\es$ are relatively compact with respect to the weak convergence for measure-function pairs and we prove the \textit{liminf} inequality of $\G_\es$ with respect to weak convergence (Theorem \ref{thetheorem}, part~\ref{thetheorem:part2}). The main tool is the estimate in Proposition~\ref{prop:inequality}, obtained in \cite{PeletierRoegerTA}. In Section  \ref{sec:strongconvergence} we prove compactness with respect to the strong convergence for measure-function pairs (Theorem \ref{thetheorem}, part~\ref{thetheorem:part1}). In Section \ref{sec:limsup} we construct explicitly a recovery sequence satisfying the \textit{limsup} inequality for $\G_\es$ (Theorem \ref{thetheorem}, part~\ref{thetheorem:part3}), by using  the characterization of $\mathcal{K}_0$ obtained in \cite{PeletierVeneroniTA}.  

\subsection{Summary of notation}
\begin{small}
\begin{longtable}{lll}
  $\F_\e(\cdot)$         & energy functional & \eqref{eq:functional}\\
  $\G_\e(\cdot)$         & rescaled functional & \eqref{def:rescaled_functional}\\
  $K$         & domain of $\F_\e$, $\G_\e$ & \eqref{defk}\\
  $\G_0(\cdot,\cdot)$		& limit functional &\pref{def:G0}\\
  $X$						& space of limit pairs $(\mu,P)$  & \pref{def:X}\\
  $\mathcal K_0(\Omega)$		& domain of $\G_0$  &Def.~\ref{limprob}\\
%  $\W(\gamma)$     & bending energy of a curve $\gamma$ & \eqref{eq:def-elastica}\\
%  $\W(\Gamma)$     & generalized bending energy \\
%                   & of a system of curves $\Gamma$ & \eqref{eq:def-W-Gamma}\\
  $d(\cdot,\cdot)$             & Monge-Kantorovich distance &Def.~\ref{def:d1}\\
%  $d_p(\cdot,\cdot)$             & $p$-Wasserstein distance & \eqref{eq:def-d-p}\\
%  $\spt(\Gamma)$   & support of a system of curves $\Gamma$   & Def.~\ref{def:sys-curves}\\
%  $\theta(\Gamma,\cdot)$
%                    & \multiplicity of a system of
%  curves $\Gamma$   & Def.~\ref{def:sys-curves}\\
%  $|\Gamma|$
%                    & total mass of a system of
%  curves $\Gamma$   & Def.~\ref{def:sys-curves}\\
  $e^\perp$          & $90^\circ$ counter-clockwise rotation of the vector $e$\\
  $X$				& space of measure-projection pairs $(\mu,P)$ & \pref{def:X}\\
  $RM(\Omega)$		& space of Radon measures on $\Omega$\\
  $\L^n$				& $n$-dimensional Lebesgue measure\\
  $\Lip_1(\R^2)$        & set of Lipschitz continuous functions\\
                     & with Lipschitz constant at most 1\\
  $\mathcal{T},\mathcal{E}$
                    & transport set and set of endpoints of rays &
  Def.~\ref{def:rays}\\
  $[\mu,P]$			& graph measures	& Def.~\ref{def:graphm}\\
  $\hf$				& one-dimensional Hausdorff measure\\
  $\partial^* A$		& essential boundary of the set $A$ & \cite[Chapter~3.5]{afp}\\
  $E$               & $\{s: \gamma(s) \text{ lies inside a transport ray}\}$ &
  Def.~\ref{def:rayquantities}\\
  $\theta(s)$          & ray direction in $\gamma(s)$ & Def.~\ref{def:rayquantities}\\
  $\ell^+(s),\ell^-(s),l^+(s)$     & positive, negative and effective\\
%  $l^+(s)$     &  effective\\
                             & ray length in $\gamma(s)$&
  Def.~\ref{def:rayquantities}\\
%  $\psi(\cdot,\cdot)$            & parametrization & Def.~\ref{def:para}\\
%  $E_\delta$        & boundary points with uniformly bounded\\
%                    & ray lengths
%                    & \eqref{eq:def-E-delta}\\ 
  $\alpha(s),\beta(s)$    & direction of ray and \\
                          & difference to tangent at $\gamma(s)$& Def.~\ref{def:alphabeta}\\
  $\mp(s,\cdot)$             & mass coordinates & Def.~\ref{defi:coordm}\\
  $\tra(s,\cdot)$             & length coordinates & \eqref{def:tra}\\
  $M(s)$               & mass over $\gamma(s)$& Def.~\ref{defi:coordm}\\
  \\
  $E_i,\theta_i$, & corresponding quantities for a \\
                    
  $\ell^+_i,\ell^-_i,l_i^+$&collection
    $\{\gamma_i\}$ & Rem.~\ref{rem:dropping-indices}\\

    $\alpha_i,\beta_i,\mp_i,\tra_i,M_i$
\\
\\
  $E_{\e,i},\theta_{\e,i},$& corresponding quantities for  a\\
  $\ell^+_{\e,i},\ell^-_{\e,i},l_{\e,i}^+$ & collection $\{\gamma_{\e,i}\}$ & Rem.~\ref{rem:dropping-indices}\\
  $\alpha_{\e,i},\beta_{\e,i},\mp_{\e,i},\tra_{\e,i},M_{\e,i}$
\end{longtable}
\end{small}
%\end{minipage}

\textbf{Acknowledgement.} The authors gratefully acknowlegde many insightful and pleasant discussions with dr. Yves van Gennip and dr. Matthias R\"oger.

%%%%%%%%%%%%%%%%%%%%%%%%%%%%%%%%%%%%%%%%%%%%%%%%%%%%%%%%%%%%%%%%%%%%%%%%%%%%%%%%%%%%%%%%%%%%%%%%%%%%%%%%%
%%%%%%%%%%%%%%%%%%%%%%%%%%%%%%%%%%%%   SECTION 2   %%%%%%%%%%%%%%%%%%%%%%%%%%%%%%%%%%%%%%%%%%%%%%%%%%%%%%
%%%%%%%%%%%%%%%%%%%%%%%%%%%%%%%%%%%%%%%%%%%%%%%%%%%%%%%%%%%%%%%%%%%%%%%%%%%%%%%%%%%%%%%%%%%%%%%%%%%%%%%%%

\section{Preliminaries and preparation}\label{definitions}
\setcounter{theo}{0}
\setcounter{equation}{0}

\subsection{The Mass Transport Problem}
\label{sec:mtp}
In this section we introduce some basic definitions and concepts and we mention some results that we use later.
\begin{defi}
\label{def:d1}
	Let $u,v \in L^1(\o)$ %have compact support in $\o$ and 
	satisfy the mass balance 
	\begin{equation}\label{eq:massbalance}
		\int_{\o} u(x)\,dx = \int_{\o} v(x)\,dx.
	\end{equation}
	The Monge-Kantorovich distance $d_1(u,v)$ is defined as 
	\begin{eqnarray}\label{eq:def-d-1}
  d_1(u,v) &:=& \min \int_{\o\times\o} |x-y| \,d\gamma(x,y)
	\end{eqnarray}
	where the minimum is taken over all Radon measures $\gamma$ on $\o\times\o$
	with \emph{marginals}
	$u\L^2$ and $v\L^2$, i.e. such that 
	\begin{eqnarray}
  	\int_{\o\times\o} \varphi(x) \,d\gamma(x,y) &=& \int_\o\varphi u\,d\L^2,\label{eq:marg-u}\\
  	\int_{\o\times\o} \psi(y) \,d\gamma(x,y) &=& \int_\o\psi v\,d\L^2\label{eq:marg-v}
	\end{eqnarray}
	for all $\varphi,\psi\in C_c(\o)$.
\end{defi}

There is a vast literature on the optimal mass transportation problem
and an impressive number of applications, see for example \cite{EGan,TW,CFC,Amb,Vil,JKF,Ott}.
We only list a few results which we will use later.
\begin{theo}[\cite{CFC,FeC}]
Let $u,v$ be given as in Definition \ref{def:d1}.
\begin{enumerate}
\item There exists an optimal transport plan $\gamma$ in~\pref{eq:def-d-1}.
\item The optimal plan $\gamma$ can be parametrized in terms of a Borel measurable \emph{optimal transport map} $S:\o\to\o$, in the following way: for every $\zeta\in C_c(\o\times\o)$
\[
\int_{\o\times\o} \zeta(x,y)\, d\gamma(x,y) = \int_\o \zeta(x,S(x))u(x)\, dx,
\]
or equivalently, $\gamma = (id\times S)_\#u\L^2$. In terms of $S$, 
\[
d_1(u,v) = \int _{\o} |S(x)-x| u(x)\, dx.
\]
\item We have the dual formulation
\begin{equation}
\label{eq:duality}
d_1(u,v) = \sup \left\{\int _\Omega \phi(x)(u-v)(x)dx: \phi\in \Lip_1(\o)\right\},
\end{equation}
where $\Lip_1(\o)$ denotes the set of Lipschitz functions on $\o$
with Lipschitz constant not larger than $1$.
\item There exists an \emph{optimal Kantorovich potential} $\phi\in \Lip_1(\o)$ which achieves optimality in~\pref{eq:duality}.
\item Every optimal transport map $S$ and every optimal Kantorovich potential $\phi$
satisfy
\begin{eqnarray}\label{eq:dual-crit}
\phi(x) -\phi(S(x)) &=& |x-S(x)|\quad\text{ for almost all }x\in\spt(u).
\end{eqnarray}
\end{enumerate}
\end{theo}
The optimal transport map and the optimal Kantorovich potential are in general not
unique. We can choose $S$ and $\phi$ enjoying some additional properties.
\begin{lemma}[\cite{CFC,FeC}]\label{lemma:add-prop}
	There exists an optimal transport map $S\in \mathcal{A}(u,v)$ and an optimal
	Kantorovich potential $\phi$ such that
		\begin{eqnarray}
  		\phi(x) &=& \min_{y\in\spt(v)}\big(\phi(y)+|x-y|\big)\quad\text{ for
  		any }x\in\spt(u),\label{eq:add-prop-phi1}\\ 
  		\phi(y) &=& \max_{x\in\spt(u)}\big(\phi(x)-|x-y|\big)\quad\text{ for
  		any }y\in\spt(v),\label{eq:add-prop-phi2}
	\end{eqnarray}
	and such that $S$ is the unique monotone transport map in the sense of
	\cite{FeC},
	\begin{gather*}
  	\frac{x_1-x_2}{|x_1-x_2|}
  	+\frac{S(x_1)-S(x_2)}{|S(x_1)-S(x_2)|}\,\neq\, 0\ \text{ for all
  	}x_1\neq x_2\in\R^2\text{ with }S(x_1)\neq S(x_2).
	\end{gather*}
\end{lemma}
We will extensively use the fact that by~\eqref{eq:dual-crit} the optimal transport is organized along
{\itshape transport rays} which are defined as follows.
\begin{defi}\cite{CFC}
	\label{def:rays}
	Let $u,v$ be as in Definition \ref{def:d1} and let $\phi\in \Lip_1(\o)$
	be the optimal transport map as in Lemma \ref{lemma:add-prop}.
	A {\itshape transport ray} is a line segment in $\o$ with endpoints $a,b\in\o$ such that
	$\phi$ has unit slope on that segment and $a,b$ are maximal, that is
	\begin{gather*}
  		a\in\spt(u),\,b\in\spt(v),\quad a\neq b,\\
  		\phi(a)-\phi(b)\,=\,|a-b|\\
  		|\phi(a+t(a-b))-\phi(b)| \,<\, |a+t(a-b)-b|\quad\text{ for all }t>0,\\
  		|\phi(b+t(b-a))-\phi(a)| \,<\, |b+t(b-a)-a|\quad\text{ for all }t>0.
	\end{gather*}
	We define the {\itshape transport set }$\mathcal{T}$ to consist of all points which
	lie in the (relative) interior of some transport ray and $\mathcal{E}$ to be the
	set of all endpoints of rays. 
\end{defi}

Some important properties of
transport rays are given in the next proposition.
\begin{lemma}[\cite{CFC}]\label{prop:cfm-rays} Let $\mathcal{E}$ be as in Definition \ref{def:rays}.
	\begin{enumerate}
		\item Two rays can only intersect in a common endpoint.
		\item The endpoints $\mathcal{E}$ form a Borel set of Lebesgue measure zero.
		\item If $z$ lies in the interior of a ray with endpoints $a\in\spt(u),b\in\spt(v)$
		then $\phi$ is differentiable in $z$ with $\nabla\phi(z)\,=\,(a-b)/|a-b|$.
	\end{enumerate}
\end{lemma}
In Section \ref{sec:lowerbound} we will use the transport rays to parametrize the support of $u$ and
to compute the Monge-Kantorovich distance between $u$ and $v$.

%%%%%%%%%%%%%%%%%%%%%%%%%%%%%%%%%%%%%%%%%%%%%%%%%%%%%%%%%%%%%%%%%%%%%%%%%%%%%%%%%%%%%%%%%%%%%%%%%%%%%%%%%
%%%%%%%%%%%%%%%%%%%%%%%%%%%%%%%%%%%%  LINE FIELDS  %%%%%%%%%%%%%%%%%%%%%%%%%%%%%%%%%%%%%%%%%%%%%%%%%%%%%%
%%%%%%%%%%%%%%%%%%%%%%%%%%%%%%%%%%%%%%%%%%%%%%%%%%%%%%%%%%%%%%%%%%%%%%%%%%%%%%%%%%%%%%%%%%%%%%%%%%%%%%%%%

\subsection{Line fields}\label{sec:linefields}
As explained in the introduction, we will capture the directionality of an admissible function $u\in K$ in terms of a \emph{projection} on the boundary $\partial \supp u$. 
By the structure theorem on functions of bounded variation (e.g.~\cite[Section~5.1]{EvansGariepy92}), $|\nabla u|$ is a Radon measure on $\Omega$, and $\supp |\nabla u|$ coincides with the essential boundary $\Gamma :=\partial^* \supp u$ of $\supp u$ up to a $\hf$-negligible set. (Recall that the essential boundary is the set of points with Lebesgue density strictly between 0 and 1; $\hf$ is one-dimensional Hausdorff measure). There exists a $|\nabla u|$-measurable function $\nu:\R^2\to S^1$ such that the vector-valued measure $\nabla u$ satisfies $\nabla u = \nu|\nabla u|$, at $|\nabla u|$-almost every $x\in\Omega$. We then set
\[
P(x) := \nu(x)^\perp \otimes \nu(x)^\perp
\qquad \text{for }|\nabla u|\text{-a.e.~$x$}.
\]
In this way, we define a \textit{line field} $P(x)\in\R^{2\times 2}$ for $\hf$-a.e. $x\in \Gamma$ (or, equivalently, for $|\nabla u|$-a.e. $x\in \Omega$).

Note that since $\nu$ is $|\nabla u|$-measurable, and $P$ is a continuous function of $\nu$, $P$ is also $|\nabla u|$-measurable.
As a projection it is uniformly bounded, and  therefore
\begin{equation}
\label{bound:PinLinfty}
P\in L^\infty(\Gamma,\hf;\mathbb{R}^{2\times 2}).
\end{equation}		
Moreover, by construction, for $\hf$-a.e. $x\in \Gamma$, $P$ satisfies
		\begin{subequations}
		\label{pj:properties}
		\begin{eqnarray}		
				&P^2(x)=P(x),\\
				&|P(x)|^2=\ds \sum_{i,j}\big|P_{ij}(x)\big|^2=1,\\
				&\mbox{rank}\big(P(x)\big)=1,\\
				&P(x) \mbox{ is symmetric}.
		\end{eqnarray}
		\end{subequations}		
		
\subsection{Measure-function pairs}\label{sec:mfpairs}
As we consider 
%\marginpar{Changed sub $\e$ to sub $n$}
a sequence $\{u_n\}\subset K$, the set $\Gamma_n:=\partial^*\supp u_n$ depends on $n$, and therefore the line fields $P_n$ are defined on different sets. For this reason we use the concept of \emph{measure-function pairs}~\cite{Hutchinson86,Moser01,ags}. 
Given a sequence $\{ u_n\}\subset K$ we consider the pair $(\mu_n,P_n)$, where 
\begin{align*}
		\mu_n:= \es_n|\nabla u_n|\in RM(\R^2)\quad & \mbox{are Radon measures supported on $\Gamma_n$},\\
		P_n \in L^{\infty}(\mu_n;\R^{2 \times 2})\quad &\mbox{are the line fields tangent to $\Gamma_n$}.
\end{align*}
We introduce two notions of convergence for these measure-function pairs. Below $n\in\N$ is a natural number, not necessarily related to the dimension of $\R^2$. 
%We formulate these concepts for convergence on $\Omega$; similar definitions hold on $\R^2$.\marginpar{How to handle $\Omega$ vs $\R^2$?}

\begin{defi}\label{def:weak-coup-conv}\textit{(Weak convergence)}. 
Fix $p\in[0,\infty)$. Let $\{\mu_n\}\subset RM(\R^2)$ converge weakly-\star\  to $\mu\in RM(\R^2)$, let $v_n \in L^p(\mu_n;\R^n)$, and let $v\in L^p(\mu;\R^n)$. We say that a pair of functions $(\mu_n,v_n)$  converges weakly in $L^p$ to $(\mu,v)$, and write $(\mu_n,v_n)\weakto (\mu,v)$, whenever
\begin{itemize} 
	 	\item[i)] $\ds \sup_n \int_{\R^2} |v_n(x)|^p \,d\mu_n(x) < +\infty,$
	 	\item[ii)] $\ds \lim_{n \ra \infty} \int_{\R^2} v_n(x)\cdot\eta(x)\,d\mu_n(x) = 
	 		 			\int_{\R^2} v(x)\cdot\eta(x)\,d\mu(x),\quad \forall\,\eta\in C^0_c(\R^2;\R^n).$
\end{itemize}
\end{defi}

\begin{rem}
\label{rem:weak-compactness}
There is a form of weak compactness: any sequence satisfying condition i) above, and for which $\mu_n$ is tight, has a subsequence that converges weakly~\cite{Hutchinson86}.
\end{rem}

\begin{defi}\label{def:strongconv}\textit{(Strong convergence)}. 
Under the same conditions, we say that $(\mu_n,v_n)$ converges strongly in $L^p$ to $(\mu,v)$, and write $(\mu_n,v_n)\to (\mu,v)$,  if 
\begin{itemize} 
	 	\item[i)] $(\mu_n,v_n) \rightharpoonup (\mu,v)$ in the sense of Definition \ref{def:weak-coup-conv},
	 	\item[ii)] $\ds \lim_{n\ra \infty} \int_{\R^2} |v_n(x)|^p\,d\mu_n(x) = 
	 			\int_{\R^2} |v(x)|^p\,d\mu(x).$
\end{itemize}
\end{defi}

\begin{remark}\label{rem:comparison} 
It may be useful to compare the last definition with the definition, introduced by Hutchinson in \cite{Hutchinson86}, of weak-$*$ convergence of the associated graph measures.

In the following let $\{(\mu_n,P_n)\}$, $(\mu,P)$ be measure-function pairs over $\R^2$ with values in $\R^n$, such that $\mu_n \ras \mu$.%, and let $F:\o \times \R^n\ra \R$ be a continuous and nonnegative function, such that $P \mapsto F(x,P)$ is convex and has superlinear growth.

%\begin{defi}\label{def:fstrong}
%\textit{($F$-strong convergence, \cite{Hutchinson86})} We say that $(\mu_n,P_n)$ converges to $(\mu,P)$ in the $F$-strong sense %, and write $(\mu_n,P_n) \stackrel{F}{\ra}(\mu,P)$ 
%if the following hold:
%\begin{itemize}
%	\item[i)] $ \ds \int_\o F(x,P_n(x))\, d\mu_n(x) < +\infty,\quad \forall\, n\in \N $;
%	\item[ii)] $\ds \lim_{j \ra \infty} \int_{E_{nj}} F(x,P_n(x))\, d\mu_n(x) = 0,$ uniformly in $n$, where
%	$$ E_{nj}= \{ x\in \o : |x|>j\ \mbox{or}\ |P_n(x)|>j\};$$
%	\item[iii)] $ \ds \lim_{n\ra \infty}\int_\o \varphi(x,P_n(x))\, d\mu_n(x)= \int_\o \varphi(x,P(x))\, d\mu(x)$, for all $\varphi \in C^\infty_c(\o \times \R^n)$.
%\end{itemize}
%\end{defi}
\begin{defi}\cite{Hutchinson86}\label{def:graphm}
The \textit{graph measure}  associated with the measure-function pair $(\mu,P)$ is defined by 
$$ [\mu,P]:= (id \times P)_\# \mu \in RM(\R^2 \times \R^{2\times2}),  $$
and the related notion of convergence is the weak-$*$ convergence in $RM(\R^2 \times \R^{2\times2})$.
\end{defi}
%We can summarize the relationships between these notions in three points:
%
%1. For all $\varphi\in C^0(\o \times \R^n)$,
%$$\quad \int_{\o \times \R^n} \varphi(y,Q)\,d[\mu,P](y,Q) = \int_\o \varphi (x,P(x))\,d\mu(x),$$
%and therefore, if conditions $i)$ and $ii)$ in Def. \ref{def:fstrong} are satisfied, then $F$-strong convergence and convergence of the graph measures are equivalent (\cite[Prop. 4.4.1, (ii)]{Hutchinson86}). 
%
%2. By \cite[Theorem 5.4.4, (iii)]{ags}, convergence in Def. \ref{def:strongconv} implies 
%$$ \lim_{n\ra \infty}\int_\o \varphi(x,P_n(x))\, d\mu_n(x)= \int_\o \varphi(x,P(x))\, d\mu(x),$$ 
%for all $\varphi \in C^0(\o \times \R^n)$ such that $|\varphi(x,Q)|\leq a(x) +b(x)|Q|^p$ for some $p>1$, $a,b\in C^0(\o)$. 
%
%3. For every function $F$ such that $P\mapsto F(x,P)$ is strictly convex, by \cite[Theorem 4.4.2, (iii)]{Hutchinson86}, $F$-strong convergence implies 
%$$ \lim_{n\ra \infty}\int_\o F(x,P_n(x))\, d\mu_n(x)= \int_\o F(x,P(x))\, d\mu(x).$$

%\marginpar{Why do we need $F$-strong convergence?}
Let $\{u_n\}\subset K$ and let $\{(\mu_n,P_n)\}$ be the associated measure-function pairs, as in Subsections \ref{sec:linefields} and \ref{sec:mfpairs}, so that $|P_n|\equiv 1$ and supp$(\mu_n)$ is contained in a compact subset of $\R^2$. Assume that $\mu_n \ras \mu \in RM(\R^2)$. Then, by \cite[Th. 5.4.4, (iii)]{ags} and \cite[Prop. 4.4.1-(ii) and Th. 4.4.2-(iii)]{Hutchinson86}, `strong' convergence in the sense of Definition \ref{def:strongconv} and convergence of the graphs are equivalent. Under these assumptions these concepts are also equivalent to $F$-strong convergence \cite[Def. 4.2.2]{Hutchinson86} in the case $F(x,P):=|P|^2$.  
%
%
%conditions $i)$ and $ii)$ in Definition \ref{def:fstrong} are satisfied and the above implications directly lead to the conclusion that, under the assumptions we made, convergence in Definition \ref{def:strongconv}, $F$-strong convergence, and convergence of the graphs are equivalent. As a direct consequence we deduce that the strong measure-function-pair convergence of Definition~\ref{def:strongconv} is metrizable.
\end{remark}

We conclude with a result for weak-strong convergence for measure-function pairs which shows a similar behaviour as in $L^p$ spaces:
\begin{theo}[\cite{Moser01}]\label{theo:weakstrong}
	Let $\mu_n$,$\mu\in RM(\R^2)$, let $P_n,H_n\in L^2(\mu_n)$ and $P,H\in L^2(\mu)$. Suppose that 
	$$ (\mu_n,P_n) \ra (\mu,P)\quad \mbox{strongly in the sense of Definition \ref{def:strongconv}}$$
	and
	$$ (\mu_n, H_n) \weakto (\mu,H)\quad \mbox{weakly in the sense of Definition \ref{def:weak-coup-conv}}.$$
	Then, for the product $P_n\cdot H_n \in L^1(\mu_n)$ we have
	$$ (\mu_n,P_n\cdot H_n) \weakto (\mu,P\cdot H)\quad \mbox{weakly in the sense of Definition \ref{def:weak-coup-conv}}.$$
\end{theo}

\bigskip

%%%%%%%%%%%%%%%%%%%%%%%%%%%%%%%%%%%%%%%%%%%%%%%%%%%%%%%%%%%%%%%%%%%%%%%%%%%%%%%%%%%%%%%%%%%%%%%%%%%%%%%%%
%%%%%%%%%%%%%%%%%%%%%%%%%%%%%%%%%%%%%%   SECTION 3  %%%%%%%%%%%%%%%%%%%%%%%%%%%%%%%%%%%%%%%%%%%%%%%%%%%%%
%%%%%%%%%%%%%%%%%%%%%%%%%%%%%%%%%%%%%%%%%%%%%%%%%%%%%%%%%%%%%%%%%%%%%%%%%%%%%%%%%%%%%%%%%%%%%%%%%%%%%%%%%

\section{Proofs of weak compactness and lower bound}\label{sec:lowerbound}

Although the statement of Theorem~\ref{thetheorem} refers explicitly to sequences $\e_n\to0$, we shall alleviate notation in the rest of the paper and consistently write $\e$ instead of $\e_n$, and $u_\e$, $\mu_\e$, and $P_\e$, instead of their counterparts $u_n$, $\mu_n$, and $P_n$; and when possible, we will even drop the index $\e$.

\subsection{Overview}
In this section, Section~\ref{sec:lowerbound}, we show that if $\G_\e(u_\e)$ is bounded independently of $\e$, then we can choose a subsequence along which the function $u_\e$ and the measure-projection pairs $(\mu_\e,P_\e)$ converge weakly. Recall that this pair is defined by (see Section~\ref{sec:linefields})
\[
\mu_\e:= \e|\nabla \ue|
\qquad\text{and}\qquad
P_\e  = \frac{\nabla u_\e^\perp}{|\nabla u_\e|}\otimes \frac{\nabla u_\e^\perp}{|\nabla u_\e|}.
\]
A corollary of this convergence is the lower bound~\pref{ineq:lowerboundintro}. 
The results of this Section~\ref{sec:lowerbound} thus provide the first half of part~\ref{thetheorem:part1} and the whole of part~\ref{thetheorem:part2} of Theorem \ref{thetheorem}.

\medskip
The argument starts by using the parametrization by rays that was mentioned in the introduction to bound certain geometric quantities in terms of the energy $\G_\e(u_\e)$ (Proposition~\ref{prop:inequality}). Using this inequality we then prove that (Lemma~\ref{lemmaconvergence})
\[
\ue\stackrel\ast\rightharpoonup \frac12 \quad\text{in }L^\infty(\o) 
\qquad\text{and}\qquad 
\mue := \e|\nabla \ue| \stackrel*\rightharpoonup \mu:= \frac12 \L^2\llcorner\Omega\quad \mbox{in }RM(\R^2).
\]
This result should be seen as a form of equidistribution: both the stripes and the interfaces separating the stripes become uniformly spaced in $\Omega$. 

From the $L^\infty$-boundedness of $P_\e$ it follows (Lemma~\ref{lemmaconvergence}) that along a subsequence
\[
		(\mu_\e,P_\e) \rightharpoonup (\mu,P) \text{ in }L^p,\qquad\text{for all $1\leq p<\infty$},
\]
and therefore 
$\div(P_\e\mu_\e)$ converges in the sense of distributions on $\R^2$. In Lemmas~\ref{bigestimate} and~\ref{lemma:divint} we use the estimate of Proposition~\ref{prop:inequality} to show that the limit of $\div(P_\e\mu_\e)$ equals a function $-H\in L^2(\R^2;\R^2)$ supported on $\Omega$, i.e. that
\[
\lim_{\es \ra 0}\int_{\R^2} P_\es(x): \nabla\eta(x)\,d\mue(x) = \frac12 \int_\o H(x)\cdot \eta(x)\,dx,\h 
		\forall\,\eta \in C^0_c(\R^2;\R^2).
\]
From this weak convergence we then deduce in Lemma~\ref{lemma:divint} the lower bound
\[
\liminf_{\es \ra 0}\, \G_\es(\ue) \geq \frac 18 \int_\o \big|\div P(x)\big|^2 dx.
\]
For the proof of part~\ref{thetheorem:part1} of Theorem \ref{thetheorem} it remains to prove that $(\mu_\e,P_\e)$ converges strongly; this is done in Section~\ref{sec:strongconvergence}. 

\medskip

\subsection{Regularization of the interfaces}
Before we set out we first show that we can restrict ourselves to a class of more regular functions. 
\begin{lemma}\label{approxi}
It is sufficient to prove parts~\ref{thetheorem:part1} and~\ref{thetheorem:part2} of Theorem \ref{thetheorem} under the additional assumption that $\gme$ is parametrizable as a finite %and disjoint 
family of %closed, 
simple, smooth curves 
$$\gamma_{\es,j} : [0,L_{\es,j}]\ra \o,\quad j=1,\ldots,J_\es,  $$
for some $J_\es \in \N$, with $L_{\es,j}\leq 1$ for all $\es,j$ and
$$ \gamma_{\es,j}\bigl((0,L_{\es,j})\bigr)\cap \gamma_{\es,i}\bigl((0,L_{\es,i})\bigr)=\emptyset\quad \mbox{if }\ \ i \neq j.$$ 
Moreover, there exists a permutation $\sigma_\e$ on the numbers $\{1,\dots,J_\e\}$ such that for all $j=1,\dots,J_\e$,
\begin{equation}\label{concatenation}
	\gamma_{\e,j}(L_{\es,j})= \gamma_{\e,\sigma_\e(j)}(0)\quad \mbox{and}\quad \gamma_{\e,j}'(L_{\es,j}^-)= \gamma_{\e,\sigma_\e(j)}'(0^+).
\end{equation}

\end{lemma}

\begin{proof} 
Let $u\in K$ be fixed for the moment. Since $\supp u$ has finite perimeter, by standard approximation results (see \cite[Sect. 5.2]{EvansGariepy92} or \cite[Theorem 3.42]{afp}) there exists a sequence $\{E_k\}$ of open subsets of~$\Omega$ with smooth boundary such that the characteristic functions $u_k:=\chi_{E_k}$ satisfy
\begin{equation}\label{arr:approx}
\begin{array}{rll}
		i)	&	u_k \ra u								& \mbox{strongly in }L^1(\o),\\
		ii) & \nabla u_k \ras \nabla u & \mbox{weakly-$*$ in }RM(\o),\\
		iii)& |\nabla u_k|(\o) \ra |\nabla u|(\o) .
\end{array}
\end{equation}
%there exists a sequence $\{\hat{u}_k\}_{k\in \N}$ of characteristic functions on $\R^2$ that approximate $\ue$ in $L^1(\R^2)$ and in perimeter~\cite[Theorem 3.42]{afp}. 
By a small dilation we can furthermore adjust the total mass so that $\int_{\R^2}u_k = |\o|/2$. By the $L^1$-continuity of the metric $d$ we obtain that for fixed $\es>0$
$$ \F_\es(u_k) \ra \F_\es(u)\quad \mbox{as }k\ra \infty.$$

Along this sequence, the corresponding measure-function pair $(\mu_k,P_k)$ converges strongly. Indeed, writing $\nu_k = d\nabla u_k/d|\nabla u_k|$ and $\nu = d\nabla u/d|\nabla u|$, the Reshetnyak Continuity Theorem (see \cite[Th. 2.39]{afp}) implies that
\begin{equation}\label{resh}
		\lim_{k \ra \infty} \int_\o f(x,\nu_k(x))d|\nabla u_k|(x) = \int_\o f(x,\nu(x))d|\nabla u|(x),
\end{equation}
for every continuous and bounded function $f :\o \times S^1 \ra \R$. Therefore, since $P = \nu^\perp\otimes\nu^\perp$ and $P_k=\nu_k^\perp\otimes\nu_k^\perp$, 
\begin{equation}\label{eq:weakapprox}
		\lim_{k \ra \infty}  \int_\o \vfi(x):P_k(x)\,d\mu_k(x) = \int_\o \vfi(x):P(x)\,d\mu(x),
\end{equation}
for every $\vfi\in C^0(\o;\R^{2 \times 2})$, and 
\begin{equation}\label{eq:strongapprox}
	\lim_{k \ra \infty}  \int_\o |P_k(x)|^2\,d\mu_k(x) = \int_\o |P(x)|^2\,d\mu(x).
\end{equation}

%This shows that in the proof of the \textit{liminf} inequality (part~\ref{thetheorem:part2} of Theorem \ref{thetheorem}) we can replace the sequence $\{\ue\}$ by a sequence which satisfies the assertion of Lemma \ref{approxi}. \marginpar{Still need to show measure-function convergence of $(\mu_k,P_k)$.}

We turn now to part~\ref{thetheorem:part1} of Theorem \ref{thetheorem}. 
%We mentioned in Section~\ref{sec:linefields} that $u_k$ admits the polar decomposition $\nabla u_k=\nu_k |\nabla u_k|$, where $\nu_k(x)\in S^1$ for $|\nabla u_k|$-a.e. $x\in \o$. Let $\mu_k:=|\nabla u_k|$ and $P_k:=\nu_k^\perp\otimes \nu_k^\perp$, then the sequence of measure-function pairs $\{(\mu_k,P_k)\}$ approximates $(\mu,P)$ with respect to the strong convergence for measure-function pairs as in Definition \ref{def:strongconv}. Indeed, by the Reshetnyak Continuity Theorem (see \cite[Th. 2.39]{afp}) 
%\begin{equation}\label{resh}
%		\lim_{h \ra \infty} \int_\o f(x,\nu_k(x))d|\nabla u_k|(x) = \int_\o f(x,\nu(x))d|\nabla u|(x),
%\end{equation}
%for every continuous and bounded function $f :\o \times S^1 \ra \R$, where $\nu = d\nabla u/d|\nabla u|$. By a direct application of (\ref{resh}) we obtain that
%\begin{equation}\label{eq:weakapprox}
%		\lim_{h \ra \infty}  \int_\o \vfi(x):P_k(x)\,d\mu_k(x) = \int_\o \vfi(x):P(x)\,d\mu(x)
%\end{equation}
%for every $\vfi\in C^0(\o;\R^{2 \times 2})$, and 
%\begin{equation}\label{eq:strongapprox}
%	\lim_{h \ra \infty}  \int_\o |P_k(x)|^2\,d\mu_k(x) = \int_\o |P(x)|^2\,d\mu(x).
%\end{equation}
Let us assume that Theorem \ref{thetheorem}.\ref{thetheorem:part1} holds under the additional assumption of Lemma \ref{approxi}. Let $\{u_n\}\in K$, and by Remark~\ref{rem:weak-compactness} we can assume that the related sequence of measure-projection pairs satisfies 
\begin{equation}\label{eq:nthweak}
	(\mu_n,P_n) \weakto (\mu,P) \mbox{ in the sense of Definition \ref{def:weak-coup-conv}}.
\end{equation}
We want to prove that, after extraction of a subsequence,
\begin{equation}\label{eq:nthstrong}
	(\mu_n,P_n) \ra (\mu,P) \mbox{ in the sense of Definition \ref{def:strongconv}}.
\end{equation}
Recall that the strong convergence of a sequence $\{(\mu_k,P_k)\}$ of measure-function pairs is equivalent to the weak-* convergence of the graph measures $[\mu_k,P_k]\in RM(\R^2 \times \R^{2\times 2})$ (see Remark \ref{rem:comparison} above and Section \ref{sec:strongtop} below), by (\ref{arr:approx}), (\ref{eq:weakapprox}), and \pref{eq:strongapprox}.

Let $d$ be a metric on $RM(\R^2 \times \R^{2\times 2})$, inducing the weak-* convergence on bounded sets  %(any $p$-Wasserstein metric, for example)
and such that
\begin{equation}\label{yip}
	\left| \int_{\R^2} \vfi(x):P(x)\,d\mu(x) -\int_{\R^2} \vfi(x):Q(x)\,d\nu(x) \right| \leq C\|\vfi\|_{C^1}d([\mu,P],[\nu,Q]),
\end{equation}
for all $\vfi \in C^1_c(\R^2;\R^{2 \times 2})$ (see e.g. \cite[Def. 2.1.3]{Yip98}). 
%As an example (see e.g. \cite[Def. 2.1.3]{Yip98}), let $U\subset \R^2$ be bounded, for every couple $[\mu,P],[\nu,Q]$ with supp$(\mu)$, supp$(\nu)\subset U$, we can define \marginpar{\footnotesize I wrote this just to check, but maybe it's better to leave it out - I need $U$ for the dashed integral}
%\begin{multline*}
%	\tilde{d}([\mu,P],[\nu,Q]):=\sup\left\{ \int_U \psi(x):P(x)\, d\mu - \int_U \psi(x):Q(x)\, d\nu,\right.\\
%	 \left.\psi \in \Lip_1(U;\R^{2 \times 2})\ \mbox{such that}\ \dashint_U \psi(x)\,dx = 0 \right\},
%\end{multline*}
%$$
%	d([\mu,P],[\nu,Q]):= \tilde{d}([\mu,P],[\nu,Q]) + \left| \int_U P(x)\,d\mu(x) - \int_U Q(x)\,d\nu(x)\right|.
%$$
By the arguments above, we can find a bounded set $U$ such that for every $n\in \N$ there exists an open set $\tilde{E}_n\subset \subset U$, with smooth boundary, such that the characteristic function $\tilde{u}_n:=\chi_{\tilde{E}_n}$ and the associated $\tilde \mu_n$ and $\tilde P_n$ satisfy
\begin{eqnarray}
			&\ds	 d([\tilde{\mu}_n,\tilde{P}_n],[\mu_n,P_n])<\frac1n,								&  \label{eq:dist1n}\\
		 	& \ds  \left|\int_{\R^2} |\tilde{P}_n(x)|^2\,d\tilde{\mu}_n(x) - \int_{\R^2} |P_n(x)|^2\,d\mu_n(x)\right| < \frac 1n. \label{eq:dist2n}& 
\end{eqnarray}

Owing to Theorem \ref{thetheorem}.\ref{thetheorem:part1} there exists a couple $(\tilde\mu,\tilde{P})$ and a subsequence, still denoted $\{\tilde{u}_n\}$, such that $(\tilde\mu_n,\tilde{P}_n) \ra (\tilde\mu,\tilde{P})$ strongly, in the sense of Definition \ref{def:strongconv}. On the other hand, $(\tilde\mu_n,\tilde{P}_n) \weakto (\mu,P)$, since
for any $\vfi\in C^0_c(\R^2;\R^{2\times 2})$,
\begin{multline*}
\left|\int_{\R^2} \vfi:\tilde P_n\,d\tilde \mu_n - \int_{\R^2} \vfi:P\,d\mu\right|
\leq \left| \int_{\R^2} \vfi:\tilde P_n \,d\tilde \mu_n - \int_{\R^2} \vfi :P_n \,d\mu_n \right|\\
+ \left|\int_{\R^2} \vfi :P_n \,d \mu_n - \int_{\R^2} \vfi : P \,d \mu\right|,
\end{multline*}
and the first converges to zero by~(\ref{eq:dist1n}), and the second by~(\ref{eq:nthweak}).
Therefore $(\tilde\mu,\tilde{P})=(\mu,P)$. In addition,
\begin{multline*}
	\left|\int_{\R^2} |P_n|^2\,d\mu_n - \int_{\R^2} |P|^2\,d\mu\right|\leq 
	\left|\int_{\R^2} |P_n|^2\,d\mu_n - \int_{\R^2} |\tilde{P}_n|^2\,d\tilde\mu_n\right| \\
	 +\left|\int_{\R^2} |\tilde{P}_n|^2\,d\tilde\mu_n- \int_{\R^2} |\tilde{P}|^2\,d\tilde\mu\right|. 
\end{multline*}	
Passing to the limit as $n\ra\infty$, by \pref{eq:dist1n} and~\pref{eq:dist2n} we obtain (\ref{eq:nthstrong}).

For part~\ref{thetheorem:part2} of Theorem~\ref{thetheorem} the argument is similar, but simpler, and we omit it. The existence of the permutation $\sigma_\e$ follows by cutting the smooth boundary of $\tilde E_n$ into sections of length no more than $1$.
\end{proof}

%%%%%%%%%%%%%%%%%%%%%%%%%%%%%%%%%%%%%%%%%%%%%%%%%%%%%%%%%%%%%%%%%%%%%%%%%%%%%%%%%%%
%%%%%%%%%%%%%%%%%%%%%%%%%%%% SUBSECTION rays and estimate %%%%%%%%%%%%%%%%%%%%%%%%% %%%%%%%%%%%%%%%%%%%%%%%%%%%%%%%%%%%%%%%%%%%%%%%%%%%%%%%%%%%%%%%%%%%%%%%%%%%%%%%%%%%

\subsection{Parametrization by rays, mass coordinates, and a fundamental estimate}
\label{sec:rays}

The central estimate (\ref{ineq:inequality}) below is derived in \cite{PeletierRoegerTA} in a very similar case. It follows from an explicit expression of the Monge-Kantorovich distance $d(u,1-u)$ obtained by a convenient parametrization of the domain in terms of the transport rays. Here we recall the basic definitions and we state the main result, Proposition~\ref{prop:inequality}, referring to \cite{PeletierRoegerTA} for further details and proofs. 
%\marginpar{Maybe eliminate all subscripts $\e$ from this section?}

Let $\phi\in\mbox{Lip}_1(\R^2)$ be an optimal Kantorovich potential for the mass transport from $u$ to $1-u$ as in Lemma \ref{lemma:add-prop}, with $\mathcal{T}$ being the set of transport rays as in Definition~\ref{def:rays}. Recall that $\phi$ is differentiable, with $|\nabla \phi|=1$, in the relative interior of any ray. We define several quantities that relate the structure of the support of $u$ to the optimal Kantorovich potential $\phi$. Finally we define a parametrization of $\o$.
\begin{defi}
\label{def:rayquantities}
		For $\g\in \Gamma$, defined on the set $[0,L]$, we define
		\begin{itemize}
			\item[1)] a set $E$ of interface points that lie in the relative interior of a ray,
				$$ E:=\{s\in [0,L] : \g(s)\in \mathcal{T}\},$$
			\item[2)] a direction field
				$$ \theta: E \ra \mathcal{S}^1,\h \theta(s):=\nabla\phi(\g(s)),$$	
				\item[3)]	the positive and negative total ray length $\ell^+,\ell^-:E \ra \R$,
						\begin{align}
							\ell^+(s) &:= \sup\,\{t>0 :\phi\bigl(\g(s)+t\theta(s)\bigr)-\phi(\g(s))=t\},\\
							\ell^-(s) &:= \inf\,\{t<0 :\phi\bigl(\g(s)+t\theta(s)\bigr)-\phi(\g(s))=t\},
						\end{align}		
				\item[4)] the effective positive ray length $l^+:E \ra \R$,			
						\begin{equation*}
%							l^+(s):=\sup\{t\in[0,L^+(s)]:\g(s)+\tau\theta(s)\in \mbox{Int(supp($\ue$)) for all }0<\tau<t\}.
							l^+(s):=\sup\{t\geq0:\g(s)+\tau\theta(s)\in \mbox{Int(supp($u$)) for all }0<\tau<t\}
						\end{equation*}
				(with the convention $l^+(s) = 0$ if the set above is empty).
		\end{itemize}		
\end{defi}
\begin{rem}
\label{rem:dropping-indices}
		All objects defined above are properties of $\g$ even if we do not denote this dependence explicitly. When dealing with a collection of curves  $\{\g_j:j=1,\ldots,J\}$ or $\{\g_{\es,j}:\es>0,j=1,\ldots,J_\e\}$, then $E_j$, $\theta_{\es,j}$ etc. refer to the objects defined for the corresponding curves. 
\end{rem}
\begin{defi}
\label{def:alphabeta}
		Define two functions $\alpha,\beta: E \ra (\R \mod 2\pi)$ by requiring that 
		\begin{equation*}
				\theta(s)=\left( \begin{array}{c}
														\cos \alpha(s)\\		
														\sin \alpha(s)
													\end{array}
									\right),\quad
				\det \left(\g'(s),\theta(s)\right) = \sin \beta(s).					
		\end{equation*}													
\end{defi}
In the following computations it will often be more convenient to employ \textit{mass coordinates} instead of  \textit{length coordinates}:
\begin{defi}\label{defi:coordm}
		For $\g\in\Gamma$ and $s\in E$ we define a map $\mp_s:E \ra \R$ and a map $M: E \ra \R$ by
\begin{align}
  \mp(s,t)	& := \left\{ \begin{array}{ll}
	  t \sin \beta(s) - \frac{t^2}{2} \alpha'(s) & \mbox{if }l^+(s)>0,\\
	  0	& \mbox{otherwise}.
    \end{array} \right.\\
  M(s)			& := \mp(s,l^+(s)).
\end{align}																						
\end{defi}
\begin{figure}[htbp]
  \begin{center}
    \psfrag{u0}[c][]{$u=0$}		
    \psfrag{u1}[c][]{$u=1$}
    \psfrag{g}[c][]{$\gamma$}
    \psfrag{gs}[c][]{$\gamma(s)$}
    \psfrag{mt}[c][]{$\mp(s,t)$}
    \psfrag{mt1}[c][]{$\mp(s,t_1)$}
    \psfrag{gpt}[c][]{$\g(s)+t\theta(s)$}
    \psfrag{gmt}[c][]{$\g(s)-t\theta(s)$}
    \psfrag{gmt2}[c][]{$\g(s)-t_2\theta(s)$}
    \psfrag{gpt1}[c][]{$\g(s)+t_1\theta(s)$}
		\includegraphics[scale=1]{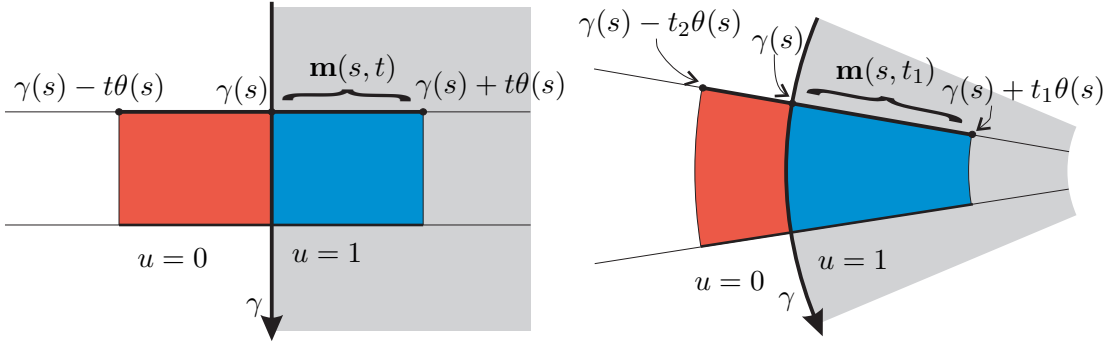}
		\caption{Mass coordinates. In the picture, a bending of the interface $\g$ produces a 
		stretching ($t\nearrow t_1$) of the transport ray in $\{u=1\}$ and a shrinking 
		($t\searrow t_2$) in $\{u=0\}$. $\mp(s,t_1)$ represents the amount of mass lying on the segment 
		stretching from $\g(s)$ to $\g(s)+t_1\theta(s)$, accounting for the change of density due to the stretching. The dimensions are exaggerated for clarity.}
		\end{center}
\end{figure}
Introducing the inverse of $\mp$ we can formulate a change of variables between length and mass coordinates:
\begin{prop}[\cite{PeletierRoegerTA}]\label{prop:coordt}
	The map $\mp(s,\cdot)$ is strictly monotonic on $(\ell^-(s),\ell^+(s))$ with inverse 
\begin{equation}
\label{def:tra}
\tra(s,m):=\frac{\sin \beta(s)}{\alpha'(s)}
	\left[ 1- \left( 1- \frac{2\alpha'(s)}{\sin^2 \beta(s)}m\right)^{\frac12}\right].
\end{equation}
%	Let $\Gamma_\es$ be given as in Lemma \ref{approxi}. 
%	We obtain for $g\in L^1(\o)$ that
%	\begin{align*}
%		\int g(x) u_\e(x) dx &= \int_0^{L_\es} 
%				\int_0^{M_\es(s)} g(\gamma_\es(s)+\tra_\es(m;s)\theta_\es(s))\, dm\,ds,\\
%		\int g(x) (1-u_\e(x)) dx &= \int_0^{L_\es} 
%				\int_{-M_\es(s)}^0 g(\gamma_\es(s)+\tra_\es(m;s)\theta_\es(s))\, dm\, ds.
%	\end{align*}
\end{prop}

Going back to the full set of curves $\Gamma = \{\gamma_{j}\}$ we have the following parameterization result:
\begin{prop}
Let $\Gamma$ be given as in Lemma \ref{approxi}. For any $g\in L^1(\Omega)$ we have
\begin{align}
%		\int g(x) u_\e(x) dx &= \sum_j \int_0^{L_{\es,j}} 
%				\int_0^{M_{\es,j}(s)} g(\gamma_{\es,j}(s)+\tra_{\es,j}(s,m)\theta_{\es,j}(s))\, dm\,ds,\\
%		\int g(x) (1-u_\e(x)) dx &= \sum_j \int_0^{L_{\es,j}} 
%				\int_{-M_{\es,j}(s)}^0 g(\gamma_{\es,j}(s)+\tra_{\es,j}(s,m)\theta_{\es,j}(s))\, dm\,ds.
\label{eq:change-of-variables1}
\int g(x) u_(x) dx &= \sum_j \int_0^{L_{j}} 
	\int_0^{M_{j}(s)} g(\gamma_{j}(s)+\tra_{j}(s,m)\theta_{j}(s))\, dm\,ds,\\
\int g(x) (1-u(x)) dx &= \sum_j \int_0^{L_{j}} 
	\int_{-M_{j}(s)}^0 g(\gamma_{j}(s)+\tra_{j}(s,m)\theta_{j}(s))\, dm\,ds.
	\notag
\end{align}
\end{prop}

%In particular, taking $g\equiv1$ we find that 
%\[
%\frac12|\Omega|  = \int_\Omega u_\e = \sum_i\int_0^{L_{\e,i}} M_{\e,i}(s)\, ds.
%\]
%Since we expect the stripe width $M_{\e,i}$ to be of order $\e$ (which is confirmed below), it follows that $\sum _i L_{\e,i}\sim \e^{-1}$. 
With this parametrization, the distance $d(u,1-u)$ takes a particularly simple form:
\[
%d(u_\e,1-u_\e) = \sum_j \int_0^{L_{\e,j}} \int_0^{M_{\e,j}(s)} \bigl[\tra_{\e,j}(s,m) - \tra_{\e,j}(s,m-M_{\e,j}(s))\bigr]\, dm \,\e ds.
d(u,1-u) = \sum_j \int_0^{L_{j}} \int_0^{M_{j}(s)} \bigl[\tra_{j}(s,m) - \tra_{j}(s,m-M_{j}(s))\bigr]\, dm \, ds.
\]
From the positivity property $m\,\tra_{j}(s,m)\geq 0$ we therefore find the estimate
\begin{equation}
\label{est:d1}
d(u,1-u) \geq \sum_j \int_0^{L_{j}} \int_0^{M_{j}(s)} \tra_{j}(s,m) \, dm \, ds.
\end{equation}

\medskip

Finally we can state the fundamental estimate:
\begin{prop}[\cite{PeletierRoegerTA}]\label{prop:inequality}
		Under the conditions provided by Lemma~\ref{approxi} we have the lower bound
\begin{multline}
  \G_\es(u) \geq \sum_{j=1}^{J_\es} \int_0^{L_j}\left[ \frac 1{\es^2}
  \left( \frac 1{\sin \beta_j(s)}-1\right) \left( \frac{M_j(s)}{\es} \right)^2 +
 \frac 1{\es^2}\left( \frac{M_j(s)}{\es} -1\right)^2 \right]\, \es\, ds +\\
{}+	\sum_{j=1}^{J_\es} \int_0^{L_j} \frac 1{4\sin \beta_j(s)}
			\left( \frac{M_j(s)}{\es\sin \beta_j(s)}\right)^4\alpha_j'(s)^2\,\es\,ds.
\label{ineq:inequality}
\end{multline}
\end{prop}
The inequality (\ref{ineq:inequality}) should be interpreted as follows. Along a sequence $\{\ue\}$ with bounded energy $\G_\es(\ue)$, the three terms of the right hand side tell us that:
\begin{enumerate}
\item $\sin \beta_\es \ra 1$, which implies that as $\es \ra 0 $ the transport rays tend to be
orthogonal to the curve $\g_\es$;
\item \label{enum:conseqinequality_ii}
$M_\es/\es \ra 1$, forcing the length of the transport rays, expressed in mass
coordinates, to be ${}\sim \es$;
\item $(\alpha_\es')^2$ is bounded in $L^2$, except on a set which tends to zero
in measure, by point~\ref{enum:conseqinequality_ii}. 
\end{enumerate}

\subsection{Regularization of the curves}
We have a Lipschitz bound for $\alpha$ on sets on which  $M(\cdot)$ is bounded from below. 
\begin{prop}[\cite{PeletierRoegerTA}]\label{prop:boundaprime}
		Let $\gamma_j \in \Gamma$. For all\ \ $0<\lambda < 1$\ \ the function $\alpha$ is Lipschitz continuous on the set 
		\begin{equation}\label{ae} 
			A_j^\lambda:=\left\{s\in E_j:\frac{M_j(s)}{\es}\geq 1-\lambda\right\}
		\end{equation}	
		and 
		$$ |\alpha_j'(s)| \leq \frac {2}{\es(1-\lambda)},\quad \mbox{for a.e. }s\in A_j^\lambda.$$
\end{prop}
\begin{rem}
Note that 
\begin{equation}\label{bigmass} 
		1\leq\frac{M_j(s)}{\es(1-\lambda)}\qquad\text{for a.e. }s\in A_j^\lambda,
\end{equation}	
and 
\begin{equation}\label{smallmass} 
		1< \left(1-\frac{M_j(s)}{\es}\right)\frac1\lambda\ 
		\qquad\text{for a.e. }s\in E_j\setminus A_j^\lambda.
\end{equation}
\end{rem}

This proposition provides a Lipschitz bound on a subset of $E$. 
In the following computations it will be more convenient to approximate the curves $\Gamma$ by a more regular family, in order to have $|\alpha'|$  bounded almost everywhere. 
%By Proposition \ref{prop:boundaprime} we have a Lipschitz bound for	$|\alpha'|$ only on a subset of its domain; in the following computations it will be more convenient to approximate the curves $\g_\es$ by a more regular family, in order to have $|\alpha'|$  bounded a.e. on $[0,L_\es]$. 
\begin{defi}\label{modified}  (Modified curves)\ Let $\gamma_{j} \in \Gamma$ and $0<\lambda <1$. Let $A_j^\lambda$ be as in (\ref{ae}), choose a %periodic
 Lipschitz continuous function $\tilde{\alpha}_{j}:[0,L_{j}] \ra \R$ mod $2\pi$ %with period $L_{\es,j}$
such that
\begin{equation}\label{defatilde}
		\tilde{\alpha}_{j} = \alpha_{j}\quad \mbox{ on } A_j^\lambda,
\end{equation}
\begin{equation}\label{boundatilde}
		 |\tilde{\alpha}_{j}'(s)| \leq \frac {2}{\es(1-\lambda)}\quad \forall\,s \in [0,L_{j}]
\end{equation}
and, according to \pref{concatenation},% for every index $i$ there exists a unique index $j$ such that
\begin{equation}\label{alphaperiodic}
%	\tilde{\alpha}_{\e,i}(L_{\es,i})= \tilde{\alpha}_{\e,j}(0).
	\tilde{\alpha}_{j}(L_{j})= \tilde{\alpha}_{\sigma_\e(j)}(0).
\end{equation}
Set
\begin{equation}\label{thetas}	 
  \tilde{\theta}_{j}=\left( \begin{array}{c}
	\cos \tilde{\alpha}_{j}\\
	\sin \tilde{\alpha}_{j}
  \end{array} \right),\h	
\tilde{\theta}_{j}^\bot=\left( \begin{array}{c}
	\sin \tilde{\alpha}_{j}\\
	-\cos \tilde{\alpha}_{j}
  \end{array} \right),\quad \mbox{so that }\ 
\frac{d}{ds}	\tilde{\theta}_{j}^\bot 	=\tilde{\alpha}'_{j}\tilde{\theta}_{j}.
\end{equation}												

We define $\tilde{\gamma}_{j}$ to be the curve in $\R^2$ which satisfies 
\begin{eqnarray}
		\tilde{\gamma}_{j}(0) 	&=&	\gamma_{j}(0),\label{initial}\\
		\tilde{\gamma}'_{j}(s)	&=& \tilde{\theta}^\bot_{j}(s)\quad \mbox{for all }s \in [0,L_{j}].\label{defgammap}
\end{eqnarray}		 	
Let $\tilde{\Gamma}, \tilde{\mu}, \tilde{P}$ be (respectively) the correspondingly modified  curves, rescaled measures on curves, projections on tangent planes. By construction we have  
\begin{equation}\label{defptilde}
		\tilde{P}(\tilde{\gamma}_j(s))=\tilde{\theta}^\bot_j(s)\! \otimes\, \tilde{\theta}^\bot_j(s).
\end{equation}
\end{defi}
\begin{rem} As in \cite[Remark 7.2, Remark 7.19]{PeletierRoegerTA} %\marginpar{removed bullet}should write everything with _{\e,j}, or everything without, I'm not sure what is better
	\begin{itemize} 
		%\item While $\gamma$ is a closed curve, $\tilde{\gamma}$ may not be closed anymore, but the function $\tilde{\gamma}'$ is periodic;
		\item Both $\gamma_j$ and $\tilde{\gamma}_j$ are defined on the same interval $[0,L_j]$;
		\item Both $\gamma_j$ and $\tilde{\gamma}_j$ are parametrized by arclength, $|\gamma_j'|=|\tilde{\gamma}_j'|=1$;
		\item Note that although a modified $\tilde{u}_\es$ would not make sense, because an open curve cannot be the boundary of any set, we still can define the rescaled measures $\tilde{\mu}_\es$ as
\begin{equation}\label{tildemu} 
		\tilde{\mu}_\es(B):=\es \hf(B\cap \tilde{\Gamma}_\es),\quad \mbox{for all Borel measurable sets }B\subset \R^2;
\end{equation}		
	\item The curves $\tilde \gamma_{j}$ need not be confined to $\Omega$; however, we show in the next section that as $\e\to0$, $\tilde \mu_\e\stackrel{*}\weakto\mu = \tfrac12 \L^2 \llcorner\Omega$. 
\end{itemize}
\end{rem}

\subsection{Weak compactness and the lower bound}

In this section we show that if $u_\e$ is an energy-bounded sequence, then the quantities $u_\e$, $\mu_\e$, $(\mu_\e,P_\e)$, as well as their regularizations,  are weakly compact in the appropriate spaces (Lemmas~\ref{lemmaconvergence} and~\ref{bigestimate}). This provides part of the proof of part~\ref{thetheorem:part1} of Theorem~\ref{thetheorem}. 
The weak convergence also allows us to deduce the lower bound estimate (Lemma~\ref{lemma:divint}), which proves part~\ref{thetheorem:part2} of Theorem~\ref{thetheorem}.

\begin{lemma}\label{lemmaconvergence} Define $\mu:=\frac12 \L^2 \llcorner \o\in RM(\R^2)$. Let the sequence $\{\ue\}$ be such that
$$
\sup_\e \G_\es(\ue) :=  \Lambda<\infty.
$$
After extracting a subsequence, we have the following. As $\es \ra 0$,
\begin{align}\label{convergenceu}
		&\ue \rightharpoonup \frac12 \quad \mbox{weakly in }L^p(\o)\qquad\text{for all $1\leq p<\infty$},\\
\label{convergencemu}
		&\mue \ras \mu\quad \mbox{weakly-\star\ in }RM(\R^2).
\end{align}
Denoting by $\tilde{\gamma}$ and $\tilde{\mu}_\es$ the modified curves and measures (see Definition~\ref{modified}), there exists a constant $C(\lambda,\Lambda)>0$ such that 
\begin{equation}\label{estimategammat}
		\sum_j \int_0^{L_{\es,j}}\left| \tilde{\gamma}'_{\es,j}(s) - \gamma'_{\es,j}(s)\right| ds \leq C(\lambda,\Lambda),
\end{equation}
and
\begin{equation}\label{patch}
		\sup_j \int_0^{L_{\es,j}}\left| \tilde{\gamma}'_{\es,j}(s) - \gamma'_{\es,j}(s)\right| ds \leq \e^{1/2}C(\lambda,\Lambda).
\end{equation}
We have 
\begin{equation}\label{convergencemut}
		\tilde{\mu}_\es \ras \mu\quad \mbox{weakly-\star\ in }RM(\R^2).
\end{equation}
There exists $P\in L^{\infty}(\R^2;\R^{2\times 2})$, with supp$(P)\subseteq \o$ such that
\begin{align}\label{weakp}
		&(\mue,P_\es)  \weakto (\mu,P),\\
\label{weakpt}
		&(\tilde{\mu}_\e,\tilde{P}_\es)  \weakto (\mu,P),
\end{align}
as $\es \ra 0$, in the sense of the weak convergence in $L^p$ for function-measure pairs of Definition \ref{def:weak-coup-conv}, for every $ 1\leq p <\infty$.
\end{lemma}
\bigskip

\begin{rem}
 Let $\{\ue\}\subset K$, than $\{\mue\}$ and $\{\tilde{\mu}_\e\}$ are \textit{tight} and thus relatively compact in $RM(\R^2)$ (see e.g. \cite[Th. 5.1.3]{ags}). 
\end{rem}

\textit{Proof of (\ref{convergenceu}), (\ref{convergencemu}). } Let $g\in C^1_c(\R^2)$. By~\pref{eq:change-of-variables1} (again we drop the subscript $\e$) we have
\begin{multline*}
  \int_{\R^2} g(x)u(x)\, dx = \sum_j \int_0^{L_j}\int_0^{M_j(s)}
    g\bigl(\gamma_j(s) + \tra_j(s,m)\theta_j(s)\bigr)\, dm\,\es ds\\
 = \sum_j \int_0^{L_j}\int_0^{M_j(s)}\Big[
  g\bigl(\gamma_j(s)\bigr)
    + \nabla g\bigl(\gamma_j(s) + \zeta_j(s,m)\theta_j(s))\bigr)\cdot \theta_j(s)
  \Big]\, dm\,\es ds,
\end{multline*}
for some $0\leq \zeta_j(s,m)\leq \tra_j(s,m)$. Therefore
\begin{multline*}
\left| \int_{\R^2} g(x)u(x)\, dx 
  - \sum_j \int_0^{L_j} M_j(s)g\bigl(\gamma_j(s)\bigr)\, ds\right|\\
\begin{aligned}
&\leq {\big\| \nabla g\big\|}_{\infn}\sum_j \int_0^{L_j}\int_0^{M_j(s)}\tra_j(s,m)\, dm\,\es ds
\stackrel{\pref{est:d1}}\leq {\big\| \nabla g\big\|}_{\infn}\, d_1(u,1-u) \\
&  \leq \e{\big\| \nabla g\big\|}_{\infn} \F_\e(u).
\end{aligned}
\end{multline*}
For $\int g(1-u)$ a similar estimate holds.
Also, 
\begin{multline*}
\left| \sum_j \int_0^{L_j} M_j(s)g\bigl(\gamma_j(s)\bigr)\, ds
  - \sum_j \int_0^{L_j} g\bigl(\gamma_j(s)\bigr)\, \e\, ds\right|\\
\begin{aligned}
&= \left|\sum_j\int_0^{L_j} \left(\frac{M_j(s)}\es -1\right)g(\gamma_j(s))\,\es\, ds\right|\\
&\leq \e  \|g\|_{\infn} \Bigl(\sum_j \e L_j\Bigr)^{1/2}
   \Bigg(\sum_j \int_0^{L_j}\frac1{\e^2} \left(\frac{M_j(s)}\es -1\right)^2\, \es\, ds
      \Biggr)^{1/2}\\
&\leq \e  \|g\|_{\infn} \F_\e(u)^{1/2} \G_\e(u)^{1/2}
\end{aligned}
\end{multline*}
Therefore, to prove \pref{convergenceu} we estimate
\begin{align*}
\lefteqn{\left|\int_{\R^2} g(x) u(x)\,dx - \frac12 \int_{\R^2} g(x)\, dx\right|}
\qquad&\\
&\leq \frac12 \left|\int_{\R^2} g(x)u(x)\, dx 
  - \sum_j \int_0^{L_j} M_j(s)g\bigl(\gamma_j(s)\bigr)\, ds\right|\\
&\qquad {}
  + \frac12 \left| \sum_j \int_0^{L_j} M_j(s)g\bigl(\gamma_j(s)\bigr)\, ds - \int_{\R^2} g(x)(1-u(x))\, dx\right|\\
&\leq \e{\big\| \nabla g\big\|}_{\infn} \F_\e(u).
\end{align*}
Since the assumptions on $\G_\e(u)$ imply that $\F_\e(u)$ is bounded, this converges to zero as $\e\to0$, which proves~\pref{convergenceu} for smooth functions $g$. For general $g\in L^{p'}(\R^2)$ we approximate by smooth functions and use the boundedness of $u_\e$ in $L^p(\R^2)$.

To prove~\pref{convergencemu} we remark that 
\begin{align*}
\left|\int_{\R^2} g \, d\mu_\e - \int_{\R^2} g\,d\mu\right| 
&\leq \left|\int_{\R^2} g \, d\mu_\e - \int_{\R^2} gu \right| 
  +\left|\int_{\R^2} gu  - \frac12\int_{\R^2} g\right|\\
&\leq 
 \left|  \sum_j \int_0^{L_j} g\bigl(\gamma_j(s)\bigr)\, \e\, ds
    - \sum_j \int_0^{L_j} M_j(s)g\bigl(\gamma_j(s)\bigr)\, ds\right|\\
& \qquad {}
  + \left| \sum_j \int_0^{L_j} M_j(s)g\bigl(\gamma_j(s)\bigr)\, ds - \int_{\R^2} gu\right|
  + {\big\| \nabla g\big\|}_{\infn} \e\F_\e(u)\\
&\leq \e  \|g\|_{\infn} \F_\e(u)^{1/2} \G_\e(u)^{1/2}
   + 2\e{\big\| \nabla g\big\|}_{\infn} \F_\e(u).
\end{align*}
Again we conclude by this estimate for smooth functions $g$, and extend the result to any $g\in C^0_c(\R^2)$ by using the tightness of $\mue$ and the uniform boundedness of $\mu_\e(\R^2)$.

\drop{
\begin{align*}
\left| \int_\o\! g(x)dx \right.&\left.- 2\es\!\!\int_\o\! g\, d|\nabla u| \right|\leq\\
&\leq \left| \int_\o\! g(x)\,u(x)\,dx - \sum_j\int_0^{L_j} M_j(s)\,g(\gamma_j(s))\, ds \right|+\\
&\qquad+\left|\sum_j\int_0^{L_j} M_j(s)\,g(\gamma_j(s))\, ds -\es\!\!\int_\o\! g\,d|\nabla u|\right|+ \\
&\quad+ \left| \int_\o\! g(x)\,(1-u(x))\,dx - \sum_j\int_0^{L_j} M_j(s)\,g(\gamma_j(s))\, ds \right|+\\
&\qquad+\left|\sum_j\int_0^{L_j} M_j(s)\,g(\gamma_j(s))\, ds -\es\!\!\int_\o\! g\, d|\nabla u|\right|\\
%\label{estimatemu}
%&\leq 2  {\big\| \nabla g\big\|}_{L^{\infty}}\, d_1(u,1-u) + 2\sqrt{\es L} {\big\| g\big\|}_{L^{\infty}}\big( \F(u)-|\Omega| \big)^{1/2}\\
&\leq  2  \es{\big\| \nabla g\big\|}_{L^{\infty}}\,\F_\es(\ue)\ +\ 2\es\sqrt{\F_\es(\ue)} {\big\| g\big\|}_{L^{\infty}}\G_\es(\ue)^{1/2}
\end{align*}

So we have
$$ 	\int_\o g(x)\,u(x)\, dx =  \sum_j \int_0^{L_j} M_j(s)g\bigl(\gamma_j(s)\bigr)\, ds +R,$$
where we can estimate
% \marginpar{suppressed $\e$ in $u_\e$}
\begin{align}
R &= \sum_j \int_0^{L_j}\int_0^{M_j(s)}\nabla g\bigl(\gamma_j(s) + \zeta_j(s,m)\theta_j(s))\bigr)\cdot \theta_j(s)\, dm\,\es ds \notag\\
\label{rest1}
 &\leq {\big\| \nabla g\big\|}_{L^{\infty}}\sum_j \int_0^{L_j}\int_0^{M_j(s)}\tra_j(s,m)\, dm\,\es ds \stackrel{\pref{est:d1}}\leq {\big\| \nabla g\big\|}_{L^{\infty}}\, d_1(u,1-u).
\end{align}
On the other hand we have the trivial equality
\begin{align*}
\int_0^{L_j} M_j(s)g(\gamma_j(s))\, ds - \int_0^{L_j} g(\gamma_j(s))\,\es ds 
=  \int_0^{L_j} \left(\frac{M_j(s)}\es -1\right)g(\gamma_j(s))\,\es ds, 
\end{align*}
so that by summing over $\{ \gamma_j \}$ we find
\begin{align}
 \left|\sum_j \right.&\left.\int_0^{L_j} M_j(s)g(\gamma_j(s))\, ds - \int_\Omega g\, d|\nabla u|\right|\notag\\
&=  \left|\sum_j \right.\left.\int_0^{L_j} M_j(s)g(\gamma_j(s))\, ds -
   \sum_j\int_0^{L_j} g(\gamma_j(s))\,\es ds \right|\notag\\
 &\leq {\big\| g\big\|}_{L^{\infty}} \left(\sum_j\int_0^{L_j}\es ds \right)^{1/2}\left(\sum_j\int_0^{L_j} \left(\frac{M_j(s)}\es -1\right)^2\es ds\right)^{1/2} \notag\\
 \label{rest2}
	 &\leq  \left(\es \sum_j L_j \right)^{1/2}{\big\| g\big\|}_{L^{\infty}}\,\e \G_\e(u) ^{1/2}.
	\end{align}
We observe that the argument can be repeated for $1-u(x)$. Finally, owing to (\ref{rest1}), (\ref{rest2}) we estimate:
\begin{align*}
\left| \int_\o\! g(x)dx \right.&\left.- 2\es\!\!\int_\o\! g\, d|\nabla u| \right|\leq\\
&\leq \left| \int_\o\! g(x)\,u(x)\,dx - \sum_j\int_0^{L_j} M_j(s)\,g(\gamma_j(s))\, ds \right|+\\
&\qquad+\left|\sum_j\int_0^{L_j} M_j(s)\,g(\gamma_j(s))\, ds -\es\!\!\int_\o\! g\,d|\nabla u|\right|+ \\
&\quad+ \left| \int_\o\! g(x)\,(1-u(x))\,dx - \sum_j\int_0^{L_j} M_j(s)\,g(\gamma_j(s))\, ds \right|+\\
&\qquad+\left|\sum_j\int_0^{L_j} M_j(s)\,g(\gamma_j(s))\, ds -\es\!\!\int_\o\! g\, d|\nabla u|\right|\\
%\label{estimatemu}
%&\leq 2  {\big\| \nabla g\big\|}_{L^{\infty}}\, d_1(u,1-u) + 2\sqrt{\es L} {\big\| g\big\|}_{L^{\infty}}\big( \F(u)-|\Omega| \big)^{1/2}\\
&\leq  2  \es{\big\| \nabla g\big\|}_{L^{\infty}}\,\F_\es(\ue)\ +\ 2\es\sqrt{\F_\es(\ue)} {\big\| g\big\|}_{L^{\infty}}\G_\es(\ue)^{1/2}
\end{align*}
%$$ \leq o(1)\ \cdot C\h +\h \es C\h +\h C \cdot o(1).\big) $$
\rightline{$\Box$}
} % end \drop
\medskip

\textit{Proof of (\ref{estimategammat}) and~\pref{patch}}. (As in \cite{PeletierRoegerTA}, with the appropriate substitutions of $\es$) Suppressing the indexes $\es,j$, we compute that
\begin{equation}\label{boundga}
		\int_0^L \!\!\!\bigl|\tilde{\gamma}'(s) -\gamma'(s)\bigr|\, ds \leq \int_0^L \!\big|\tilde{\gamma}'(s) -\gamma'(s) \big|\,\chi_{\left\{\frac{M}{\es}\geq 1-\lambda\right\}}\,ds + \int_0^L \!\big|\tilde{\gamma}'(s) -\gamma'(s) \big|\,\chi_{\left\{\frac{M}{\es}< 1-\lambda\right\}}\,ds.
\end{equation}
Recall that if $s\in A_j^\lambda$ (defined in \pref{ae}), then by (\ref{defatilde}) and (\ref{defgammap}) we have 
\begin{equation}\label{gate}
			\tilde{\gamma}'(s)= \tilde{\theta}^\bot(s)=\theta^\bot(s).
\end{equation}
By definition of $\beta$ it follows that for all $s$
\begin{equation}\label{gate2}
		\big| \gamma'(s) - \theta^\bot(s)\big|^2 = 2\big(1-\sin\beta(s)\big), 
\end{equation}
and 
\begin{equation}\label{gate3}
		\big| \tilde{\gamma}'(s) - \gamma'(s)\big| \leq \big| \tilde{\gamma}'(s)\big|+\big|\gamma'(s) \big| =2.
\end{equation}
Collecting (\ref{gate}), (\ref{gate2}), (\ref{gate3}) and (\ref{bigmass}), (\ref{smallmass}), we can estimate (\ref{boundga}) as
\begin{multline*}
   \int_0^L \!\bigl|\tilde{\gamma}'(s) -\gamma'(s)\bigr|\, ds \leq \sqrt 2\int_0^L |1-\sin\beta(s)|^{1/2} \frac{M(s)}{\es(1-\lambda)} ds + \int_0^L 2 \left(1-\frac{M(s)}{\es}\right)^2\frac1{\lambda^2}\, ds\\
   \begin{aligned}
		&\leq \frac{\sqrt{2L}}{(1-\lambda)}\left( \int_0^L \left|\frac1{\sin\beta(s)}-1\right| \left(\frac{M(s)}{\es}\right)^2 ds\right)^{1/2} +\frac{2}{\lambda^2}\int_0^L \left(1-\frac{M(s)}{\es}\right)^2\, ds\\
		&\leq \frac{\sqrt{2L\e}}{(1-\lambda)}\; \G_\e(u)^{1/2} +\frac{2\e}{\lambda^2}\,\G_\e(u).
  \end{aligned}
\end{multline*}
Since $L_j \leq 1$ for all $j$, we obtain \pref{patch}. 
%For a generic collection of functions $f_j\in L^2(0,L_j)$ it holds
%$$ \left|\sum_j\int_0^{L_j} f_j(s)\,ds \right| \leq \left( \sum_j L_j  \right)^{1/2}\!\! \left( \sum_j\int_0^{L_j} f_j^2(s)\,ds\right)^{1/2},$$
%re-doing the estimate over the sum of all the curves, in conclusion we find	\marginpar{Here was the mistake!}
Turning to~\pref{estimategammat}, we repeat the same estimate while taking all curves together, to find
\begin{equation}
\label{est:pre-gammat}
  \sum_j\int_0^{L_j} \bigl|\tilde{\gamma}_j'(s) -\gamma_j'(s)\bigr|\, ds 
%\leq \frac{\sqrt{2}}{(1-\lambda)}\Big(\sum_j L_j\Big)^{1/2}\left(\frac{ \F(u)-1}{\e}\right)^{1/2} +\frac{2}{\lambda^2}\big(\F(u)-1\big)\\
 \leq \frac{\sqrt{2}}{(1-\lambda)}\Big(\es \sum_j L_j\Big)^{1/2}\G_\es(u)^{1/2} +\frac{2\es}{\lambda^2}\G_\es(u).
 %\leq C(\lambda)\big( \F_\es(u) \G_\es(u)\big)^{1/2} +\frac{2\es^2}{\lambda^2}\G_\es(u).
\end{equation}
This proves~\pref{estimategammat}. 
\medskip

\textit{Proof of (\ref{convergencemut}).} We have to prove that 
$$ \int_{\R^2} g(x)\,d\tilde{\mu}_\es(x) \ra \int_{\R^2} g(x)\,d\mu(x),\h \forall\,g \in  C_c(\R^2).$$
%Owing to (\ref{convergencemu}), it is sufficient to show that there exist constants $C,\delta>0$ such that
%$$ \left| \int_{\R^2} g(x)\,d\tilde{\mu}_\es(x) -\int_{\R^2} g(x)\,d\mue(x) \right| \leq C\es^\delta \|\nabla g\|_{\infty},\hh \forall\,g \in  C^1_c(\R^2).$$
We deduce from~\pref{patch} using
\[
\left|\tilde{\gamma}_j(s)-\gamma_j(s)\right| \leq \left|\tilde{\gamma}_j(0)-\gamma_j(0)\right| +\int_0^s \left|\tilde{\gamma}'_j(\sigma)-\gamma_j'(\sigma)\right|\,d\sigma \leq \int_0^{L_j} \left|\tilde{\gamma}'_j(\sigma)-\gamma_j'(\sigma)\right|\,d\sigma,
\]
the estimate
\begin{equation}
\label{est:estimategamma}
\sum_j\int_0^{L_j} \left|\tilde{\gamma}_j(s)-\gamma_j(s)\right| \, \e\, ds
\leq \e^{3/2} C(\lambda,\Lambda) \sum_j L_j
\leq \e^{1/2} C_1(\lambda,\Lambda).
\end{equation}
Combining this with the calculation
\begin{align*}
\left| \int_{\R^2} g\,d\tilde{\mu}_\es -\int_{\R^2} g\,d\mue \right| 
  &= \left| \sum_j\int_0^{L_j} g(\tilde{\gamma}_j(s))\,\es ds -\sum_j\int_0^{L_j} g(\gamma_j(s))\,\es ds \right|\\
%= \left|\sum_j\int_0^{L_j} g(\tilde{\gamma}_j(s))-g(\tilde{\gamma}_j(s))\,\es ds \right|
&\leq \|\nabla g\|_{\infn}\sum_j \int_0^{L_j}\left|\tilde{\gamma}_j(s)-\gamma_j(s)\right| \es ds,
\end{align*}
%Now, since
%$$ \left|\tilde{\gamma}_j(s)-\gamma_j(s)\right| \leq \left|\tilde{\gamma}_j(0)-\gamma_j(0)\right| +\int_0^s \left|\tilde{\gamma}'_j(\sigma)-\gamma_j'(\sigma)\right|\,d\sigma \leq \int_0^{L_j} \left|\tilde{\gamma}'_j(\sigma)-\gamma_j'(\sigma)\right|\,d\sigma,$$
%owing to (\ref{patch}) we obtain 
we find
\[
\left| \int_{\R^2} g\,d\tilde{\mu}_\es -\int_{\R^2} g\,d\mue \right| \leq \|\nabla g\|_{\infn}\;\es^{1/2} C_1(\lambda,\Lambda).
%\leq \es^{1/2} C_2(\lambda,\Lambda)\|\nabla g\|_{\infty}. 
\]
\qed 

\textit{Proof of (\ref{weakp}) and (\ref{weakpt}).} As a norm for the projections we adopt the Frobenius norm: 
$$ |P|:=\Big(\sum_{i,j=1}^2 P_{ij}^2\Big)^{1/2}.$$
Let $p\geq 1$, since $|P_\es|=1$ for every $\es$, by compactness in $RM(\o)$ \cite[Theorem 5.4.4]{ags} or \cite[Theorem 3.1]{Moser01} and (\ref{convergencemu}), we obtain the existence of a limit point $P\in L^\infty(\o;\R^{2 \times 2})$, such that
$$ 	(\mue,P_\es) \weakto (\mu,P) \mbox{ weakly in the sense of Def. \ref{def:weak-coup-conv} on $\Omega$.}$$ 
In the same way, owing to (\ref{convergencemut}), there exists a $\tilde{P}\in L^\infty(\R^2;\R^{2 \times 2})$ such that
$$ 	(\tilde{\mu}_\es,\tilde{P}_\es) \weakto (\mu,\tilde{P}), \mbox{ weakly in the sense of Def. \ref{def:weak-coup-conv} on $\R^2$}.$$
For every $\eta \in C^1_c(\R^2;\R^2)$ we have
\begin{multline*}
\int_{\R^2} P_\es \,\eta\,d\mue - \int_{\R^2} \tilde{P}_\es \,\eta\,d\tilde{\mu}_\es 
  = \sum_j \int_0^{L_j}\left[P_\es(\gamma_{\es,j}(s)) \,\eta(\gamma_{\es,j}(s)) 
    - \tilde{P}_\es(\tilde{\gamma}_{\es,j}(s)) \,\eta(\tilde{\gamma}_{\es,j}(s))\right]\,\es ds\\
= \sum_j \int_0^{L_j}\left[ P_\es(\gamma_\es)\big( \eta(\gamma_{\es,j}) - \eta(\tilde{\gamma}_{\es,j})\big) + \big( P_\es(\gamma_{\es,j}) -\tilde{P}_\es(\tilde{\gamma}_{\es,j}) \big)\,\eta(\tilde{\gamma}_{\es,j})\right]\,\es ds,
\end{multline*}
and we can estimate
\begin{equation*}
		\big| P_\es(\gamma_{\es,j})\big( \eta(\gamma_{\es,j}) - \eta(\tilde{\gamma}_{\es,j})\big) \big| \leq \|\nabla\eta\|_{\infn}\,|\gamma_{\es,j}-\tilde{\gamma}_{\es,j}|,
\end{equation*}
and
\begin{align*}
  \big| \big( P_\es(\gamma_{\es,j}) -\tilde{P}_\es(\tilde{\gamma}_{\es,j}) \big)\,\eta(\tilde{\gamma}_{\es,j}) \big|& = 
  \Bigl|\bigl(\gamma_{\e,j}'\cdot \eta(\tilde\gamma_{\e,j})\bigr)\gamma_{\e,j}'
    - \bigl(\tilde\gamma_{\e,j}'\cdot \eta(\tilde\gamma_{\e,j})\tilde\gamma_{\e,j}'\bigr)\Big|\\
  &= \Big| \big( (\gamma_{\es,j}'-\tilde{\gamma}_{\es,j}')\cdot\eta(\tilde{\gamma}_{\es,j})\big) \gamma_{\es,j}'+ (\tilde{\gamma}_{\es,j}'\cdot \eta(\tilde{\gamma}_{\es,j}) )(\gamma_{\es,j}'-\tilde{\gamma}_{\es,j}')\Big|\nonumber\\
		& \leq \ 2 \|\eta\|_{\infn}\left|\gamma_{\es,j}'-\tilde{\gamma}_{\es,j}'\right|. 
\end{align*}
Therefore, using estimates (\ref{patch}) and~\pref{est:estimategamma}, there exists a constant $C_2(\lambda,\Lambda)$ such that 
$$ \left|\int_{\R^2} P_\es \,\eta\,d\mue - \int_{\R^2} \tilde{P}_\es \,\eta\,d\tilde{\mu}_\es  \right| \leq \es^{1/2}C_2(\lambda,\Lambda)\|\eta\|_{C^1(\o)},\h \forall\,\eta \in C^1_c(\R^2;\R^2),$$
and we conclude that $P=\tilde{P}$.
\bigskip

%%%%%%%%%%%%%%%%%%%%%%%%%%%%%%%%%%%%%%%%%%%%%%%%%%%%%%%%%%%%%%%%%%%%%%%%%%%%%%%%%%%%%%%%%%%%%%%%%%%%%%
%%%%%%%%%%%%%%%%%%%%%%%%%%%%%%%%%%%%%%%%%%% LEMMA  ESTIMATE %%%%%%%%%%%%%%%%%%%%%%%%%%%%%%%%%%%%%%%%%%
%%%%%%%%%%%%%%%%%%%%%%%%%%%%%%%%%%%%%%%%%%%%%%%%%%%%%%%%%%%%%%%%%%%%%%%%%%%%%%%%%%%%%%%%%%%%%%%%%%%%%%

\begin{lemma}\label{bigestimate}
Let $\es>0$, $u_\e \in K$, $\mu_\e=\es|\nabla u_\e|$ and $\tilde{\alpha_\e},\ldots,\tilde{\mu_\e}$ as in (\ref{defatilde})-(\ref{tildemu}).  Then, for all $0<\lambda <1$ and for all $\eta\in C^1_c(\R^2;\R^2)$, we have
\begin{align}
		\ds\int_{\R^2} \tilde{P_\e}:\nabla \eta\,d\tilde{\mu_\e} &\leq \ds 
		\frac2{(1-\lambda)^2}\G_\e(u_\e)^{1/2}{\|\eta\|}_{L^2(\tilde \mu_\e)}\notag\\
		&+\ds\frac{2\es}{\lambda(1-\lambda)}\G_\e(u_\e){\|\eta\|}_{\infn}\notag\\
		&+\ds\frac{\es\sqrt 2}{(1-\lambda)}\Big(\es \sum_j L_{\e,j}\Big)^{1/2}\G_\e(u_\e)^{1/2}{\|\nabla\eta\|}_{\infn}\notag\\
		&+\ds\frac{2\es^2}{\lambda^2}\G_\e(u_\e){\|\nabla\eta\|}_{\infn}.
\end{align}	
\end{lemma}

%%%%%%%%%%%%%%%%%%%%%%%%%%%%%%%%%%%%%%%%%%%%%%%%%%%%%%%%%%%%%%%%%%%%%%%%%%%%%%%%%%%%%%%%%%%%%%%%%%%%%%
%%%%%%%%%%%%%%%%%%%%%%%%%%%%%%%%%%%%%%%%%%% PROOF  %%%%%%%%%%%%%%%%%%%%%%%%%%%%%%%%%%%%%%%%%%%%%%%%%%%
%%%%%%%%%%%%%%%%%%%%%%%%%%%%%%%%%%%%%%%%%%%%%%%%%%%%%%%%%%%%%%%%%%%%%%%%%%%%%%%%%%%%%%%%%%%%%%%%%%%%%%

\textit{Proof.} We again suppress the subscripts $\e$ for clarity. Write
\begin{align*}
\int_{\R^2} \tilde{P} : \nabla\eta\, d\tilde{\mu} 
  &= \sum_{j=1}^{J} \int_0^{L_j} \tilde{P}(\tilde{\gamma}_j): \nabla\eta(\tilde{\gamma}_j)\, \es\, ds
  \stackrel{\substack{\pref{defgammap},\\\pref{defptilde}}}= 
    \sum_{j=1}^{J}\int_0^{L_j} \tilde{\theta}_j^\bot \cdot(\nabla\eta(\tilde{\gamma}_j)\ 
          \tilde{\gamma}_j')\, \es\, ds\\
  &= \sum_{j=1}^{J}\int_0^{L_j}  \tilde{\theta}_j^\bot \cdot \frac{d}{ds} 
          \eta(\tilde{\gamma}_j)\, \es\, ds \\
  &= -\sum_{j=1}^{J}\int_0^{L_j}  (\tilde{\theta}_j')^\bot \cdot 
          \eta(\tilde{\gamma}_j)\, \es\, ds 
		+ \es \sum_{j=1}^{J}\left[ \tilde{\theta}_j^\bot \cdot \eta(\tilde{\gamma}_j)\right]_0^{L_j},
\end{align*}
and we rewrite this using (\ref{thetas}) as 
\begin{multline}
\label{est:three_estimates}
-\sum_{j=1}^{J}\int_0^{L_j} \!\! \big[\tilde{\alpha}_j'\tilde{\theta}_j \cdot \eta(\tilde{\gamma}_j)\big] \chi_{\left\{\frac{M_j}{\es}\geq 1-\lambda\right\}}\, \es\, ds\ 
		-\sum_{j=1}^{J}\int_0^{L_j}  \!\!\big[\tilde{\alpha}_j'\tilde{\theta}_j \cdot \eta(\tilde{\gamma}_j)\big] \chi_{\left\{\frac{M_j}{\es}< 1-\lambda\right\}}\, \es\, ds \\
{}+	\e\sum_{j=1}^{J} \left[ \tilde{\theta}_j^\bot \cdot \eta(\tilde{\gamma}_j) \right]_0^{L_j}.	
\end{multline}
%
%Let $E\subset \o$ be an open set such that $E\cap \tilde{\Gamma}$ is the graph of a $C^1$ curve, then
%		$$
%		\int_E \tilde{P} : \nabla\eta\, d\tilde{\mu} = \int_0^L \tilde{P}(\tilde{\gamma}): \nabla\eta(\tilde{\gamma})\, \es\, ds
%		 = \int_0^L \tilde{\theta}^\bot \cdot(\nabla\eta(\tilde{\gamma})\ \tilde{\gamma}')\, \es\, ds=
%		$$
%where we used (\ref{defgammap}) and (\ref{defptilde}),
%		$$ 
%		= \int_0^L \tilde{\theta}^\bot \cdot \frac{d}{ds}\eta(\tilde{\gamma})\, \es\, ds = 
%		-\int_0^L \frac{d}{ds}\tilde{\theta}^\bot \cdot \eta(\tilde{\gamma})\, \es\, ds 
%		+ \es \left[ \tilde{\theta}^\bot \cdot \eta(\tilde{\gamma}) \right]_0^L = 
%		$$
%\begin{equation*}
%		=-\int_0^L\!\!\! \big[\tilde{\alpha}'\tilde{\theta} \cdot \eta(\tilde{\gamma})\big] \chi_{\left\{\frac{M}{\es}\geq 1-\lambda\right\}}\, \es\, ds\ 
%		-\int_0^L \!\!\!\big[\tilde{\alpha}'\tilde{\theta} \cdot \eta(\tilde{\gamma})\big] \chi_{\left\{\frac{M}{\es}< 1-\lambda\right\}}\, \es\, ds 
%		+	\es \left[ \tilde{\theta}^\bot \cdot \eta(\tilde{\gamma}) \right]_0^L,	
%\end{equation*}
%where we used (\ref{thetas}). 
Now we separately estimate the three parts of this expression.
\medskip

%%%%%%%%%%%%%%%%%%%%%%%%%%%%%%%%%%%%%%%%%%% ESTIMATE I  %%%%%%%%%%%%%%%%%%%%%%%%%%%%%%%%%%%%%%%%%%%%%%%%%%%

\textit{Estimate I}. 
Observe that as $s\in A_{j}^\lambda$ (defined in \pref{ae}), by (\ref{defatilde}) we have $|\tilde{\alpha}_j'(s)|=|\alpha_j'(s)|$.
%and by (\ref{ae})
%\begin{equation}\label{bigmass} 
%		1\leq\frac{M_j(s)}{\es(1-\lambda)}\leq \left(\frac{M_j(s)}{\es(1-\lambda)}\right)^2.
%\end{equation}		
Therefore, using~\pref{bigmass}, and taking a single curve $\tilde \gamma_j$ to start with, 
\begin{align*}
- \int_0^{L_j}\!\!\! \big[\tilde{\alpha}_j'\tilde{\theta}_j \cdot \eta(\tilde{\gamma}_j)\big] \chi_{\left\{\frac{M_j}{\es}\geq 1-\lambda\right\}}\, \es\, ds
		&\leq \int_0^{L_j} |\alpha_j'| \left(\frac{M_j(s)}{\es(1-\lambda)}\right)^2 |\eta(\tilde{\gamma}_j)| \, \es\, ds\\
		&\leq \left( \int_0^{L_j} |\alpha_j'|^2 \left(\frac{M_j(s)}{\es(1-\lambda)}\right)^4\, \es\, ds\right)^{1/2}\left( \int_0^{L_j} |\eta(\tilde{\gamma}_j)|^2\, \es\, ds\right)^{1/2}\\
		&\leq \frac2{(1-\lambda)^2}\left(\frac14 \int_0^{L_j}\!\! {|\alpha_j'|^2} \left(\frac{M_j(s)}{\es}\right)^4\!\! \es\, ds\right)^{1/2}\!\!\!\left( \int_0^{L_j}\! |\eta(\tilde{\gamma}_j)|^2\, \es\, ds\right)^{1/2}
\end{align*}
Now, re-doing this estimate while summing over all the curves, we find by Proposition~\ref{prop:inequality} 
\begin{equation*}
-\sum_j\int_0^{L_j}\!\! \big[\tilde{\alpha_j}'\tilde{\theta_j} \cdot \eta(\tilde{\gamma_j})\big] \chi_{\left\{\frac{M_j}{\es}\geq 1-\lambda\right\}}\, \es\, ds 
\leq \frac2{(1-\lambda)^2}
		%\left( \frac{\F(u)-1}{\es^2}\right)^{1/2}
		\G_\e(u)^{1/2}{\|\eta\|}_{L^2(\tilde\mu)}.
\end{equation*}
\medskip

%%%%%%%%%%%%%%%%%%%%%%%%%%%%%%%%%%%%%%%%%%% ESTIMATE II  %%%%%%%%%%%%%%%%%%%%%%%%%%%%%%%%%%%%%%%%%%%%%%%%%%%

\textit{Estimate II}. Observe that as $s\in \R/A_{j}^\lambda$, by (\ref{boundatilde}) 
%we have that 
% 	$$|\tilde{\alpha}_j'(s)| \leq \frac {2}{\es(1-\lambda)},$$
and 
%by 
\pref{smallmass},
%\begin{multline*}
%  \left|\int_0^{L_j}\!\! \big[\tilde{\alpha}_j'\tilde{\theta}_j \cdot \eta(\tilde{\gamma}_j)\big] \chi_{\left\{\frac{M_j}{\es}< 1-\lambda\right\}}\, \es\, ds\right|
%		  \leq \int_0^{L_j} \frac {2}{\es(1-\lambda)}\left(1-\frac{M_j(s)}{\es}\right)\frac1\lambda |\eta(\tilde{\gamma}_j)| \, \es\, ds \\
%		 \leq \frac {2}{\lambda(1-\lambda)}\left( \int_0^{L_j} \frac1{\es^2}\left(1-\frac{M_j(s)}{\es}\right)^2\, \es\, ds\right)^{1/2}\left( \int_0^{L_j} |\eta(\tilde{\gamma}_j)|^2\, \es\, ds\right)^{1/2},
%\end{multline*}
%and summing as before we find
%\begin{equation*}
%		\left|\sum_j\int_0^{L_j}\!\!\! \big[\tilde{\alpha}_j'\tilde{\theta}_j \cdot \eta(\tilde{\gamma}_j)\big] \chi_{\left\{\frac{M_j}{\es}< 1-\lambda\right\}}\, \es\, ds\right| \leq \frac2{\lambda(1-\lambda)}
%		%\left( \frac{\F(u)-1}{\es^2}\right)^{1/2}
%		\G_\e(u)^{1/2}{\|\eta\|}_{L^2(\tilde\mu)}.
%\end{equation*}
%Another possibility: \marginpar{What's wrong here?}
\begin{multline*}
-\int_0^{L_j}\!\! \big[\tilde{\alpha}_j'\tilde{\theta}_j \cdot \eta(\tilde{\gamma}_j)\big] \chi_{\left\{\frac{M_j}{\es}< 1-\lambda\right\}}\, \es\, ds
  \leq \int_0^{L_j} \frac {2}{\es(1-\lambda)}\left(1-\frac{M_j(s)}{\es}\right)^2\frac1{\lambda^2} |\eta(\tilde{\gamma}_j)| \, \es\, ds \\
 \leq \frac{2\e}{\lambda^2(1-\lambda)} \;{\|\eta\|}_{\infn}\;
 \int_0^L \frac1{\es^2} \left(1-\frac{M(s)}{\es}\right)^2\, \es\, ds,
\end{multline*}
and summing as before we find
\begin{equation*}
-\sum_j\int_0^{L_j}\!\!\! \big[\tilde{\alpha_j}'\tilde{\theta_j} \cdot \eta(\tilde{\gamma_j})\big]
		\chi_{\left\{\frac{M_j}{\es}< 1-\lambda\right\}}\, \es\, ds \leq 
		\frac{2\e}{\lambda^2(1-\lambda)}\G_\e(u){\|\eta\|}_{\infn}.
\end{equation*}
\medskip

%%%%%%%%%%%%%%%%%%%%%%%%%%%%%%%%%%%%%%%%%%% ESTIMATE III  %%%%%%%%%%%%%%%%%%%%%%%%%%%%%%%%%%%%%%%%%%%%%%%%%%%

\textit{Estimate III}.
Write the last term in~\pref{est:three_estimates}  as
%Consider the sum: %over $j=1,\ldots,J_\e$ of the differences:
\[
\e\sum_{j=1}^{J} \left[ \tilde{\theta}_j^\bot \cdot \eta(\tilde{\gamma}_j) \right]_0^{L_j}
=\e\sum_{j=1}^J\Bigl[\tilde{\g}_j'(L_j)\cdot\eta(\tilde{\g}_j(L_j)) - \tilde{\g}_j'(0)\cdot\eta(\tilde{\g}_j(0))\Bigr].
\]
%Since $\tilde{\alpha}$ is $L$-periodic, by definition (\ref{thetas}) we have that 
By Definition \ref{modified}, $\tilde{\g}_j(0)=\g_j(0)$ for every $j$, and using \pref{concatenation} and~\pref{alphaperiodic} we find 
%for every~$i$ there exists a unique $j$ such that
\begin{equation*}
	\tilde{\g}_j'(0)\cdot\eta(\tilde\g_j(0))= \tilde{\g}_j'(0)\cdot\eta(\g_j(0)) 	= \tilde{\g}_{\sigma_\e^{-1}(j)}'(L_{\sigma_\e^{-1}(j)})\cdot \eta\bigl(\g_{\sigma_\e^{-1}(j)}(L_{\sigma_\e^{-1}(j)})\bigr),
\end{equation*}											
and therefore 
\[
\e\sum_{j=1}^{J} \left[ \tilde{\theta}_j^\bot \cdot \eta(\tilde{\gamma}_j) \right]_0^{L_j}
=\e\sum_{j=1}^{J_\e}\tilde{\g}_j'(L_j) \big[ \eta(\tilde{\gamma}_j(L_j))-\eta(\gamma_j(L_j)) \big].
\]
We estimate the difference in the right-hand side by
\begin{equation*}
		\big| \eta(\tilde{\gamma}_j(L_j))-\eta(\gamma_j(L_j)) \big|\leq {\|\nabla\eta\|}_{\infn}
		\big|\tilde{\gamma}_j(L_j) - \gamma_j(L_j)\big|.
\end{equation*}
Using
\begin{equation*}
		\tilde{\gamma}_j(L_j) -\gamma_j(L_j)=\tilde{\gamma}_j(0) -\gamma_j(0) +\int_0^{L_j} \big[\tilde{\gamma}_j'(s) -\gamma'_j(s) \big]\,ds,
\end{equation*}
by estimate~\pref{est:pre-gammat} we find
\begin{align*}
  \e \sum_{j=1}^J \left[ \tilde{\theta}_j^\bot \cdot \eta(\tilde{\gamma}_j) \right]_0^{L_j}
 &\leq \es{\|\nabla\eta\|}_{\infn}\sum_{j=1}^{J}\int_0^{L_j}\left|\tilde{\gamma}_j'(s) -\gamma'_j(s)\right|\, ds\\
 &\leq \e {\|\nabla\eta\|}_{\infn}\Biggl\{
    \frac{\sqrt{2}}{(1-\lambda)}\Big(\es \sum_j L_j\Big)^{1/2}\G_\es(u)^{1/2} +\frac{2\es}{\lambda^2}\G_\es(u)\Biggr\}.
\end{align*}
\rightline{$\Box$}

%%%%%%%%%%%%%%%%%%%%%%%%%%%%%%%%%%%%%%%%%%%%%%%%%%%%%%%%%%%%%%%%%%%%%%%%%%%%%%%%%%%%%%%%%%%%%%%%%%%%%%
%%%%%%%%%%%%%%%%%%%%%%%%%%%%%%%%%%%%%%%%%%% LEMMA GAMMA INF %%%%%%%%%%%%%%%%%%%%%%%%%%%%%%%%%%%%%%%%%%
%%%%%%%%%%%%%%%%%%%%%%%%%%%%%%%%%%%%%%%%%%%%%%%%%%%%%%%%%%%%%%%%%%%%%%%%%%%%%%%%%%%%%%%%%%%%%%%%%%%%%%

%\begin{lemma}\label{lemma:divint}
%		Let $\{\ue\} \subset K$ be such that $\G_\es(\ue)<\Lambda$, for some constant $\Lambda>0$. Then there exist $H\in L^2(\o;\R^2)$ and a subsequence  $\{\ue\}$ such that
%\begin{enumerate}
%		\item\label{lemma:divint:a} $\ds \lim_{\es \ra 0}\int_\o H_\es(x)\cdot \eta(x)\,d\mue(x) = \frac12 \int_\o H(x)\cdot \eta(x)\,dx,\h 
%		\forall\,\eta \in C^0_c(\o;\R^2)$
%		\item\label{lemma:divint:b} $\ds \liminf_{\es \ra 0}\, \G_\es(\ue) \geq \frac 18 \int_\o \big|H(x)\big|^2 dx$.
%\end{enumerate}		
%\end{lemma}
Define the divergence of a matrix $P=(P_{ij})$ as $(\div P)_i:=\sum_{i,j}\partial_{x_j}  P_{ij} $.
\begin{lemma}\label{lemma:divint}
		Let the sequence $\{\ue\} \subset K$ be such that $\G_\es(\ue)$ is bounded, and let $(\mu,P)$ be a weak limit for $(\mue,P_\es)$, with $\mu = \tfrac12 \L^2\llcorner \Omega$, as in (\ref{weakp}). Extend $P$ by zero outside of~$\Omega$. Then %there exist $H\in L^2(\o;\R^2)$ and a subsequence  $\{\ue\}$ such that
\begin{enumerate}
		\item\label{lemma:divint:div} $\div P \in L^2(\R^2;\R^2)$ ,
		%\item\label{lemma:divint:pn} $P\, n =0$ a.e. on $\partial \o$,
		%\item\label{lemma:divint:a} $\ds \lim_{\es \ra 0}\int_\o P_\es(x) : \nabla\eta(x)\,d\mue(x) =- \frac12 \int_\o \div P(x)\cdot \eta(x)\,dx,\h \forall\,\eta \in C^1(\R^2;\R^2)$,
		\item\label{lemma:divint:b} $\ds \liminf_{\es \ra 0}\, \G_\es(\ue) \geq \frac 18 \int_\o \big|\div P(x)\big|^2 dx$.
\end{enumerate}	
%where $n$ is the unitary outward normal to $\partial \o$.	
\end{lemma}
\medskip

\begin{proof}%[Proof of point \ref{lemma:divint:a}.] 
Note that by Lemma~\ref{lemmaconvergence} the pair $(\mu,P)$ is also the weak limit of $(\tilde\mu_\e, \tilde P_\e)$. 
%Let $P$ be the weak limit for $P_\es$ and $\tilde{P}_\es$, as in (\ref{weakp}, \ref{weakpt}), trivially extended to 0 outside $\o$, then 
By Lemma \ref{bigestimate} we have for all $\lambda\in(0,1)$ and for all $\eta \in C^1_c(\R^2;\R^2)$,
\begin{align*}%\label{eq:divpl2}
 \frac12 \int_{\R^2} P(x):\nabla \eta(x)\,dx
  &=\int_{\R^2} P(x):\nabla \eta(x)\,d\mu(x) \\
  &= \lim_{\es \ra 0}\int_{\R^2} \tilde{P}_\es(x):\nabla \eta (x)\,d\tilde{\mu}_\es(x) \nonumber \\
	 	&\leq \liminf_{\es \ra 0} \frac2{(1-\lambda)^2}\G_\e(u_\e)^{1/2}\|\eta\|_{L^2(\tilde{\mu}_\es)} \\
  &= \frac{\sqrt 2}{(1-\lambda)^2}\,\|\eta\|_{L^2(\o)}\, \liminf_{\es \ra 0} \G_\e(u_\e)^{1/2}.
\end{align*}	
This implies that the divergence of $P$, in the sense of distributions on $\R^2$, is an $L^2$ function; by taking the limit $\lambda\to0$ the inequality in part~\ref{lemma:divint:b} of the Lemma follows. 
\end{proof}

%%%%%%%%%%%%%%%%%%%%%%%%%%%%%%%%%%%%%%%%%%%%%%%%%%%%%%%%%%%%%%%%%%%%%%%%%%%%%%%%%%%%%%%%%%%%%%%%%%%%%%%%%
%%%%%%%%%%%%%%%%%%%%%%%%%%%%%%%%%%%%   SECTION 4  %%%%%%%%%%%%%%%%%%%%%%%%%%%%%%%%%%%%%%%%%%%%%%%%%%%%%
%%%%%%%%%%%%%%%%%%%%%%%%%%%%%%%%%%%%%%%%%%%%%%%%%%%%%%%%%%%%%%%%%%%%%%%%%%%%%%%%%%%%%%%%%%%%%%%%%%%%%%%%%

%%%%%%%%%%%%%%%%%%%%%%%%%%%%%%%%%%%%%%%%%%%%%%%%%%%%%%%%%%%%%%%%%%%%%%%%%%%%%%%%%%%%%%%%%%%%%%%%%%%%%%%%%%%%
%%%%%%%%%%%%%%%%%%%%%%%%%%%%%%%%%%%%%   ESTIMATES ON CURVES   %%%%%%%%%%%%%%%%%%%%%%%%%%%%%%%%%%%%%%%%%%%%%%
%%%%%%%%%%%%%%%%%%%%%%%%%%%%%%%%%%%%%%%%%%%%%%%%%%%%%%%%%%%%%%%%%%%%%%%%%%%%%%%%%%%%%%%%%%%%%%%%%%%%%%%%%%%%

\section{Strong convergence}
\label{sec:strongconvergence}
\subsection{An estimate for the tangents}

In this section we use the nonintersection property of $\partial\spt{\ue}$ and the inequality in Proposition~\ref{prop:inequality} to obtain the crucial bound on the orthogonal projections $P_\es$. The notation is rather involved, because we are dealing with a \textit{system} of curves and Proposition~\ref{prop:inequality} provides a bound only on the $L^2$-norms of $\tilde{\alpha}'_\es$, which \textit{approximate}, as $\es \ra 0$, the curvature of (a smooth approximation of) $\partial\spt{\ue}$. The underlying idea is that if the tangent lines to two nonintersecting curves are far from parallel, then either the supports of the curves are distant (Fig.~\ref{fig:pena}a) or curvature is large (Fig.~\ref{fig:pena}b). 
\begin{figure}[h]
	\begin{minipage}{7cm}
		\begin{center}
		\includegraphics{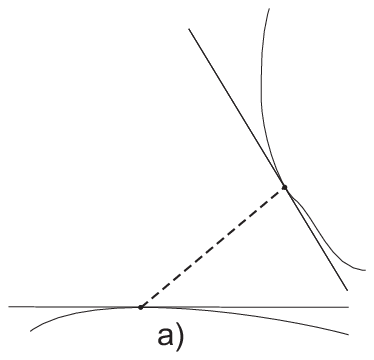}
		\end{center}
	\end{minipage}
	\begin{minipage}{7cm}
		\begin{center}
		\includegraphics[scale=1]{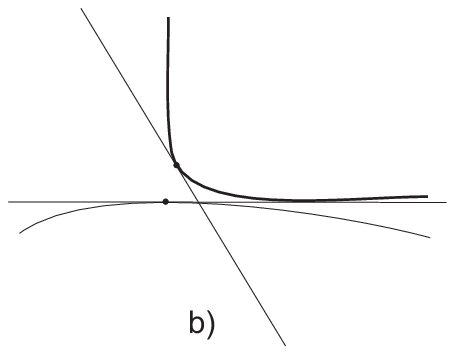}
		\end{center}
	\end{minipage}
	\caption{Curves with distant tangent lines.}
	\label{fig:pena}
\end{figure} 
In Proposition~\ref{esttg}, which  expresses this property, we also include a parameter $\ell>0$, representing  the length of curve on each side of the tangency point that is taken into account. This parameter will be optimized later in the argument. 

We make use of a family of approximations $\{\tilde{\gamma}_\e\}$, similar to the one in Definition \ref{modified}. The approximation is different because in this Section, instead of dividing closed curves into curves with bounded length, we directly exploit the periodicity of the curves in $\Gamma_\e$.
\begin{defi}\label{def:periodicmodified}
	We reparametrize $\Gamma_\e$ (see Lemma \ref{approxi}) as a finite and disjoint family of closed, simple, smooth curves 
	$$ \g_{\e,j}:\R/[0,L_{\e,j}] \ra \o,\quad j=1,\ldots J_\e,$$
	for some $J_\e \in \N$. Note that $L_{\e,j}$ may not be bounded, as $\e \ra 0$, and $\g_{\e,j}$ is $L_{\e,j}$-periodic.
	
	Let $\tilde{\alpha}_{\e,j}$, $\tilde{\theta}_{\e,j}$ be the functions defined in Def. \ref{modified}. According to the new parametrization of $\g$, property \pref{alphaperiodic} entails that $\tilde{\alpha}_{\e,j}$ is $L_{\e,j}$-periodic. We define $\tilde{\g}_{\e,j}$ to be the curve which satisfies
	\begin{eqnarray*}
		\tilde{\gamma}_{\e,j}(0) 	&=&	\gamma_{\e,j}(0),\\
		\tilde{\gamma}'_{\e,j}(s)	&=& \tilde{\theta}^\bot_{\e,j}(s)\quad \mbox{for all }s \in [0,L_{\e,j}].
\end{eqnarray*}	
Note that $\tilde{\g}_{\e,j}$ is not $L_{\e,j}$-periodic, since it may take different values in $s=0$ and in $s=L_{\e,j}$, nonetheless, by definition, $\tilde{\g}'_{\e,j}$ is $L_{\e,j}$-periodic.
\end{defi}
\begin{prop}\label{esttg}
		Let $\g_{\es,1},\ \g_{\es,2}$ be two curves as in Section \ref{definitions}, and let $P_\es,\ \tilde{\g}_{\es,j},\ \tilde{\alpha}_{\es,j},\ \beta_{\es,j},\ L_{\es,j}$, $j=1,2$, be the related quantities as in Sect. \ref{definitions} and Def. \ref{def:periodicmodified}. There exists a constant $C>0$ such that $\forall\, \es>0$, $\forall\, s_1,s_2 \in \R$, and $\forall\, \ell >0$, it holds:
		$$ |P_\es(\g_{\es,1}(s_1))-P_\es(\g_{\es,2}(s_2))|\quad \leq\quad \frac {C}{\ell} |\g_{\es,1}(s_1)- \g_{\es,2}(s_2)| \ \ +\hspace{4cm} $$
		$$ +\ \ C\sum_{j=1,2}\left(\ell^{1/2}\min\left\{\left(\int_{s_j-\ell}^{s_j+\ell} |\tilde{\alpha}_{\es,j}'(\sigma)|^2 d\sigma\right)^{1/2}, \left(2\int_0^{L_{\es,j}}|\tilde{\alpha}_{\es,j}'(\sigma)|^2 d\sigma\right)^{1/2}\right\} +\right.$$
		$$ \left. + \frac 1\ell \min\left\{ \int_{s_j-\ell}^{s_j+\ell} |\tilde{\g}'_{\e,j}(s)-\g'_{\e,j}(s)|\,d\sigma\ ,\ 2\!\int_0^{L_{\es,j}}\! |\tilde{\g}'_{\e,j}(s)-\g'_{\e,j}(s)|\,d\sigma \right\}+ |\tilde{\g}'_{\e,j}(s)-\g'_{\e,j}(s)| \right).$$
\end{prop}

\textit{Proof.}\ For sake of notation, we drop the index $\es$ throughout this whole section. First of all, note that since $P(\g_j)=\g'_j \otimes \g'_j$, it holds 
$$|P(\g_1(s_1))-P(\g_2(s_2))| \leq 2\sqrt{2} \min\bigl\{|\g'_1(s_1)-\g'_2(s_2)|, |\g'_1(s_1)+\g'_2(s_2)|\bigr\},$$ 
moreover, $\forall\, a,b \in S^1$ we have
\begin{equation}\label{equivalence}
	\sqrt{2}|b \times a|\geq \min\{ |b-a| , |b+a|\}  \geq |b \times a|,
\end{equation}	
where `$\times$' denotes the wedge product, i.e.
$$ a \times b = \mbox{det} \left( \begin{array}{cc}
																a_1 & b_1\\
																a_2& b_2\\
													 \end{array} \right) = |a||b|\sin \theta			$$ 
where $\theta$ is the angle between $a=(a_1,a_2)$ and $b=(b_1,b_2)$.
Thus, by~(\ref{equivalence}), it is sufficient to estimate $|\g'_1(s_1)\times\g'_2(s_2)|$.

We divide the proof of this proposition into three lemmas. First we estimate the difference between the tangents of two nonintersecting curves in terms of the curve-tangent distance and of the curve-curve distance (Lemma \ref{twocurves}). Then we estimate the deviation of a curve $\g$ from its tangent line in the point $\g(s)$ in terms of its curvature (Lemma \ref{onecurve}). Finally, in Lemma \ref{approximation}, we express the estimate just obtained in terms of the approximating curves $\tilde{\g}$ defined in Definition \ref{def:periodicmodified}. The case when the tangents lie on the same curve is then straightforward (Remark \ref{remone}). These estimates depend explicitly on a real parameter $\ell$ which can be thought as the length of the stretch of curve we are using for computing the curvature. In order to prove Proposition \ref{prop:strongconv} below, we will take the limit as $\ell \ra 0$, but since we have no lower bound for the length $L$ of a curve, in this step we have to take into account also the possibility $\ell > L$ (Corollary \ref{corperiod}).

\begin{lemma}\label{twocurves}
		Let $\g_i:\R \ra \R^2,\ i=1,2$ be two smooth curves, parametrized by arclength (i.e. $|\g_1'|\equiv|\g_2'|\equiv 1$), and such that $\g_1(\R)\cap \g_2(\R)=\emptyset$. Then there exists a constant $C>0$ such that $\forall\, \ell>0 $ and $\forall\, (s,t) \in \R^2$ it holds:
\begin{align*}
	\ell|\g_1'(s)\times \g_2'(t)| \leq\  &C\left( |\g_1(s)- \g_2(t)|+ \max_{\sigma\in [s-\ell,s+\ell]}|\g_1(s) + \sigma \gamma_1'(s) - \g_1(s+\sigma)| + \right.\\
	& \left. +\max_{\tau\in [t-\ell,t+\ell]}|\g_2(t) + \tau \g_2'(t) - \g_2(t+\tau)|\right).
	\end{align*}
\end{lemma}
\begin{lemma}\label{onecurve}
	Let $\ell>0$, $\g:\R\ra \R^2$ be a smooth curve, then
	$$ \max_{\sigma\in [s,s+\ell]}|\g(s)+\sigma\g'(s) - \g(s+\sigma)|\leq \frac 23 \ell^{3/2}\left(\int_{s}^{s+\ell} |\g''(\sigma)|^2d\sigma\right)^{1/2}.$$	
\end{lemma}
\begin{cor}\label{corperiod}
Let $\ell>0$, $L>0$, $\g:\R\ra \R^2$ be a smooth curve, such that $\g'$ is $L$-periodic, then
	$$ \max_{\sigma\in [s,s+\ell]}|\g(s)+\sigma\g'(s) - \g(s+\sigma)|\leq \frac 23 \ell^{3/2} \min\left\{\left(\int_{s}^{s+\ell} |\g''(\sigma)|^2d\sigma\right)^{1/2}, \left(\int_{0}^L |\g''(\sigma)|^2d\sigma\right)^{1/2} \right\}.$$	
\end{cor}

Remark that we only assume that $\g'$ is periodic (and not $\g$) since we need to apply this corollary to the approximating curves $\tilde{\g}$.
\begin{lemma}\label{approximation}
Let $\g,\ \tilde{\g}, \ \tilde{\alpha},\ \beta,$ be as in Definition \ref{def:periodicmodified}, then $\forall\, s\in \R ,\ \forall\, \ell >0$ it holds:

$$ \max_{\sigma \in [s,s+\ell]}|\g(s)+\sigma\g'(s)-\g(s+\sigma)|\leq \frac 23 \ell^{3/2}\min\left\{ \int_{s}^{s+\ell}|\tilde{\alpha}'(\sigma)|^2d\sigma\, , \, \int_{0}^{L}|\tilde{\alpha}'(\sigma)|^2d\sigma \right\}^{1/2} +  $$
$$ + \min\left\{\int_{s}^{s+\ell}|\tilde{\g}'(s)-\g'(s)|\,d\sigma\, , \, \int_{0}^{L}|\tilde{\g}'(s)-\g'(s)|\,d\sigma\right\} + \ell|\tilde{\g}'(s)-\g'(s)| .$$
\end{lemma}

\textit{Proof of Lemma \ref{twocurves}}\ \ It is not restrictive to assume, $t=s=0$, $\g_1'(0)\times \g_2'(0)\neq 0$, $|\g_1(0) - \g_2(0)|\neq 0$. Let $\ell >0$. Let $\bar{\g}_i(s):= \g_i(0)+s\g'_i(0)$, $s\in[-\ell,\ell]$.
If
$$ \ell|\g_1'(0)\times \g_2'(0)|\leq |\g_1(0)-\g_2(0)|,$$
then, by (\ref{equivalence}) the proof is complete. 
Thus, assume
\begin{equation*}
		\ell|\g_1'(0)\times \g_2'(0)|>|\g_1(0)-\g_2(0)|,
\end{equation*}
which implies that the segments $\bar{\g}_1$ and $\bar{\g}_2$ have an internal crossing point and that  $\bar{d}>0.$ In order to prove that the segments intersect, consider the function
$$ \vfi_1(t):= (\bar{\g}_1(t)-\g_2(0))\times \g_2'(0),\quad t\in [-\ell,\ell], $$
which represents a signed distance between the point $\bar{\g}_1(t)$ and the line which lies on $\g_2'(0)$. The derivative of $\vfi_1$ is 
$$ \vfi_1'(t):= \bar{\g}'_1(t)\times \g_2'(0) = \g'_1(0)\times \g_2'(0),$$
and therefore
$$ \vfi_1(t) = 0\quad \mbox{if and only if}\quad \vfi_1(0) + t\vfi_1' = 0,$$
iff
$$ (\g_1(0)-\g_2(0))\times \g_2'(0) + t\g'_1(0)\times \g_2'(0) = 0.$$
By (\ref{equivalence}), a sufficient condition for such a $t\in[-\ell,\ell]$ to exist is then:
\begin{equation}\label{xing}
		\ell|\g_1'(0)\times \g_2'(0)|>|\g_1(0)-\g_2(0)|.
\end{equation}
If we want to make sure that the two segments intersect, (and not only that $\bar{\g}_1$ intersects the whole line lying on $\bar{\g}_2$), we have to ask also that there exist $s\in [-\ell,\ell]$ such that 
$$ \vfi_2(s)= (\bar{\g}_2(s)-\g_1(0))\times \g_1'(0)=0, $$
which is implied, in the same way as above, by condition (\ref{xing}). Define
$$\bar{d}:=\min\{ d(\bar{\g}_1(-\ell),\bar{\g}_2), d(\bar{\g}_1(\ell),\bar{\g}_2),d(\bar{\g}_1,\bar{\g}_2(-\ell)), d(\bar{\g}_1,\bar{\g}_2(\ell)) \}.$$ 
Now we use the fact the curves do not intersect: if each curve is close enough to its tangent line then two tangent lines cannot cross, otherwise the curves themselves would have to intersect. We claim that either 
\begin{eqnarray*} 
i) & &\ds \max\{|\g_i(s)-\bar{\g}_i(s)|,\ s\in [-\ell,\ell],\ i=1,2\}\geq \frac 12\, \bar{d}\\
 \mbox{or} & & \\
ii) & & \g_1 \cap \g_2 \neq \emptyset.
\end{eqnarray*} 
We argue by contradiction: assume that
\begin{equation}
\label{ineq:to_contradict}
\sup_{s\in [-\ell,\ell]}|\g_i(s)-\bar{\g}_i(s)|< \frac 12\, \bar{d},\ i=1,2, \hspace{2cm}
\end{equation}
then the traslated segments $\bar{\g}_1^\pm:= \bar{\g}_1 \pm \delta(\bar{\g}_1')^\perp,\ \delta <  (1/2) \bar{d}$, intersect the segments $\bar{\g}_2^\pm:= \bar{\g}_2 \pm \delta(\bar{\g}_2')^\perp$ (observe that shifting a segment $\bar{\g_i}$ in the direction perpendicular to $\bar{\g}_i'$ implies proportional changes in distances, (see Fig.\ref{fig:inutile})). Let $P$ be the internal part of the parallelogram given by the intersections of the segments $\bar{\g}_i^\pm$. By \pref{ineq:to_contradict} it holds: $P\cap\g_i \neq \emptyset,\ i=1,2$. By construction, following $\partial P$ in counterclockwise sense, we find: $\bar{\g}_1^+ \cap \g_2$, $\bar{\g}_2^+ \cap \g_1$, $\bar{\g}_1^- \cap \g_2$, $\bar{\g}_2^- \cap \g_1$, which is a contradiction since by Jordan's curve theorem (see e.g. \cite[Theorem 11.7]{Wall72}) $\g_1$ disconnects $P$ into two sets $P_1$ and $P_2$, so that any continuous curve $\g$ with $P\cap \g \neq \emptyset$ and $\g \cap \g_1=\emptyset$ would have either $\{ \g \cap \partial P\}\subset \partial P_1$ or $\{ \g \cap \partial P\}\subset \partial P_2$. By contradiction of \pref{ineq:to_contradict} we conclude
\begin{equation}\label{maxd}
		\max\left\{\sup_{s\in [0,\ell]}|\g_1(s)-\bar{\g}_1(s)|, \sup_{s\in [0,\ell]}|\g_2(s)-\bar{\g}_2(s)|\right\}\geq \frac 12\, \bar{d}.
\end{equation}
\begin{figure}[h]
	\begin{center}
		\psfrag{g1m}{$\bar{\g}_1^-$}
		\psfrag{g1}{$\bar{\g}_1$}
		\psfrag{g1p}{$\bar{\g}_1^+$}
		\psfrag{g2m}{$\bar{\g}_2^-$}
		\psfrag{g2}{$\bar{\g}_2$}
		\psfrag{g2p}{$\bar{\g}_2^+$}
		\psfrag{d}{$\bar{d}$}
		\psfrag{del}{$\delta$}
		\psfrag{P}{$P$}
		\includegraphics[bb=196 644 402 780]{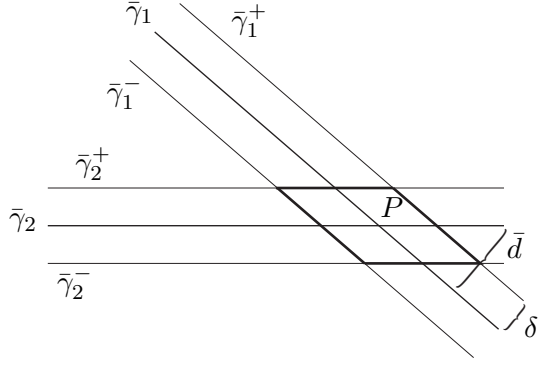}
	\end{center}
	\caption{Crossing segments force curves to intersect inside the parallelogram $P$.}
	\label{fig:inutile}
\end{figure} 
Now let us estimate $\bar{d}$ from below. Denoting by $r_2$ the line which lies on $\g_2'(0)$, we have
$$ d(\bar{\g}_1(t),r_2)= |(\bar{\g}_1(t)-\bar{\g}_2(s))\times \g_2'(0)|,\quad \forall\,s,t\in [-\ell,\ell],$$
so, in particular, it holds:
$$ \min\{ d(\bar{\g}_1(\ell),\bar{\g}_2), d(\bar{\g}_1,\bar{\g}_2(\ell))\} \geq \min\{ |(\bar{\g}_1(\ell) - \bar{\g}_2(\ell))\times \g_2'(0)|, |(\bar{\g}_2(\ell) - \bar{\g}_1(\ell))\times \g_1'(0)| \}.$$
We compute
\begin{align*} 
		|(\bar{\g}_2(\ell) - \bar{\g}_1(\ell))\times \g_1'(0)| &= |(\g_2(0)+\ell \g_2'(0) -\g_1(0) - \ell \g_1'(0))\times \g_1'(0)|\\
		& = |(\g_2(0)-\g_1(0) +\ell(\g_2'(0)-\g_1'(0))\times \g_1'(0)|\\
		& \geq \ell|\g_2'(0)\times \g_1'(0)| - |\g_2(0)-\g_1(0)|.
\end{align*}		
The same estimate holds for $|(\bar{\g}_1(\ell) - \bar{\g}_2(\ell))\times \g_2'(0)|$, and for the endpoints in $-\ell$. By (\ref{maxd}) we find that $\forall\, \ell >0$
\begin{align}
 		\ell|\g_2'(0)\times \g_1'(0)| - |\g_2(0)-\g_1(0)| &\leq 2 \max\left\{\max_{s\in [-\ell,\ell]}|\g_i(s)-\bar{\g}_i(s)|:\ i=1,2 \right\}\nonumber\\
 		& \leq 2\sum_{i=1,2}\max_{s\in [-\ell,\ell]}|\g_i(s)-\bar{\g}_i(s)|.\nonumber
\end{align} 		
\rightline{$\Box$}
\medskip

\textit{Proof of Lemma \ref{onecurve}}\ \ 
Again, it is not restrictive to prove the statement in the point $s=0$. For every $s\in [0,\ell]$ it holds:
$$ \big|\g(s) - \big(\g(0)+s\g'(0)\big)\big|=\left| \int_0^s \gamma'(\sigma)d\sigma - s\gamma'(0)\right|\leq   \int_0^s \left|  \gamma'(\sigma) - \gamma'(0)\right| d\sigma = $$	
$$ =\int_0^s \left|  \gamma'(0) + \int_0^\sigma \gamma''(\tau)\,d\tau - \gamma'(0)\right| d\sigma \leq \int_0^s \int_0^\sigma |\gamma''(\tau)|\,d\tau\, d\sigma $$
$$ \leq \int_0^s \left(\int_0^\sigma 1\,d\tau\right)^{1/2} \!\!\!\left(\int_0^\sigma |\gamma''(\tau)|^2\tau\right)^{1/2} \!\!\!d\sigma = \int_0^s \sigma^{1/2} \left(\int_0^\sigma |\gamma''(\tau)|^2\tau\right)^{1/2}d\sigma\leq$$
$$\leq \int_0^s \sigma^{1/2}d\sigma \left(\int_0^s |\gamma''(\tau)|^2d\tau\right)^{1/2}=\frac 23 s^{3/2}\left(\int_0^s |\gamma''(\tau)|^2d\tau\right)^{1/2} \leq \frac 23 \ell^{3/2}\left(\int_0^\ell |\gamma''(\tau)|^2d\tau\right)^{1/2} .$$
\rightline{$\Box$}
\medskip

\textit{Proof of Corollary \ref{corperiod}}\ \
Let $\ell, L>0$. If $\ell<L$ we obtain the thesis by Lemma \ref{onecurve}. Let then $L<\ell$. We argue by induction. Assume first that $s\in [0,L]$; following the proof of Lemma \ref{onecurve} we get
\begin{equation*}
		\big|\gamma(s) - \big(\gamma(0)+s\gamma'(0)\big)\big| \leq \frac 23 s^{3/2}\left(\int_0^s |\gamma''(\tau)|^2d\tau\right)^{1/2} \leq \frac 23 L^{3/2}\left(\int_0^L |\gamma''(\tau)|^2d\tau\right)^{1/2}.
\end{equation*}
Now let $n \in \N$, ($\ell>(n+1)L$), and assume that 
\begin{equation}\label{induction}	 
		\max_{s\in [0,nL]}|\g(0)+s\g'(0) - \g(s)|\leq \frac 23 (nL)^{3/2} \left(\int_{0}^L |\g''(\sigma)|^2d\sigma\right)^{1/2},
\end{equation}		
we have	to show the analogue estimate for all $s\in [nL,(n+1)L]$. For such an $s$ we have:
\begin{align} 
		|\g(s) -(\g(0) +s\g'(0))| = &|\g(s) - \g(nL) + \g(nL) - \g(0) - s\g'(0) +nL\g'(0) - nL\g'(0) | \nonumber \\
															\leq 	& |\g(nL) - (\g(0) +nL\g'(0))| + |\g(s) - \g(nL) -(s-nL)\g'(0)|,  \label{euno}
\end{align}
by the induction hypothesis (\ref{induction}) it holds
\begin{equation}\label{edue}
		|\g(nL) - (\g(0) +nL\g'(0))| \leq \frac 23 (nL)^{3/2}\left(\int_{0}^L |\g''(\sigma)|^2d\sigma\right)^{1/2}.		
\end{equation}
On the other hand, by $L$-periodicity of $\g'$ and Lemma \ref{onecurve}
\begin{align} 
		|\g(s) - \g(nL) &+(s-nL)\g'(0)| = |\g(s) - \g(nL) -(s-nL)\g'(nL)|\leq \nonumber \\
		 \leq & \frac 23 (s-nL)^{3/2}\left(\int_{nL}^{s} |\g''(\sigma)|^2d\sigma\right)^{1/2} \leq \frac 23 (s-nL)^{3/2}\left(\int_0^L |\g''(\sigma)|^2d\sigma\right)^{1/2}.\label{etre}
\end{align}		 								
Combining (\ref{euno}), (\ref{edue}), (\ref{etre}), for all $s\in [nL,(n+1)L]$ we find
$$ |\g(s) -(\g(0) +s\g'(0))| \leq \frac 23 (nL)^{3/2}\left(\int_{0}^L |\g''(\sigma)|^2d\sigma\right)^{1/2} + \frac 23 (s-nL)^{3/2}\left(\int_0^L |\g''(\sigma)|^2d\sigma\right)^{1/2} $$
and by superlinearity of $x^{3/2}$ we conclude
$$ |\g(s) -(\g(0) +s\g'(0))| \leq \frac 23 s^{3/2} \left(\int_0^L |\g''(\sigma)|^2d\sigma\right)^{1/2}\leq \frac 23 \ell^{3/2} \left(\int_0^L |\g''(\sigma)|^2d\sigma\right)^{1/2}.$$
\rightline{$\Box$}
\medskip

\textit{Proof of Lemma \ref{approximation}}\ \ As in Definition \ref{def:periodicmodified}, define a family of approximating curves $\{\tilde{\g}_\sigma\}_{\sigma\in \R}$ by 
\begin{eqnarray*}
		\tilde{\gamma}_\sigma(\sigma) 	&=&	\gamma(\sigma),\\
		\tilde{\gamma}_\sigma'(s)	&=& \tilde{\theta}^\bot(s)\quad \mbox{for all }s \in \R,
\end{eqnarray*}
see definitions (\ref{defatilde})-\dots-(\ref{defgammap}). Observe that $\forall\,s,\sigma \in \R:\ $ $\tilde{\g}'_\sigma(s)=\tilde{\g}'(s),$ $\tilde{\g}''_\sigma(s)=\tilde{\g}''(s).$
Then, by $L$-periodicity of $\g$, for all $s\in \R$, $\forall\,r \in[0,\ell]$, there exists $t\in \R$ such that $\g(s)=\g(t)$, $L> (s+r) -t >0$, and it holds:
\begin{align}
		|\g(s)+r\g'(s)-\g(s+r)| &= |\g(t)+r\g'(t)-\g(s+r)| \\
		&\leq |\tilde{\g}_\sigma(t) +r  \tilde{\g}'_\sigma(t) -\tilde{\g}_\sigma(s+r)| +\nonumber \\
		&+  |\g(t) - \tilde{\g}_\sigma(t)| + r|\g'(t)- \tilde{\g}'_\sigma(t)| + |\g(s+r) -\tilde{\g}_\sigma(s+r)|.\nonumber
\end{align}		
By Corollary \ref{corperiod} we have
\begin{align*}
	|\tilde{\g}_\sigma(t) +r  \tilde{\g}'_\sigma(t) -\tilde{\g}_\sigma(s+r)| &\leq \frac 23 r^{3/2}\min\left\{ \int_{s}^{s+r}|\tilde{\alpha}'(\tau)|^2d\tau\, , \int_{0}^{L}|\tilde{\alpha}'(\tau)|^2d\tau \right\}^{1/2} \\
	& \leq \frac 23 \ell^{3/2}\min\left\{ \int_{s}^{s+\ell}|\tilde{\alpha}'(\tau)|^2d\tau\, , \int_{0}^{L}|\tilde{\alpha}'(\tau)|^2d\tau \right\}^{1/2}.
\end{align*}	
By definition of $\{\tilde{\g}_\sigma\}_{\sigma\in \R}$, choosing $\sigma=t$ we find
$$|\g(t) - \tilde{\g}_t(t)|=0,$$
$$|\g(s+r) - \tilde{\g}_t(s+r)|\leq \int_{t}^{s+r}\hspace{-0.6cm} |\g'(\tau)- \tilde{\g}'(\tau)|d\tau \leq \min\left\{ \int_{s}^{s+r} \hspace{-0.6cm}|\g'(\tau)- \tilde{\g}'(\tau)|d\tau\, , \, \int_{0}^{L}\hspace{-0.3cm} |\g'(\tau)- \tilde{\g}'(\tau)|d\tau\right\}.$$
\rightline{$\Box$}

\begin{rem}\label{remone}
If there is just one smooth periodic curve $\g$ (instead of $\g_1$ and $\g_2$), then we can obtain the same estimate as in Proposition \ref{esttg}, using the same arguments as in Lemmas \ref{twocurves}--\ref{approximation}. Let $L,\ell>0$ and $(s,t)\in \R^2$, we address three cases:

$i)$ Case $|t-s|>2\ell$: then $\g([s-\ell, s+\ell])\cap \g([t-\ell, t+\ell])=\emptyset$ and we can apply Lemmas \ref{twocurves} - \ref{approximation} directly, as for two disjoint curves.

$ii)$ Case $|t-s|\leq 2\ell$, $L>\ell$: (assume $s<t$) we have $[s,t]\subset [s,s+\ell]\cup [t-\ell,t]$, and it holds:
\begin{equation}\label{ole}
	\frac{\sqrt{2}}{4}|\g'(s)\times \g'(t)|\leq |\g'(s) - \g'(t)|= \left|\int_s^t \g''(\sigma)\, d\sigma \right| \leq \left( \int_s^t 1\, d\sigma\right)^{1/2}\left(\int_s^t |\g''(\sigma)|^2 d\sigma \right)^{1/2}
\end{equation}	
\begin{equation*}
		\leq (2\ell)^{1/2}\left( \int_{s}^{s+\ell}|\g''(\sigma)|^2 +\int_{t-\ell}^{t}|\g''(\sigma)|^2\right)^{1/2}\!\!\!.		
\end{equation*}

$iii)$ Case $|t-s| \leq 2\ell$, $L\leq \ell $: by periodicity of $\g'$ we can find $\bar{s},\bar{t} \in \R$ such that $\g'(s)=\g'(\bar{s}),\ \g'(t)=\g'(\bar{t})$ and $|\bar{s}-\bar{t}|\leq L$. Then, by (\ref{ole}) we find
\begin{equation}\label{ole2}
		\frac{\sqrt{2}}{4}|\g'(s)\times \g'(t)|\leq \left|\int_{\bar{s}}^{\bar{t}} \g''(\sigma)\, d\sigma \right| \leq (L)^{1/2}\left( \int_0^L|\g''(\sigma)|^2\right)^{1/2}\!\!\!.		
\end{equation}		
Finally, we introduce the approximating curves $\tilde{\g}$ as in Lemma \ref{approximation}, and we use equations (\ref{ole}), (\ref{ole2}):
\begin{align*} 
		|\g'(s) - \g'(t)|&\leq |\tilde{\g}'(s)-\tilde{\g}'(t)| + |\g'(s)-\tilde{\g}'(s)| + |\tilde{\g}'(t)-\g'(t)|\leq \\
		&\leq \frac 23 \ell^{1/2}\min\left\{ \int_{s-\ell}^{s+\ell}|\tilde{\alpha}'(\sigma)|^2d\sigma\, , \, \int_{0}^{L}|\tilde{\alpha}'(\sigma)|^2d\sigma \right\}^{1/2} +  \\
 		&+\frac 23 \ell^{1/2}\min\left\{ \int_{t-\ell}^{t+\ell}|\tilde{\alpha}'(\sigma)|^2d\sigma\, , \, \int_{0}^{L}|\tilde{\alpha}'(\sigma)|^2d\sigma \right\}^{1/2} +  \\
 		&+ |\g'(s) - \tilde{\g}'(s)| + |\tilde{\g}'(t)-\g'(t)|.
\end{align*}		
\end{rem}

%%%%%%%%%%%%%%%%%%%%%%%%%%%%%%%%%%%%%%%%%%%%%%%%%%%%%%%%%%%%%%%%%%%%%%%%%%%%%%%%%%%%%%%%%%%%%%%%%%%%%%%%%
%%%%%%%%%%%%%%%%%%%%%%%%%%%%%%%%%%%%   SECTION 4.2  %%%%%%%%%%%%%%%%%%%%%%%%%%%%%%%%%%%%%%%%%%%%%%%%%%%%%
%%%%%%%%%%%%%%%%%%%%%%%%%%%%%%%%%%%%%%%%%%%%%%%%%%%%%%%%%%%%%%%%%%%%%%%%%%%%%%%%%%%%%%%%%%%%%%%%%%%%%%%%%

%%%%%%%%%%%%%%%%%%%%%%%%%%%%%%%%%%%%%%%%%%%%%%%%%%%%%%%%%%%%%%%%%%%%%%%%%%%%%%%%%%%%%%%%%%%%%%%%%%%%%%%%%%%%
%%%%%%%%%%%%%%%%%%%%%%%%%%%%%%%%%%%%%  COMPACTNESS  %%%%%%%%%%%%%%%%%%%%%%%%%%%%%%%%%%%%%%%%%%%%%%%%%%%%%%%
%%%%%%%%%%%%%%%%%%%%%%%%%%%%%%%%%%%%%%%%%%%%%%%%%%%%%%%%%%%%%%%%%%%%%%%%%%%%%%%%%%%%%%%%%%%%%%%%%%%%%%%%%%%%

\subsection{Compactness in the strong topology}\label{sec:strongtop}

We first comment on the definition of \textit{weak} and \textit{strong} convergence for a couple of functions and measures $(\mue,P_\es)$. Let $\{\mue\} \subset RM(\o)$, $\mue \ras \bar{\mu}=\frac12\L^2\llcorner \o$ and let $\{P_\es\} \subset L^1(\o,\mue;\R^{2 \times 2})$ such that $|P_\es|=1$. %Define
%\begin{equation}\label{def:pushed}
%	\g_\es:= (id\times P_\es)_\#\,\mue \in RM(\o \times \R^{2 \times 2}),
%\end{equation}
%i.e. the family of measures which satisfy
%$$ \int_{\o \times \R^{2 \times 2}} \vfi(x,y)\,d\g_\es(x,y)= \int_\o \vfi(x,P_\es(x))\,d\mue(x),\h \forall\, \vfi \in C^0_b(\o\times \R^{2 \times 2}).$$
By compactness there exists a subsequence weakly-$*$ converging to a measure $\gamma \in RM(\o\times \R^{2 \times 2})$, %such that (\textit{weak} convergence)\marginpar{Update}
%\begin{equation}\label{eq:limitpf}
%	\lim_{\es \ra 0} \int_{\o \times \R^{2 \times 2}} \vfi(x,y)\,d\g_\es(x,y)= \int_{\o \times \R^{2 \times 2}} \vfi(x,y)\,d\g(x,y),\h \forall\, \vfi \in C^0_b(\o\times \R^{2 \times 2}), 
%\end{equation}
%but in general it is not true that (\textit{strong} convergence)
%\begin{equation}\label{eq:decomp}
%	\exists\, P:\o\ra \R^{2 \times 2}\ : \ \g=(id \times P)_\#\bar{\mu},
%\end{equation}
but we can only represent the limit measure $\g$ through a family of Young measures $\{\nu_x\}_{x\in \o}$ (see e.g. \cite{afp}) satisfying
$$\int_{\o \times \R^{2 \times 2}} \vfi(x,y)\,d\gamma(x,y)= \int_\o \left(\int_{\R^{2 \times 2}} \vfi(x,y)\,d\nu_x(y)\right)d\bar{\mu}(x).$$

In this section we prove that for a sequence $\{\ue\}$ with bounded energy, it is possible to decompose the limit measure as $[\bar{\mu},P]$ and we show which properties of $P_\es$ are inherited by $P$ in the limit. Proposition~\ref{prop:strongconv} collects the statements, and this proposition ends the proof of part~\ref{thetheorem:part1} of Theorem~\ref{thetheorem}.

\begin{prop}\label{prop:strongconv}
		Let $\{\ue\} \subset K$ be a smooth sequence such that $ \G_\es(\ue)\leq \Lambda$ for some $\Lambda>0$, and let $\mue, P_\es$ be the related sequences of measures on the boundary of the support and orthogonal projections on the tangent space. Let $\bar{\mu}=\frac12\L^2\llcorner \o \in RM(\o).$ Then there exists $P\in L^2(\o; \R^{2\times 2})$ such that, up to subsequences, 
\begin{equation}\label{eq:strongconv}
	\begin{array}{c}
		\ds \lim_{\es \ra 0} \int_\o \vfi(x,P_\es(x))\,d\mue(x)= \int_\o \vfi(x,P(x))\,d\bar{\mu}(x),\quad \forall\, \vfi \in C^0(\o\times \R^{2\times 2}).
%		\ds \mbox{such that}\quad |\vfi(x,P)|\leq a(x) +b(x)|P|^2,\quad \mbox{ for $a,b \in C^0(\o)$}.
	\end{array}
\end{equation}
		Thus
		$$ (P_\es,\mue) \ra (P,\bar{\mu})\h \mbox{strongly in $L^2$, in the sense of Def. \ref{def:strongconv}.}$$ 
		Moreover $P$ satisfies
		\begin{subequations}
		\label{pb:main}
		\begin{eqnarray}
         P^2 = P && \mbox{a.e. in } \o,\label{pr:pro}\\
         \rank(P)=1 && \mbox{a.e. in } \o,\label{pr:rank}\\
         P \mbox{ is symmetric} && \mbox{a.e. in } \o,\label{pr:symm}\\
         \div P \in L^2(\R^2;\R^2) && (\mbox{extended to $0$ outside  }\o),\label{pr:divpint}\\
         P\, \div P = 0  && \mbox{a.e. in } \o.\label{pr:pdivpzero}
		\end{eqnarray}
		\end{subequations}
\end{prop}
\textit{Proof.} First of all we note that properties (\ref{pr:pro}--\ref{pr:symm}) are a direct consequence of the strong convergence (\ref{eq:strongconv}) and (\ref{pj:properties}). Property (\ref{pr:divpint}) is proved in Lemma \ref{lemma:divint}. Property (\ref{pr:pdivpzero}) corresponds to $P\cdot H=0$, which is trivially true at level $\es$ since the interfaces are smooth; it is conseved in the limit as $\es \ra 0$ owing to (\ref{weakp}), Lemma \ref{lemma:divint}, (\ref{eq:strongconv}), and Theorem \ref{theo:weakstrong}. 

Let $\{\rho^k\}$  be a sequence of smooth mollifiers in $\R^2$, and let $(\mue,P_\es)$ be a subsequence such that the graph measures $[\mu_\es,P_\e]$ converge to $\gamma \in RM(\o\times \R^{2 \times 2})$ in the weak-$*$ sense. Let $\{\nu_x\}_{x\in \o}$ be the family of Young measures associated to $\g$ and let $S:=\{f\in \R^{2 \times 2}: |f|\leq 1\}$. In order to prove (\ref{eq:strongconv}) it sufficient to show that 
\begin{equation}\label{eq:youngzero}
	\lim_{k\ra \infty} \lim_{\es \ra 0}\ \int_{\o \times \o} \rho^k(x-y)|P_\es(x)-P_\es(y)|\,d\mue(x)\,d\mue(y) = 0.
\end{equation}
In fact, if we show that 
\begin{equation}\label{eq:limyoung}
	\lim_{k\ra \infty} \lim_{\es \ra 0}\ \int_{\o \times \o} \rho^k(x-y)|P_\es(x)-P_\es(y)|\,d\mue(x)\,d\mue(y) = \frac 12\int_\o \int_S \int_S |f-g|\,d\nu_x(g)\,d\nu_x(f)\,d\bar{\mu}(x),
\end{equation}
then equation (\ref{eq:youngzero}) implies that 
$$ \int_\o \int_S \int_S |f-g|\,d\nu_x(g)\,d\nu_x(f)\,d\bar{\mu}(x) =0, $$
that is
$$ \int_S \int_S |f-g|\,d\nu_x(g)\,d\nu_x(f)=0\quad \mbox{for $\L^2$-a.e. }x\in \o,  $$
and this is true if and only if the support of each $\nu_x$ is atomic, i.e. if there exists a function $\bar{P}:\o \ra S$ such that
$$ \nu_x =\delta_{\bar{P}(x)}\quad \mbox{for $\L^2$-a.e. }x\in \o.  $$
Notice that $ \bar{P}(x)=\int_S x \ d\nu_x$ is measurable, owing to the weak measurability of $x \mapsto \nu_x$. Therefore we have that
$$\int_{\o \times S} \vfi(x,y)\,d\gamma(x,y)=  \int_\o \left(\int_S \vfi(x,y)\,d\delta_{\bar{P}(x)}(y)\right)d\bar{\mu}(x) =  \int_\o \vfi(x,\bar{P}(x))\,d\bar{\mu}(x), $$
i.e. $\gamma= [\bar{P},\bar{\mu}]$ and the measure-function pairs $(\mue,P_\es) $ strongly converge to $(\bar{\mu}, \bar{P})$.
\smallskip

The remaining part of the section is devoted to the proof of (\ref{eq:limyoung}) and of (\ref{eq:youngzero}).
\smallskip

\subsection{Proof of (\ref{eq:limyoung}).} For sake of brevity, denote $S:=\{f\in \R^{2 \times 2}: |f|\leq 1\}$, $Q:=\o \times S$ and $\vfi^k:Q^2\ra \R$, $\vfi^k(x,f,y,g):=\rho^k(x-y)|f-g|$. Note that the weak convergence of $[\mue,P_\e]$ to $\gamma$ on $Q$ implies that the product measures $[\mue,P_\e] \times [\mue,P_\e]$ converge weakly to $\gamma\times\gamma$ on $Q\times Q$.  We have
$$\int_{\o \times \o} \rho^k(x-y)|P_\es(x)-P_\es(y)|\,d\mue(x)\,d\mue(y)=\int_{Q^2} \vfi^k(x,f,y,g)\,d([\mue,P_\e] \times [\mue,P_\e])(x,f,y,g), $$
and passing to the limit as $\es \ra 0$,  we obtain
\begin{equation}\label{tonelli}
		\int_{Q^2} \vfi^k(x,f,y,g)\,d(\gamma\ext\gamma)(x,f,y,g). 
\end{equation}	
By Fubini's theorem this is equal to
$$\int_Q\int_Q \vfi^k(x,f,y,g)\,d\gamma(y,g)d\gamma(x,f), $$
which we can now disintegrate into
$$ \int_Q \left(\int_\o \int_S \vfi^k(x,f,y,g)d\nu_y(g) d\bar{\mu}(y) \right)d\gamma(x,f).$$
Define
\[
\psi(x,f):=\frac12 \int_S |f-g|\, d\nu_x(g) \in L^1(Q,\g),
\]
and define $\psi^k\in C^0(Q)$ as the partial convolution, with respect to $x$, of $\rho^k$ with $\psi$:
\begin{align*}
\psi^k(x,f):=(\rho^k \conv_x \psi) (x,f)&= \int_\o \left(\rho^k(x-y)\int_S |f-g|\, d\nu_y(g)\right)d\bar{\mu}(y) \\
&= \int_\Omega\int_S \vfi^k(x,f,y,g)\,d\nu_y(g) d\bar{\mu}(y) .
\end{align*}
By standard results on convolution 
$$ \psi^k \ra \psi,\quad \mbox{strongly in $L^1(Q,\g)$, as } k\ \ra \infty,$$
which implies
$$ \lim_{k\ra \infty} \int_Q \psi^k(x,f)\,d\g(x,f) = \int_Q \psi(x,f)\,d\g(x,f),$$
that is \pref{eq:limyoung}. 
\qed
\subsection{Proof of (\ref{eq:youngzero}).} Since we are still in the context of the proof of Proposition~\ref{prop:strongconv}, we adopt the assumptions of that Proposition.
\begin{prop}\label{estimatecomp}
Under the hypothesis of Proposition \ref{prop:strongconv} there exists a constant $C>0$ such that, $\forall\,k\in \N$, $\forall\,\es >0$, and $\forall\, \ell>0$, we have
$$ \int_{\o \times \o} \rho^k(x-y)|P_\es(x)-P_\es(y)|\,d\mue(x)\,d\mue(y)\quad \leq \quad C\,\frac {1}{k\ell} \left\| \rho^k \conv \mue\right\|_{L^1(\o;\mue)} +$$ 		
$$ + \quad C\, \ell\, \left\| \rho^k \conv \mue\right\|^{1/2}_{L^{\infty}(\o)} \left\| \rho^k \conv \mue\right\|^{1/2}_{L^1(\o;\mue)} \left( \int_0^{L_\es} |\tilde{\alpha}_\es'(\sigma)|^2\es d\sigma\right)^{1/2}\quad +$$
$$ +\quad C\,\es\left\| \rho^k \conv \mue\right\|^{1/2}_{L^{\infty}(\o)}\,\left\| \rho^k \conv \mue\right\|^{1/2}_{L^1(\o;\mue)}\, \left( \G_\e(\ue)\right)^{1/2}+ $$
$$ + \quad C\, \es \, \left\| \rho^k \conv \mue\right\|_{L^2(\o;\mue)} \left(\G_\e(\ue)\right)^{1/2}.$$ 
\end{prop}

Delaying the proof of this proposition, we first complete the proof of (\ref{eq:youngzero}) and of Proposition \ref{prop:strongconv}. We show that if 
\begin{equation}\label{gebound}
		\G_\e(\ue)\leq \Lambda <+\infty,
\end{equation}
then there exists a constant $C>0$ such that, $\forall\,\ell>0$, $\forall\,k\in \N$ we have
\begin{equation}\label{limites}
		\lim_{\es \ra 0}\ \int_{\o \times \o} \rho^k(x-y)|P_\es(x)-P_\es(y)|\,d\mue(x)\,d\mue(y) \leq C\left( \frac{1}{k\ell} +\ell\right),
\end{equation}	
so that we can conclude (\ref{eq:youngzero}):
$$ \lim_{k \ra +\infty}\lim_{\es \ra 0}\ \int_{\o \times \o} \rho^k(x-y)|P_\es(x)-P_\es(y)|\,d\mue(x)\,d\mue(y) =0.$$
In order to prove (\ref{limites}) we examine the limits, as $\es\ra 0$, of the four members on the right-hand side of the inequality in Proposition \ref{estimatecomp}. %integrals $E_i(\es)$. 
In particular, we need to estimate $|\rho^k\conv\mue(x)|$. Recall that
$$\rho^k\conv\mue(x) = \int_\o \rho^k(x-y)\,d\mue(y).$$
By Lemma \ref{lemmaconvergence} we have that 
$$ \lim_{\es \ra 0}\mue =\mu:=\frac12 \L^2 \llcorner \o,\quad \mbox{weakly-\star\ in the sense of }RM(\o), $$
therefore, by basic properties of the convolution (see e.g. \cite[par 2.1]{afp})
\begin{equation}\label{convconv}
		\lim_{\es \ra 0} \rho^k\conv \mue = \rho^k \conv \mu,\quad \mbox{strongly in }C^0(\o).
\end{equation}		
Since
$$ {\|\rho^k\conv\mue\|}_{L^1(\o;\mue)}= \int_\o \rho^k\conv\mue(x)\, d\mue(x) =_{\ \ RM(\o)}\!\langle\mue \,  , \rho^k \conv \mue\rangle_{C^0(\o)},$$
we have
\begin{equation}\label{convlone}
		\lim_{\es \ra 0} {\|\rho^k\conv\mue\|}_{L^1(\o;\mue)} = \frac 14,
\end{equation}		
\begin{equation}\label{convltwo}
		\lim_{\es \ra 0} {\|\rho^k\conv\mue\|}_{L^2(\o;\mue)} = \frac {1}{2\sqrt{2}}.
\end{equation}		
\begin{equation}\label{convinfty}
		\lim_{\es \ra 0} {\|\rho^k\conv\mue\|}_{L^\infty(\o)} = \frac {1}{2}.
\end{equation}		

We also estimate

\begin{lemma}
	\begin{equation}\label{alphabound}
			\int_0^{L_\es} |\tilde{\alpha}_\e'(s)|^2 \,\es ds \leq C.
	\end{equation}
\end{lemma}
\begin{proof}
Consider again (see \pref{ae}-\pref{smallmass}) the set
$$ A_\es:=\left\{ s\in [0,L_\es]: \frac{M_\es(s)}{\es}\geq\frac 12 \right\}, $$
by which definition we have
$$ \F_\es(u_\es)-1 \geq \int_0^{L_\es} \left( 1- \frac{M_\es(s)}{\es}\right)^2 \es ds \geq \int_{A_\es^c} \frac 14 \,\es ds = \frac \es 4 |A_\es^c|,$$
and therefore by (\ref{gebound})
\begin{equation*}
		|A_\es^c|\leq C\es.
\end{equation*}
Then, by definition of $\tilde{\alpha}_\e$ (see (\ref{boundatilde}))
\begin{align*}
\int_{A_\es^c} |\tilde{\alpha}_\e'(s)|^2\, \es ds &\leq |A_\es^c|\left( \frac{2}{\es(1-\frac 12)}\right)^2 \es \leq C,\\
\int_{A_\es} |\tilde{\alpha}_\e'(s)|^2 \,\es ds &\leq 2^4\int_0^{L_\es} \left( \frac{M_\es(s)}{\es}\right)^4|\tilde{\alpha}_\e'(s)|^2\, \es ds \leq  \G_\e(\ue) \leq C.
\end{align*}
\end{proof}

Thus, using (\ref{convlone}), (\ref{convltwo}), (\ref{convinfty}), and (\ref{alphabound}) we compute
$$ \lim_{\es \ra 0} \int_{\o \times \o} \rho^k(x-y)|P_\es(x)-P_\es(y)|\,d\mue(x)\,d\mue(y)\quad \leq C\left( \frac {1}{k\ell} + \ell + 0 +0 \right).$$
This implies that $\forall\, \ell>0$
$$ \lim_{k \ra +\infty}\lim_{\es \ra 0}\ \int_{\o \times \o} \rho^k(x-y)|P_\es(x)-P_\es(y)|\,d\mue(x)\,d\mue(y) \leq C\ell,$$
and by the arbitrary choice of $\ell$ we obtain (\ref{eq:youngzero}).
\qed

\begin{proof}[Proof of Proposition \ref{estimatecomp}.]
\newcommand{\sumi}{\sum_{i=1}^J}
\newcommand{\sumj}{\sum_{j=1}^J}
\newcommand{\inti}{\int_0^{L_i}}
\newcommand{\intj}{\int_0^{L_j}}
\newcommand{\hsp}{\hspace{-0.3cm}}

For sake of notation, we drop the index $\es$ throughout this whole section. Owing to Proposition \ref{esttg}, it is sufficient to estimate separately the following four terms:
\begin{eqnarray*}
		E_1 &=& \ds \frac 1\ell \sum_{i=1}^J \sum_{j=1}^J\int_0^{L_i} \int_0^{L_j} \rho^k(\g_j(s)-\g_i(t)) |\g_j(s)-\g_i(t)|\,\es ds\,\es dt\\
		E_2 &=& \ds \ell^{1/2} \sum_{i=1}^J\sum_{j=1}^J \int_0^{L_i}\hspace{-0.2cm} \int_0^{L_j}\hspace{-0.2cm} \rho^k(\g_j(s)-\g_i(t))\min\left\{\int_{s-\ell}^{s+\ell}\hspace{-0.3cm}|\tilde{\alpha}'_j(\sigma)|^2d\sigma\, ,\,  2\!\!\int_{0}^{L_j} \hspace{-0.3cm} |\tilde{\alpha}'_j(\sigma)|^2d\sigma\right\}^{1/2} \hspace{-0.5cm}\es ds\,\es dt\\
		E_3 &=& \ds \frac 1\ell \sum_{i=1}^J\sum_{j=1}^J \int_0^{L_i}\hspace{-0.2cm} \int_0^{L_j}\hspace{-0.2cm} \rho^k(\g_j(s)-\g_i(t))\min\left\{\int_{s-\ell}^{s+\ell}\hspace{-0.4cm}|\tilde{\g}'_j(\sigma)-\g'_j(\sigma)|\, d\sigma ,\,  2\!\!\int_{0}^{L_j} \hspace{-0.3cm} |\tilde{\g}'_j(\sigma)-\g'_j(\sigma)|\,\sigma\right\}\es ds\,\es dt\\
		E_4 &=& \ds \sum_{i=1}^J \sum_{j=1}^J\int_0^{L_i} \int_0^{L_j} \rho^k(\g_j(s)-\g_i(t))|\tilde{\g}'_j(s)-\g'_j(s)|\,\es ds\,\es dt.
\end{eqnarray*}	
Recall that 
$$ \rho^k \conv \mue (x)=\int_\o \rho^k(x-y)d\mue(y)=\sum_{i=1}^J\int_0^{L_i} \rho^k(x-\g_i(t))\es dt,$$
and for $1\leq p<\infty$
$$ \| \rho^k \conv \mue \|^p_{L^p(\o;\mue)}=\int_\o |\rho^k \conv \mue(x)|^p d\mue(x)\leq (\es L)^{p-1}\sumj \sumi \intj \inti |\rho^k(\g_j(s)-\g_i(t))|^p \es dt\, \es ds.$$
\smallskip

\noindent \textit{Estimate for $E_1$.} Since supp$(\rho^k) = \overline{B(0,1/k)}$, it holds:
$$ E_1 \leq \frac 1\ell \sum_{i=1}^J\sum_{j=1}^J\int_0^{L_i} \int_0^{L_j} \rho^k(\g_j(s)-\g_i(t)) \frac 1k \,\es ds\,\es dt = \frac {1}{k\ell} \left\| \rho^k \conv \mu_\e\right\|_{L^1(\o;\mu_\es)}.$$
\smallskip

\noindent \textit{Estimate for $E_2$.}
Let $J_1 \subset \{1,\ldots,J\}$ be the set of indexes such that $\ell \leq L_j$, and $J_2$ be the set of indexes such that $\ell > L_j$. For every $j \in J_1$ let $N_j:=\lfloor \frac{L_j}{\ell} \rfloor$ and define $\tau:=L/N$. Then $ \ell \leq \tau < 2 \ell, $ and we can partition the interval $[0,L_j]$ into $N_j$ subsequent subintervals $I_j^n:=[n\tau,(n+1)\tau],\ n=0,\ldots,N_j-1$. Define also $I_j^{-1}:=[-\tau,0]$, and $I_j^{N_j}:=[L_j,L_j+\tau]$). 
Let
$$E_{2,1} :=  \ell^{1/2} \sum_{j\in J_1} \sumi \intj \hspace{-0.2cm} \inti\hspace{-0.2cm} \rho^k(\g_j(s)-\g_i(t))\left(\int_{s-\ell}^{s+\ell}|\tilde{\alpha}'_j(\sigma)|^2d\sigma\right)^{1/2} \hspace{-0.3cm}\es dt\,\es ds,$$
$$E_{2,2} :=  \ell^{1/2} \sum_{j\in J_2} \sumi \intj \hspace{-0.2cm} \inti \hspace{-0.2cm} \rho^k(\g_j(s)-\g_i(t))\left(2\!\int_{0}^{L_j}|\tilde{\alpha}'_j(\sigma)|^2d\sigma\right)^{1/2} \hspace{-0.3cm}\es dt\,\es ds,$$
(so that $E_2 = E_{2,1} + E_{2,2}$). If $s\in I_j^n$, then $[s-\ell, s+\ell]\subset I_j^{n-1}\cup I_j^{n}\cup I_j^{n+1}$, so that
$$ E_{2,1} \leq \ell^{1/2} \sum_{j\in J_1}\sum_{n=0}^{N_j-1} \sumi\int_{I_j^n} \inti \rho^k(\g_j(s)-\g_i(t)) \left(\int_{I_j^{n-1}\cup I_j^{n}\cup I_j^{n+1}} \hspace{-0.8cm}|\tilde{\alpha}'_j(\sigma)|^2d\sigma\right)^{1/2}\es dt\,\es ds  $$
$$ \leq \ell^{1/2} \sum_{j\in J_1}\sum_{n=0}^{N_j-1} \left( \int_{I_j^n}\sumi\inti \rho^k(\g_j(s)-\g_i(t)) \es dt\,\es ds\right) \left( \int_{I_j^{n-1}\cup I_j^{n}\cup I_j^{n+1}} \hspace{-0.8cm} |\tilde{\alpha}'_j(\sigma)|^2d\sigma \right)^{1/2}$$
$$ \leq \ell^{1/2} \sum_{j\in J_1}\sum_{n=0}^{N_j-1} \left( \int_{I_j^n} \rho^k \conv \mu_\e (\g_j(s))\, \es ds\right) \left( \int_{I_j^{n-1}\cup I_j^{n}\cup I_j^{n+1}} \hspace{-0.8cm} |\tilde{\alpha}'_j(\sigma)|^2d\sigma \right)^{1/2}.$$
Now we separate the integrals of $\rho^k$ and $\tilde{\alpha}'_j$ using H\"older's inequality. 
$$ E_{2,1} \leq \ell^{1/2} \sum_{j\in J_1}\left(\sum_{n=0}^{N_j-1} \left( \int_{I_j^n}\rho^k \conv \mu_\e (\g_j(s))\,\es ds\right)^2\right)^{1/2} \left(\sum_{n=0}^{N_j-1}\int_{I_j^{n-1}\cup I_j^{n}\cup I_j^{n+1}} \hspace{-0.8cm} |\tilde{\alpha}'_j(\sigma)|^2d\sigma \right)^{1/2}$$
$$ \leq \ell^{1/2} \left(\sum_{j\in J_1}\sum_{n=0}^{N_j-1} \left( \int_{I_j^n}\rho^k \conv \mu_\e (\g_j(s))\,\es ds\right)^2\right)^{1/2} \left(\sum_{j\in J_1} \sum_{n=0}^{N_j-1} \int_{I_j^{n-1}\cup I_j^{n}\cup I_j^{n+1}} \hspace{-0.8cm} |\tilde{\alpha}'_j(\sigma)|^2d\sigma \right)^{1/2}$$
$$ \leq \ell^{1/2} \left(\sum_{j\in J_1}\sum_{n=0}^{N_j-1} \left( \int_{I_j^n} \rho^k\conv\mue(\g_j(s))\,\es ds\right)^2\right)^{1/2}\left( 3  \sum_{j\in J_1} \intj |\tilde{\alpha}'_j(\sigma)|^2d\sigma \right)^{1/2}.$$
Using
$$ \int_{I_j^n} \rho^k\conv\mue (\g_j(s))\,\es ds \leq \es|I_j^n|\sup_{s\in [0,L_j]} |\rho^k\conv\mue (\g_j(s)) | \leq \es 2\ell \sup_{x \in \o} |\rho^k\conv\mue (x)| $$
we find
$$ E_{2,1} \leq \ell^{1/2} \left( 2\es \ell \|\rho^k\conv\mue\|_{L^{\infty}(\o)}\sum_{j\in J_1}\sum_{n=0}^{N_j-1} \int_{I_j^n} \rho^k\conv\mue (\g_j(s))\,\es ds \right)^{1/2} \left( 3 \sumj\intj |\tilde{\alpha}_j'(\sigma)|^2d\sigma \right)^{1/2}$$
$$\leq  \ell^{1/2} \left( 2\es \ell \|\,\rho^k\conv\mue\,\|_{L^{\infty}(\o)}\sumj\intj \rho^k\conv\mue(\g_j(s))\,\es ds \right)^{1/2} \left( 3\sumj \intj |\tilde{\alpha}'(\sigma)|^2d\sigma \right)^{1/2}$$
$$\leq  \ell\sqrt{2} \left( {\|\rho^k\conv\mue\|}_{L^{\infty}(\o)} {\|\rho^k\conv\mue\|}_{L^1(\o;\mue)} \right)^{1/2} \left( 3 \sumj \intj |\tilde{\alpha}'(\sigma)|^2\es\,d\sigma \right)^{1/2}.$$
In the case $\ell>L_j$, using again H\"older's inequality we find
$$ E_{2,2}\leq (2\ell)^{1/2} \left(\sum_{j\in J_2} \left(\intj \sumi\inti\hspace{-0.2cm} \rho^k(\g_j(s)-\g_i(t))\,\es dt\,\es ds\right)^2\right)^{1/2}\left(\sum_{j\in J_2}\!\int_{0}^{L_j}|\tilde{\alpha}'_j(\sigma)|^2d\sigma\right)^{1/2},$$
$$\leq (2\ell)^{1/2} \left(\sum_{j\in J_2} \left( \intj \rho^k\conv\mue(\g_j(s))\,\es ds\right)^2\right)^{1/2}\left( \sumj \intj |\tilde{\alpha}'_j(\sigma)|^2d\sigma \right)^{1/2}.$$
Arguing as before we find
$$ \intj \rho^k\conv\mue (\g_j(s))\,\es ds \leq \es|L_j|\sup_{s\in [0,L_j]} |\rho^k\conv\mue (\g_j(s)) | \leq \es \ell \sup_{x \in \o} |\rho^k\conv\mue (x)|, $$
so that
$$ E_{2,2} \leq (2\ell)^{1/2} \left( \es \ell \|\rho^k\conv\mue\|_{L^{\infty}(\o)}\sum_{j\in J_2} \intj  \rho^k\conv\mue (\g_j(s))\,\es ds \right)^{1/2} \left( \sumj\intj |\tilde{\alpha}_j'(\sigma)|^2d\sigma \right)^{1/2}$$
$$\leq  \ell\sqrt{2} \left( {\|\rho^k\conv\mue\|}_{L^{\infty}(\o)} {\|\rho^k\conv\mue\|}_{L^1(\o;\mue)} \right)^{1/2} \left( \sumj\intj |\tilde{\alpha}_j'(\sigma)|^2\,\es\,d\sigma \right)^{1/2}.$$

\smallskip

\noindent \textit{Estimate for $E_3$.}
Arguing as we did for $E_2$, we divide $E_3$ into
$$E_{3,1} :=  \frac 1\ell \sum_{j\in J_1} \sumi \intj \hspace{-0.2cm} \inti\hspace{-0.2cm} \rho^k(\g_j(s)-\g_i(t))\left(\int_{s-\ell}^{s+\ell}|\tilde{\g}'_j(\sigma)-\g'_j(\sigma)|\,d\sigma\right) \es dt\,\es ds,$$
$$E_{3,2} :=  \frac 1\ell \sum_{j\in J_2} \sumi \intj \hspace{-0.2cm} \inti \hspace{-0.2cm} \rho^k(\g_j(s)-\g_i(t))\left(2\!\intj|\tilde{\g}'_j(\sigma)-\g'_j(\sigma)|\,d\sigma\right) \es dt\,\es ds.$$
and, using Jensen's inequality,
	$$ \left(\int_{s-\ell}^{s+\ell}|\tilde{\g}'_j(\sigma)-\g'_j(\sigma)|\,d\sigma\right)^2\leq (2\ell)\int_{s-\ell}^{s+\ell}|\tilde{\g}'_j(\sigma)-\g'_j(\sigma)|^2\,d\sigma,$$
	$$ \left(2\intj|\tilde{\g}'_j(\sigma)-\g'_j(\sigma)|\,d\sigma\right)^2\leq (4L_j)\intj|\tilde{\g}'_j(\sigma)-\g'_j(\sigma)|^2\,d\sigma,$$
we compute:
\begin{align} E_{3,1}+E_{3,2} &\leq \frac C\ell\left( \es \ell {\|\rho^k\conv\mue\|}_{L^{\infty}(\o)} {\|\rho^k\conv\mue\|}_{L^1(\o;\mue)} \right)^{1/2}\left( \ell \sumj\intj |\tilde{\g}'_j(\sigma)-\g'_j(\sigma)|^2\,d\sigma \right)^{1/2},\nonumber\\
	&\leq  C\left({\|\rho^k\conv\mue\|}_{L^{\infty}(\o)} {\|\rho^k\conv\mue\|}_{L^1(\o;\mue)} \right)^{1/2}\left( \sumj\intj |\tilde{\g}'_j(\sigma)-\g'_j(\sigma)|^2\es\,d\sigma \right)^{1/2}.
\end{align}	
As in \pref{ae}, we define
$$ A_{\es,j}:=\left\{s\in[0,L_j]: \frac{M_j(s)}{\es}\geq\frac 12\right\}.$$
Then (see \pref{gate}-\pref{gate3} and \pref{bigmass}, \pref{smallmass}) it holds
$$ \intj |\tilde{\g}'_j(\sigma)-\g'_j(\sigma)|^2\es\, ds \leq C \int_{A_{\es,j}} \left(\frac{1}{\sin \beta_j(s)} -1\right)\left(\frac{M_j(s)}{\es} \right)^2 \es\, ds + C\int_{A_{\es,j}^c} \left(1-\frac{M_j(s)}{\es}\right)^2\es\, ds.$$
Finally, owing to Proposition \ref{prop:inequality} we obtain
\begin{equation}\label{estgammat2}
	\left( \sumj\intj |\tilde{\g}'_j(\sigma)-\g'_j(\sigma)|^2\es\,d\sigma \right)^{1/2} \leq C\es \left( \mathcal{G}_\es(u)\right)^{1/2}.
\end{equation} 
\smallskip

\noindent \textit{Estimate for $E_4$.}
$$ E_4= \sumj\intj \left(\sumi \inti \rho^k(\g_j(s)-\g_i(t))\, \es dt\right)|\tilde{\g}'_j(\sigma)-\g'_j(\sigma)|\,\es\, ds$$
$$  = \sumj \intj \rho^k\conv \mue (\g_j(s))\, |\tilde{\g}'_j(\sigma)-\g'_j(\sigma)|\,\es\, ds $$
$$ \leq \left( \sumj\intj (\rho^k\conv \mue (\g_j(s)))^2\es ds \right)^{1/2}\left( \sumj\intj |\tilde{\g}'_j(\sigma)-\g'_j(\sigma)|^2\,\es\, ds \right)^{1/2}$$
$$ \leq C\|\rho^k\conv \mue\|_{L^2(\o;\mue)}\es \left( \mathcal{G}_\es(u)\right)^{1/2},$$
where, in the last step, we used \pref{estgammat2}.
\end{proof}

%%%%%%%%%%%%%%%%%%%%%%%%%%%%%%%%%%%%%%%%%%%%%%%%%%%%%%%%%%%%%%%%%%%%%%%%%%%%%%%%%%%%%%%%%%%%%%%%%%%%%%%%%
%%%%%%%%%%%%%%%%%%%%%%%%%%%%%%%%%%%%%%   SECTION 4  %%%%%%%%%%%%%%%%%%%%%%%%%%%%%%%%%%%%%%%%%%%%%%%%%%%%%
%%%%%%%%%%%%%%%%%%%%%%%%%%%%%%%%%%%%%%%%%%%%%%%%%%%%%%%%%%%%%%%%%%%%%%%%%%%%%%%%%%%%%%%%%%%%%%%%%%%%%%%%%

%%%%%%%%%%%%%%%%%%%%%%%%%%%%%%%%%%%%%%%%%%%%%%%%%%%%%%%%%%%%%%%%%%%%%%%%%%%%%%%%%%%%%%%%%%%%%%%%%%%%%%%%%%%%%%
%%%%%%%%%%%%%%%%%%%%%%%%%%%%%%%%   THE LIMSUP ESTIMATE   %%%%%%%%%%%%%%%%%%%%%%%%%%%%%%%%%%%%%%%%%%%%%%%%%%%%%
%%%%%%%%%%%%%%%%%%%%%%%%%%%%%%%%%%%%%%%%%%%%%%%%%%%%%%%%%%%%%%%%%%%%%%%%%%%%%%%%%%%%%%%%%%%%%%%%%%%%%%%%%%%%%%

%\today
\section{The limsup estimate}\label{sec:limsup}
\setcounter{theo}{0}
\setcounter{equation}{0}

Throughout this section, $\o$ is an open, bounded,  connected subset of $\R^2$, with $C^2$ boundary, and $n$ is the outward normal unit vector to $\partial\o$. We recall that  $\mathcal{K}_0(\o)$ is defined as the set of all $P\in L^2(\o;\R^{2 \times 2})$ such that
\begin{eqnarray}
         P^2 = P && \mbox{a.e. in } \o,\label{prpro}\\
         \rank P=1 && \mbox{a.e. in } \o,\label{prrank}\\
         P \mbox{ is symmetric} && \mbox{a.e. in } \o,\label{prsymm}\\
         \div P \in L^2(\R^2;\R^2) && (\mbox{extended to $0$ outside  }\o),\label{divpint}\\
         P\, \div P = 0  && \mbox{a.e. in } \o.\label{pdivpzero}
\end{eqnarray}
\begin{rem}
\label{rem:divPL2}
        The sense of property (\ref{divpint}) is that the divergence of $P$ (extended to $0$ outside $\o$), in the sense of distributions in $\R^2$, is an $L^2(\R^2)$ function, i.e. there exists $C>0$ such that for any test function $\varphi \in C^{\infty}_c(\R^2,\R^2)$
\begin{equation}\label{distr}
    \left|\int_{\R^2}P(x):\nabla \varphi(x)\,dx\right| \leq C{\|\varphi\|}_{L^2(\R^2)}.
\end{equation}
Since for any $P\in H^1(\o)$
\begin{equation*}
         -\int_\o P:\nabla\varphi\, dx = \int_\o \div P \cdot \varphi\, dx - \int_{\partial\o}\!(P n)\cdot \varphi\, dS,
\end{equation*}
then (\ref{distr}) implies
\begin{equation}\label{ptgbound}
        Pn = 0\quad \mbox{in the sense of traces on }\partial\o.
\end{equation}
\end{rem}
\smallskip

In this section we construct a recovery sequence for each element of the limit set $\mathcal{K}_0(\o)$. Proposition~\ref{lemmalimsup} collects the relevant results, and provides the proof of part~\ref{thetheorem:part2} of Theorem~\ref{thetheorem}.

\begin{prop}[The limsup estimate]\label{lemmalimsup}
          Let $\Omega$ be a tubular neighbourhood of width $2\delta$ and of regularity $C^3$,  and let the sequence $\e_n\to0$ satisfy 
\begin{equation*}
\delta/2\e_n\in \N.
\end{equation*}
If $P\in \mathcal{K}_0(\o)$ there exists a sequence $\{u_n\} \subset K$ such that
	\begin{eqnarray*}
		&u_n \ras \frac 12 																					& \mbox{weakly-\star}\ \mbox{in}\ L^{\infty}(\o),\\
		&\mu_n:= \es_n|\nabla u_n|\ras \frac12 \L^2\,\llcorner\,\o 	& \mbox{weakly-\star}\ \mbox{in}\ RM(\o),\\
	  &(\mu_n,P_n)  \ra \left(\tfrac 12 \L^2 \llcorner\, \o,P\right), & \mbox{strongly, in the sense of Def. }\ref{def:strongconv},
  \end{eqnarray*}
%in the sense of strong convergence of measure-functions pairs introduced in Section \ref{definitions}, 
and
        
\begin{equation}\label{ungl:last}
	\limsup_{n\to\infty} \G_{\es_n}(u_n)\leq \frac 18 \int_\o |\,\mbox{\upshape{div}}\, P(x)|^2dx.
\end{equation}
\end{prop}
For this purpose, we will use the following characterization, given in \cite{PeletierVeneroniTA}:
\begin{theo}
\label{th:ee6}
Among domains $\Omega$ with $C^2$ boundary, $\mathcal K_0(\Omega)$ is non-empty if and only if $\Omega$ is a tubular domain. In that case $\mathcal K_0(\Omega)$ consists of a single element.
\end{theo}
Recall that a \emph{tubular domain} is a domain in $\R^2$ that can be written as
\[
\Omega = \Gamma + B(0,\delta),
\]
where $\Gamma$ is a simple, closed, $C^2$ curve in $\R^2$ with curvature $\kappa$ and $0<\delta< \|\kappa\|^{-1}_\infty$. In this case the \emph{width} of the domain is defined to be $2\delta$. The unique element $P\in \mathcal(\Omega)$ in the theorem is given by
\[
P(x) = \tau(\pi x)\otimes \tau(\pi x),
\]
where $\pi:\Omega\to\Gamma$ is the orthogonal projection onto $\Gamma$ (which is well-defined by the assumption on $\delta$) and $\tau(x)$ is the unit tangent to $\Gamma$ at $x$. 

\medskip
\begin{rem} By the strong compactness result in Theorem \ref{thetheorem}, any admissible sequence satisfying $\G_{\e_j}(u_{\e_j}) \leq C$ admits a subsequence such that the related measure-function pairs %$(\mu_{\e_j},P_{\e_j})$ 
strongly converge to a limit $(\frac12 \L,P)$, with $P\in \mathcal{K}_0(\o)$. Thanks to Theorem \ref{th:ee6} we know there is a unique such $P$ and we recover strong convergence for the whole sequence $(\mu_{\e_j},P_{\e_j})$. Thus, in the proof of Proposition \ref{lemmalimsup} we need only to build an admissible recovery sequence and to prove the limsup inequality \pref{ungl:last}.
\end{rem}

\drop{

We refer to \cite{PeletierVeneroniTA} for the proof of Theorem \ref{th:ee6}; here we briefly summarize the underlying ideas. A line field $P$ in $\mathcal{K}_0(\o)$ enjoys very restrictive properties, the main consequences of which are that\\
\begin{tabular}{ll}
        \textit{i)}& the projections lines are parallel to the boundary,\\
        \textit{ii)}& $P \in H^1(\o).$
\end{tabular}\\
Owing to an orientability criterium in two dimensions (see \cite{BAllZarnescu07TA}), these two facts imply that \\
\begin{tabular}{ll}
        \textit{iii)}& there exists a vector field $m \in H^1(\o;\R^2)$ such that $P=m \otimes m$,\\
        \textit{iv)}&  $ |m| = 1$  almost everywhere in $\o$,\\
        \textit{v)}& $\div m = 0  \mbox{ distributionally in } \R^2$.
\end{tabular}\\
Vector fields which satisfy \textit{iv)} and \textit{v)} (but not \textit{iii)}) were shown to have a very rigid behaviour (see \cite{jop}),  so that for $\o$ in a very wide class of domains, which for the moment we call $\mathcal{A}$, there exist\\
\centerline{ $\alpha \in \R$ \mbox{and} $\ds x_0\in \o:\ \ m(x)=\alpha\frac{(x-x_0)^\perp}{|x-x_0|}$\quad a.e. on $\o$.}
The first consequence is that if $ \o \in \mathcal{A}$ and it is bounded, then $\o$ is a disk, and the second is that if $\o$ is a disk, then the point $x_0$ is a singularity for the line field $P$, which forces
$$ \int_\o |\div P(x)|^2dx = +\infty,$$
in contradiction with (\ref{divpint}). The conclusion is that for $\o\in \mathcal{A}$ it holds: $\mathcal{K}_0(\o)=\emptyset$. Thus, we have to restrict ourselves to tubular neighbourhoods.
}

Owing to Theorem \ref{th:ee6}, in the following, we can parametrize the tubular domain $\o$ by level sets of a scalar map $\phi$, whose main properties are: 
\begin{lemma}\label{levelsets}
Let $P\in \mathcal{K}_0(\o)$ and let $\partial\o^0$ be one of the connected components of $\partial\o$, then
\begin{equation}\label{defphi}
     \o\ni x\mapsto\phi(x):=d(x,\partial\o^0),
\end{equation}
satisfies $\phi\in C^2(\o)$, $|\nabla \phi|\equiv 1 $ on $\Omega$, 
\begin{equation*}
        P(x)=\nabla \phi^\perp(x) \otimes \nabla \phi^\perp(x),
\end{equation*}
and it is possible to parametrize every $t$-level set of $\phi$ by a simple, closed, $C^2$-curve $\g_t: [0,L_t]\ra \o,\ |\g'|\equiv 1$, which satisfies
\begin{gather*}
    \{x\in\Omega: \phi(x)=t\}=\{\g_t(s):s\in [0,L_t]\},\\
        \g_t'(s)= \nabla \phi^\perp(\g_t(s))\quad \forall\,s \in [0,L_t],\\
       \g_t''(s)=\mbox{\upshape div} P(\g_t(s))\quad \forall\,s \in [0,L_t].
\end{gather*}
\end{lemma}

%%%%%%%%%%%%%%%%%%%%%%%%%%%%%%%%%%%%%%%%%%%%%%%%%%%%%%%%%%%%%%%%%%%%%%%%%%%%%%%%%%%%%%%%%%%%%%%%%%%%%%%%%%%%%%
%%%%%%%%%%%%%%%%%%%%%%%%%%%%%%%% HOW TO BUILD UE %%%%%%%%%%%%%%%%%%%%%%%%%%%%%%%%%%%%%%%%%%%%%%%%%%%%%%%%%%%%%
%%%%%%%%%%%%%%%%%%%%%%%%%%%%%%%%%%%%%%%%%%%%%%%%%%%%%%%%%%%%%%%%%%%%%%%%%%%%%%%%%%%%%%%%%%%%%%%%%%%%%%%%%%%%%%

\subsection{Building a recovery sequence $\ue$}\label{recovery}
Let $\o$ be a tubular neighbourhood of width $2\delta$, let $P\in \mathcal{K}_0(\o)$ be given and let $\phi\in C^2(\o)$ be the corresponding potential, as in Lemma \ref{levelsets}. The construction of the recovery sequence is an adaptation of the method introduced in \cite{PeletierRoegerTA} and here we divide it into three steps. First we divide the domain $\o$ into stripes $\S_\es$ according to the level sets of $\phi$, and we define a function $u_\es^h$ on $\S_\es$.
%\marginpar{We need to change notation here} 
Then on every stripe we compute the contribution to $\F_\es(u_\es^h)$ due to the length of the interface and we estimate from above the term due to the Wasserstein distance. Finally we glue together the functions on the stripes in order to get a function $\ue$ on the whole $\o$ and complete the proof of Proposition \ref{lemmalimsup}.  
\medskip

Let $\es$ such that $\delta/2\es \in \N$, set $N_\es:=\delta/2\es$, and define the stripes
\begin{equation}\label{stripes}
	\S^h_\es := \left\{ x\in \o: 4\es h \leq \phi(x)\leq 4\es(h+1) \right\},\ h=0,\ldots,N_\es-1.
\end{equation}
Note that $\o=\bigcup_{h=0}^{N_\es-1} \S^h_\es$ and that each stripe $\S^h_\es$ is a $2\es$-tubular neighbourhood of the curve $\{\phi=2\es(2h+1)\}$. In order to present the proof in a convenient way we exploit this particular geometry and we compute the estimates on the tubular neighbourhood of a generic $C^2$ closed curve $\g$. 
\medskip

\textbf{Step 1 - Construction of $u$ on a stripe $\S$}. 
Let $\g:[0,L]\ra \R^2$ be a $C^2$ closed curve, parametrized by arclength, and let $\S$ be the $2\es$-neighbourhood of $\g$. Let $\nu:[0,L] \ra S^1$ be the unit normal field of $\g$ that satisfies
$$ \nu(s)=\nabla \phi (\g(s)) $$
and let $\k(s)$ be the curvature of $\g$ in direction of $\nu(s)$,
\begin{equation}\label{defcurv}
        \k(s)= -\nu'(s)\cdot\g'(s)=\nu(s)\cdot\g''(s).
\end{equation}
Owing to the geometry of $S$ we can introduce the parametrization
$$ \Phi: [0,L]\times [-2\es,2\es]\ra \R^2,$$
$$ \Phi(s,t):= \g(s)+t\nu(s)$$
and we calculate
\begin{equation}\label{determinant}
    |\det \nabla\Phi(s,t)|=1-t\k(s)\quad \mbox{for }0<s<L,\ -2\es<t < 2\es.
\end{equation}
We recall the mass coordinates (see Def. \ref{defi:coordm})
$$ \mp_s(t)=\mp(s,t):=t-\frac{t^2}{2}\k(s)$$
and the inverse mapping $\tra_s(m)=\mp_s^{-1}(m)$ (see Prop. \ref{prop:coordt}), 
$$ \tra_s(m)=\tra(s,m):=\frac{1}{\k(s)}\left[ 1- (1-2m\k(s))^{1/2}\right].$$
Define two functions $ \rho_+\!,\, \rho_-:[0,L]\ra \R$
$$ \rho_+(s):= \frac{1}{\k(s)}\left( 1-(1-2\es \k(s) + 2\es^2\k(s)^2)^{1/2}\right),$$
$$ \rho_-(s):= \frac{1}{\k(s)}\left( 1-(1+2\es \k(s) + 2\es^2\k(s)^2)^{1/2}\right),$$
so that  %maybe it's more clear in this order:
$$ D^+:= \left\{ (s,t): 0\leq s <L,\ 0\leq t <\rho_+(s)\right\},$$
$$ D^-:= \left\{ (s,t): 0\leq s <L,\ \rho_-(s)<t\leq 0\right\},$$
divide $\S$ into two sets having the same area, i.e.
$$ \mp(s,\rho_+(s))=\frac 12 \mp(s,2\es)=\es(1-\es \k(s)), $$
$$ \mp(s,\rho_-(s))=\frac 12 \mp(s,-2\es)=-\es(1+\es \k(s)).$$
Finally we can set
\begin{equation}\label{defustr}
        u(x):=\left\{   \begin{array}{ll}
                                                        1& \mbox{if } x\in \Phi(D^+\!\cup D^-),\\
                                                        0& \mbox{otherwise}.
                                                \end{array} \right.
\end{equation}

%%%%%%%%%%%%%%%%%%%%%%%%%%%%%%%%%%%%%%%%%%%%%%%%%%%%%%%%%%%%%%%%%%%%%%%%%%%%%%%%%%%%%%%%%%%%%%%%%%%%%%%%%%%%%%
%%%%%%%%%%%%%%%%%%%%%%%%%%%%%%%%%%%  d_1(u,1-u)  %%%%%%%%%%%%%%%%%%%%%%%%%%%%%%%%%%%%%%%%%%%%%%%%%%%%%%%%%%%%%
%%%%%%%%%%%%%%%%%%%%%%%%%%%%%%%%%%%%%%%%%%%%%%%%%%%%%%%%%%%%%%%%%%%%%%%%%%%%%%%%%%%%%%%%%%%%%%%%%%%%%%%%%%%%%%

\textbf{Step 2 - How to compute \ $ d_1(u,1-u)$.}\ 
We now define an injective transport map $T:\Phi(D^+\!\cup D^-)\ra \R^2,$ between $\{u=1\}$ and $\{u=0\}$ on $\S$, 
$$ T(\Phi(s,t)):= \Phi(s,\tra(s,\mp(s,t)+\es(1-\es\k(s)))),\quad \mbox{ if }(s,t)\in D^+\!, $$
$$ T(\Phi(s,t)):= \Phi(s,\tra(s,\mp(s,t)-\es(1+\es\k(s)))),\quad \mbox{ if }(s,t)\in D^-\!. $$
First we show that $T$ is a proper transport map, i.e. that $\forall\,\eta\in C_c^0(\o)$
\begin{align*}
\int_\o &\eta(T(x))\,u(x)\,dx = \int_{\Phi(D^+\cup D^-)}\hspace{-0.6cm}\eta(T(x))\, dx=\int\!\int_{D^+\cup D^-}\hspace{-0.5cm} \eta(T(\Phi(s,t)))\,|\det \Phi(s,t)|\,ds\,dt\\
&=\int_0^L\int_{\rho_-(s)}^{\rho_+(s)} \eta(T(\Phi(s,t)))\,|\det \Phi(s,t)|\,ds\,dt\\
&=\int_0^L \int_{-\es(1+\es\k(s))}^{\es(1-\es\k(s))} \eta(T(\Phi(s,\tra(s,m))))\,|\det \Phi(s,\tra(s,m))|\,|\tra_s'(m)|\, dm\,ds\\
%$$ +\int_0^L \int_{-\es(1+\es\k(s))}^0 \eta(T(\Phi(s,\tra(s,m))))\,|\det \Phi(s,\tra(s,m))|\,|\partial_m \tra(s,m)|\, dm\,ds.$$
&=\int_0^L \int_{-\es(1+\es\k(s))}^{\es(1-\es\k(s))} \eta(T(\Phi(s,\tra(s,m))))\, \mp_s'(\mp_s^{-1}(m))\,(\mp_s^{-1})'(m)\, dm\,ds\\
&=\int_0^L \int_{-\es(1+\es\k(s))}^{\es(1-\es\k(s))} \eta(T(\Phi(s,\tra(s,m))))\, dm\,ds\\
&=\int_0^L\! \int_0^{\es(1-\es\k(s))} \hspace{-1.2cm} \eta(\Phi(s,\tra_s(m+\es(1-\es\k(s))))) \,dm\,ds +\int_0^L\! \int_{-\es(1+\es\k(s))}^0 \hspace{-1.2cm} \eta(\Phi(s,\tra_s(m-\es(1+\es\k(s))))) \,dm\,ds\\
&=\int_0^L\! \int_{\es(1-\es\k(s))}^{2\es(1-\es\k(s))} \hspace{-0.6cm} \eta(\Phi(s,\tra_s(m))) \,dm\,ds +\int_0^L\! \int_{-2\es(1+\es\k(s))}^{-\es(1+\es\k(s))} \hspace{-0.6cm} \eta(\Phi(s,\tra_s(m))) \,dm\,ds\\
& =\int_0^L\int_{\rho_+(s)}^{2\es} \eta(\Phi(s,t))\,|\det \Phi(s,t)|\,ds\,dt + \int_0^L\int_{-2\es}^{\rho_-(s)} \eta(\Phi(s,t))\,|\det \Phi(s,t)|\,ds\,dt\\
&=\int_\o \eta(x) (1-u(x))_{|\S}\, dx.
\end{align*}
In view of estimating from above the Wasserstein distance $d_1(u,1-u)$ we compute
\begin{align*}
        \int_{\phi(D^+)} |x-T(x)|\,u(x)\,dx &= \int_0^L\! \int_0^{\es(1-\es\k(s))} |\tra_s(m)-\tra_s(m+\es(1-\es\k(s)))|\,dm\,ds\\
        &=\int_0^L\! \int_0^{\es(1-\es\k(s))}\hspace{-0.5cm} (\tra_s(m+\es(1-\es\k(s)))-\tra_s(m))\,dm\,ds\\
        &=\int_0^L\! \int_{\es(1-\es\k(s))}^{2\es(1-\es\k(s))}\hspace{-0.5cm} \tra_s(m)\,dm\,ds- \int_0^L\! \int_0^{\es(1-\es\k(s))}\hspace{-0.5cm}\tra_s(m)\,dm\,ds.
\end{align*}
In order to simplify the following computations, let $\alpha,\beta\in\R$, and $M:=\es(1-\es\k(s))$. We find
$$ \int_\alpha^{\beta} \tra_s(m)\, dm = \frac{1}{3k^2}\left[ 3k(\beta - \alpha) + (1-2 k\beta)^{3/2}-(1-2k\alpha)^{3/2}\right],$$
$$ \int_M^{2M} \tra_s(m)\, dm -\int_0^{M} \tra_s(m)\, dm = \frac{1}{3k^2}\left[ (1-4Mk)^{3/2}-2(1-2Mk)^{3/2}+1\right].$$
By Taylor expansions (where $\xi=Mk$)
$$(1-4\xi)^{3/2} =1 -6\, \xi + 6\, \xi^2 +4\,\xi ^3 + 6\, \xi^4 + O(\xi^5), $$
$$2(1-2\xi)^{3/2}=2 -6\, \xi + 3\, \xi^2 +\,\xi ^3 + \frac 34\, \xi^4 + O(\xi^5), $$
we obtain
\begin{gather}
\notag
\int_M^{2M} \tra_s(m)\, dm -\int_0^{M} \tra_s(m)\, dm = \frac{1}{3\k^2}\left[3\, \xi^2 + 3\,\xi^3 +\frac{21}{4}\, \xi^4 + O(\xi^5)\right]\\
=\es^2 - \es^3\k(s) -\frac{1}{4}\es^4\k^2(s) +\es^5C(s).
\label{box1}
\end{gather}
In the same way we compute the transport on $\phi(D^-)$: let $M(s):=-\es(1+\es\k(s))$, then
\begin{align*}
        \int_{\phi(D^-)} |x-T(x)|\,u(x)\,dx &= \int_0^L\! \int_{M(s)}^0 |\tra_s(m)-\tra_s(m-\es(1+\es\k(s)))|\,dm\,ds\\
        &=\int_0^L\! \int_{M(s)}^0\hspace{-0.7cm} -(\tra_s(m+\es(1-\es\k(s)))-\tra_s(m))\,dm\,ds\\
        &=\int_0^L\! \int_{2M(s)}^{M(s)}\hspace{-0.7cm} -\tra_s(m)\,dm\,ds+ \int_0^L\! \int_{M(s)}^0\hspace{-0.5cm}\tra_s(m)\,dm\,ds,
\end{align*}
so that
\begin{equation}
\label{box2}
\int_{2M(s)}^{M(s)} -\tra_s(m)\, dm +\int_{M(s)}^0 \tra_s(m)\, dm =\es^2 + \es^3\k(s) -\frac{1}{4}\es^4\k^2(s) +\es^5C(s).
\end{equation}
Combining estimates~\pref{box1} and~\pref{box2} we find
\begin{equation}\label{supdist}%up to know no h-index and no epsilon
        \frac 1\es d_1(u,(1-u)_{|\S})\leq \frac 2\es  \int_0^{L} \es^2  -\frac{1}{4}\es^4\k^2(s) +\es^5C(s)\, ds  \leq 2 \es L  - \frac12\int_0^{L} \es^2\k^2(s)\,\es ds + \es^3 C
\end{equation}
for the function $u$ defined in (\ref{defustr}) which has support in $\S$.

%%%%%%%%%%%%%%%%%%%%%%%%%%%%%%%%%%%%%%%%%%%%%%%%%%%%%%%%%%%%%%%%%%%%%%%%%%%%%%%%%%%%%%%%%%%%%%%%%%%%%%%%%%%%%%
%%%%%%%%%%%%%%%%%%%%%%%%%%%%%%%%%%%%  perimeter  %%%%%%%%%%%%%%%%%%%%%%%%%%%%%%%%%%%%%%%%%%%%%%%%%%%%%%%%%%%%%
%%%%%%%%%%%%%%%%%%%%%%%%%%%%%%%%%%%%%%%%%%%%%%%%%%%%%%%%%%%%%%%%%%%%%%%%%%%%%%%%%%%%%%%%%%%%%%%%%%%%%%%%%%%%%%

\textbf{Step 3 - How to compute\ $\ds \int |\nabla u|$.}\ \
We compute the length of the curves which bound supp$(u)$. Define
$$ \gp(s):=\g(s) + \rho_+(s)\,\nu(s),\quad \mbox{and}\quad \gm(s):=\g(s) + \rho_-(s)\,\nu(s),$$
so that
$$ \gp'(s)= \g'(s)+\rho_+'(s)\,\nu(s)+\rho_+(s)\,\nu'(s),$$
and by (\ref{defcurv}) and $|\g'|\equiv 1$ we get
$$ L^+:=\int_0^L |\gp'(s)|\, ds =\int_0^L \left( 1 -2\es \k(s) + 2\es^2 \k^2(s) +(\rho_+')^2(s)\right)^{1/2}ds$$
We compute
$$ \rho_+' =-\frac{\k'}{\k^2}\left[1-\left(1+2\es\k+2(\es\k)^2\right)^{1/2} \right]-\frac{2\es\k'+4\es^2\k\k'}{2\k}\left( 1+2\es\k+2(\es\k)^2 \right)^{1/2},$$
and by Taylor expansion we find
$$ (\rho_+')^2(s)=C\es^4\left(\k'\right)^2(s) + O(\es^5),$$
%$$ f(x)=f(x_0) +f'(x_0)(x-x_0) +\frac{f''(x_0)}{2}(x-x_0)^2 +\frac{f'''(x_0)}{3\cdot 2}(x-x_0)^3 + \frac{f^{iv}(x_0)}{3\cdot 2^3}(x-x_0)^4 + O((x-x_0)^5)$$
%$$ f(x)=x^{1/2}\quad f'(x)=\frac12\, x^{-1/2}\quad f''(x)=-\frac14\, x^{-3/2}\quad f'''(x)=-\frac3{2^3}\, x^{-5/2}\quad f^{iv}(x)=\frac{15}{2^4}\,x^{-7/2} $$
%$$(1+\xi)^{1/2}= 1 + \frac12\, \xi -\frac{1}{2^3} \, \xi^2 -\frac{1}{2^4} \, \xi^3 + O(\xi^4)$$
%and
$$ L^+=\int_0^L \left( 1-\es \k(s) +\frac{\es^2 \k^2(s)}{2} + O(\es^3)\right) ds,  $$
and in the same way
$$ L^-:=\int_0^L |\gm'(s)|\, ds = \int_0^L \left( 1+\es \k(s) +\frac{\es^2 \k^2(s)}{2} + O(\es^3)\right) ds. $$
Therefore
\begin{equation}
\label{supper}
        \es \!\int_\o |\nabla u(x)|\, dx = 2\es L +\int_0^L\!  \es^2 \k^2(s)\ \es ds +\ O(\es^3).
\end{equation}

\begin{rem}
\label{rem:higher-regularity}
It is at this point that the $C^3$ regularity of $\Omega$ is required. The derivative $\kappa'$ enters in the higher-order terms $O(\e^3)$; therefore boundedness of $\kappa'$ is required for these terms to be actually of order $\e^3$. Note that since these terms vanish in the limit, the value of the derivative contributes nothing to the limit. 

Whether the boundedness requirement on $\kappa'$ is sharp is not clear. In a related context a method was developed to circumvent such a requirement (see~\cite[Ch.~7]{VanGennip08TH}); it is possible that a similar construction would apply to the case at hand.
\end{rem}

%%%%%%%%%%%%%%%%%%%%%%%%%%%%%%%%%%%%%%%%%%%%%%%%%%%%%%%%%%%%%%%%%%%%%%%%%%%%%%%%%%%%%%%%%%%%%%%%%%%%%%%%%%%%%%
%%%%%%%%%%%%%%%%%%%%%%%%%%%%%%%% PUTTING ALL TOGETHER %%%%%%%%%%%%%%%%%%%%%%%%%%%%%%%%%%%%%%%%%%%%%%%%%%%%%%%%
%%%%%%%%%%%%%%%%%%%%%%%%%%%%%%%%%%%%%%%%%%%%%%%%%%%%%%%%%%%%%%%%%%%%%%%%%%%%%%%%%%%%%%%%%%%%%%%%%%%%%%%%%%%%%%

\textbf{Conclusion.}\ Finally, collecting  the estimates (\ref{supdist}) and (\ref{supper}) we obtain
\begin{equation*}
        \es \!\int_\o |\nabla u(x)|\, dx + \frac1\es\, d_1(u,(1-u)_{|\S}) \leq 4 \es L\ +\  \es^2\!\! \int_0^L \frac{ \k^2(s)}{2}\,  \es ds + C\es^3,
\end{equation*}
for the function $u$ defined in (\ref{defustr}) which has support in $\S$.
\begin{rem} Since $\S$ is a tubular neighbourhood of the curve $\g$, we have\\
\centerline{Area$(\S)$ = Diameter$(\S)$ $\cdot$ Length$(\g)$,}
and therefore
\begin{equation*}
        |\S|=4\es L.
\end{equation*}
\end{rem}
\noindent We deduce that
\begin{equation}\label{almfinest}
    \frac{\es \!\int_\o |\nabla u(x)|\, dx + \es^{-1}\, d_1(u,1-u)-|\S|}{\es^2} \leq  \int_0^L \frac{ \k^2(s)}{2}\,  \es ds + C\es.
\end{equation}
Let 
\begin{equation}\label{defmut}
        \mbox{
        \begin{tabular}{l}
                $\g_\es^h:[0,L_\es^h]\ra \o$ be the parametrization, as in Lemma \ref{levelsets},\\
                of the set $\{\phi=2\es(2h+1)\}$, for $h=0,\ldots,N_\es-1$
        \end{tabular}
        }
\end{equation}
so that $\S_\es^h$ defined in (\ref{stripes}) is a $2\es$-neighbourhood of $\g_\es^h$. Following the construction of Step 1 we can define a function $u_\es^h$ on each $\S^h_\es$, as in (\ref{defustr}). Then let
$$ \ue(x):=\ue^{h}(x),\quad \mbox{ if }x\in \S_\es^h.$$
Since
$$ \int_\o |\nabla \ue(x)|\,dx = \sum_{h=0}^{N_\es-1} \int_{\S_\es^h} |\nabla \ue^h(x)|\,dx \quad \mbox{and}\quad d_1(\ue,1-\ue)\leq \sum_{h=0}^{N_\es-1} d_1(\ue^h,1-\ue^h),$$
owing to (\ref{almfinest}) we obtain
\begin{equation*}
    \frac{\F_\es(\ue)-|\o|}{\es^2} \leq \sum_{h=0}^{N_\es -1} \int_0^{L_\es^h} \frac{ {(\k_\es^h(s))}^2}{2}\,  \es\, ds + C\es,
\end{equation*}
where $\k_\es^h(s)=|\div P(\g_\es^h(s))|$ is the curvature in the direction of $\nabla \phi(\g_\es^h(s))$.
Define the system of curves $\Gamma_\es$ given by
\begin{equation}\label{defgamm}
        \Gamma_\es := \bigcup_{h=0}^{N_\es-1} \g_\es^h,
\end{equation}
and the corresponding measure
\begin{equation}\label{defmutt}
    \tilde{\mu}_\es:=\es \mathcal{H}^1 \llcorner\, \Gamma_\es.
\end{equation}
In order to conclude the proof of Proposition \ref{lemmalimsup} we need to show that 
\begin{equation}\label{mutildcon}
        \tilde{\mu}_\es \ras \frac14 \mathcal{L}\, \llcorner\,\o\quad \mbox{ in }RM(\o).
\end{equation}
Indeed, we have
\begin{equation}\label{limmut}
     \sum_{h=0}^{N_\es -1} \int_0^{L_\es^h} \frac{ {(\k_\es^h(s))}^2}{2}\, \es\, ds= \sum_{h=0}^{N_\es -1} \int_0^{L_\es^h} \frac{ \left|\div P(\g_\es^{h}(s))\right|^2}{2}\, \es\, ds =\frac12 \int_\o |\div P(x)|^2d\tilde{\mu}_\es(x),
\end{equation}
so that (\ref{mutildcon}) and (\ref{limmut}) imply
$$ \lim_{\es \ra 0} \sum_{h=0}^{N_\es -1} \int_0^{L_\es^h} \frac{ {(\k_\es^h(s))}^2}{2}\,  \es\, ds = \frac 18 \int_\o |\div P(x)|^2dx.$$
As a consequence, $\es^{-2}(\F_\es(\ue)-|\o|)$ is bounded and therefore, owing to Lemma \ref{lemmaconvergence},
$$ \es|\nabla \ue|\ras \frac12 \mathcal{L}\, \llcorner\, \o\quad \mbox{ in }RM(\o)$$
and
$$ \ue\, \mathcal{L}^2\,\llcorner\,\o \ras \frac12 \mathcal{L}\, \llcorner\, \o\quad \mbox{ in }RM(\o).$$
We make use of the following
\begin{lemma}\label{trivial}
Let $\o$ be a tubular neighbourhood of width $2\delta$ and let the map $ \o\ni x\mapsto\phi(x):=d(x,\partial\o^0)$ be as in (\ref{defphi}), then there exists $C=C(\o)>0$ such that
$$ \left|\int_{\{\phi=t\}}\hspace{-0.3cm}\eta\ d\mathcal{H}^1 -\int_{\{\phi=s\}}\hspace{-0.3cm} \eta \ d\mathcal{H}^1\right|\leq C\,\sigma(|s-t|),\quad\forall\,s,t\in [0,2\delta]\ \forall\,\eta\in C^0(\overline{\o}),$$
where $\sigma$ is the modulus of continuity of $\eta$.
\end{lemma}
\textit{Proof of Lemma \ref{trivial}.}\ \
As in Lemma \ref{levelsets}, let $\g_t:[0,L_t]\ra \o$ be an arclength parametrization of $\{\phi=t\}$. Owing to the geometry of $\o$ we also have the parametrization $\hat{\g}_t:[0,L_0]\ra \o,$ $\hat{\g}_t(\tau):=\g_0(\tau)+t\nu_0(\tau)$. Owing to (\ref{defcurv}) and (\ref{determinant}) it holds
\begin{equation*}
        \int_{\{\phi=t\}}\eta\ d\mathcal{H}^1 = \int_0^{L_t}\eta(\g_t(\tau))\,d\tau = \int_0^{L_0}\eta(\hat{\g}_t(\tau))(1-t\kappa_0(\tau))\,d\tau.
\end{equation*}
We compute
$$ \int_{\{\phi=t\}}\hspace{-0.3cm}\eta\ d\mathcal{H}^1 -\int_{\{\phi=s\}}\hspace{-0.3cm} \eta \ d\mathcal{H}^1= \int_0^{L_0}\eta(\hat{\g}_t(\tau)) - \eta(\hat{\g}_s(\tau))\,d\tau + \int_0^{L_0}\kappa_0(\tau)(s\eta(\hat{\g}_s(\tau)) - t\eta(\hat{\g}_t(\tau)))\,d\tau =$$
$$ =\int_0^{L_0}\eta(\hat{\g}_t(\tau)) - \eta(\hat{\g}_s(\tau))\,d\tau + \int_0^{L_0}s\kappa_0(\tau)(\eta(\hat{\g}_s(\tau)) - \eta(\hat{\g}_t(\tau)))+(s-t)\kappa_0(\tau)\eta(\hat{\g}_t(\tau))\,d\tau, $$
and we conclude
$$ \left| \int_{\{\phi=t\}}\hspace{-0.3cm}\eta\ d\mathcal{H}^1 -\int_{\{\phi=s\}}\hspace{-0.3cm} \eta \ d\mathcal{H}^1\right| \leq (|\partial\o^0|+2\delta\|\kappa_0\|_{L^1})\sigma(|s-t|)+\|\kappa_0\|_{L^1}|t-s|\|\eta\|_{C^0(\o)}.$$
\rightline{$\Box$}
\bigskip

\textit{Proof of (\ref{mutildcon}).}\ Let $\o$ be a tubular neighbourhood of width $2\delta$, let $\psi\in C^0_c(\o)$ and $ \o\ni x\mapsto\phi(x):=d(x,\partial\o^0)$ as in (\ref{defphi}), then by the coarea formula
\begin{equation}\label{coar}
    \int_\o \psi(x)\, dx =\int_\o \psi(x)|\nabla \phi(x)|\, dx= \int_0^{2\delta}\left( \int_{\phi=t}\psi\ d\mathcal{H}^1\right)dt.
\end{equation}
Define the functions\ \ $\ds g(t):=\int_{\phi=t}\psi\ d\mathcal{H}^1,$ and
$$ g_\es(t):= g((4h+2)\es)\quad \mbox{if}\ t\in [4h\es,4(h+1)\es[,\quad h=0,\ldots,N_\es-1.$$
By definitions (\ref{defmut}), (\ref{defgamm}) and (\ref{defmutt})
\begin{equation}\label{change}
        \int_0^{2\delta}    g_\es(t)\, dt = 4\es\sum_{h=0}^{N_\es-1} g((4h+2)\es) = 4\int_{\Gamma_\es}\psi(y)\,\es d\mathcal{H}^1(y)= 4\int_\o \psi(x)\, d\tilde{\mu}_\es(x).
\end{equation}
By Lemma \ref{trivial}
\begin{equation}\label{limuni}
    \lim_{\es \ra 0}\sup_{t\in[0,2\delta]}|g_\es(t)-g(t)|=0,
\end{equation}
therefore, by (\ref{coar}), (\ref{change}) and (\ref{limuni}), $\forall\,\psi\in C^0_c(\o)$
$$ \int_\o \psi(x)\, dx=\int_0^{2\delta}    \lim_{\es \ra 0}g_\es(t)\, dt= \lim_{\es \ra 0}\int_0^{2\delta} g_\es(t)\, dt=\lim_{\es \ra 0}4\int_\o \psi(x)\, d\tilde{\mu}_\es(x).$$
\rightline{$\Box$}
\bigskip

\appendix
\section{Appendix: A varifold interpretation}
\renewcommand{\theequation}{A.\arabic{equation}}

The result of compactness stated in Theorem \ref{thetheorem} may be naturally read in the language of the theory of varifolds. The effort made for further definitions and abstraction is paid back by the direct access to useful tools and concepts which are employed in the proof of the main result. In Theorem \ref{theo:varif} we restate Theorem \ref{thetheorem}  in terms of convergence of nonintegral varifolds and first variations.  
%\begin{theo}\label{theo:varif}
%Let the hypothesis of Theorem \ref{thetheorem} hold, define the sequence of 1-rectifiable varifolds $V_{\gme,\es}:=(id \times P_\es)_\# \es|\nabla \ue|$. There exists a unique $P\in \mathcal{K}_0$ and a subsequence of indexes $\{\es_j\}_{j\in \N}$  such that
%\begin{align*}
%		V_{\gme,\es} \ra V:=(id \times P)_\# \frac12\mathcal{L}^2 \h &\mbox{as varifolds on }G_1(\o)\\
%		\lim_{\es_j \ra 0} \delta V_{\gme,\es} (A) = \delta V(A) \h & \forall\ \mbox{open set }A \subset \o.
%\end{align*}		
%\end{theo}
%Notice that an application of Allard's Compactness Theorem would give the existence of a limit varifold $V$ such that
%$$ \liminf_{\es_j \ra 0} |\delta V_{\gme,\es}| (A) \geq |\delta V|(A) \h \forall\ \mbox{open set }A \subset \subset \o.$$
%We proved that the compactness enforced by $\G_\es$ is much stronger: instead of a lower bound we obtain a limit for the first variations and we have a precise characterization of the limit varifold which, note, is not 1-rectifiable as the elements of the sequence. 
\medskip

%\subsection{Varifolds}\label{introvari}
%Another point of view, using the language of Geometric Measure Theory, would be to regard these couples as \textit{rectifiable varifolds}. The main advantage of this perspective is the possibility of extending the notion of curvature to a larger class of manifolds, through a suitable integration by parts formula. 
We recall some basics definitions, referring to \cite{Hutchinson86} for a general introduction to the subject. 

Let $G_{1,2}$ be the Grassmann manifold consisting of all the $1$-dimensional subspaces of $\R^2$, we identify any element $P\in G_{1,2}$ with the orthogonal projection onto $P$ and therefore with a matrix in $R^{2\times 2}$. Let $\o$ be an open subset of $\R^2$ and define $G_1(\o):=\o \times G_{1,2}$, then a \textit{1-varifold in $\o$} is a Radon measure $V$ on $G_1(\o)$. 
\begin{defi} \textit{(Rectifiable varifolds)}\
Let $\Gamma$ be a 1-rectifiable set embedded in $\R^2$, let $\theta:\Gamma \ra (0,+\infty)$ be a Borel function locally integrable w.r.t. $\hf\llcorner \Gamma$ and let $\mu:=\theta\hf\llcorner \Gamma$, then a \textit{rectifiable 1-varifold} $V_{\Gamma,\theta}$ associated to $\Gamma$ is defined as
$$ V_{\Gamma,\theta}(A):=(id \times P)_\#\mu(A)\h \forall\, A\subset G_1(\R^2),$$
where $P$ is the $\mu$-measurable application which maps $x\in \R^2$ into the approximated tangent space $P(x)=ap\,T_x\Gamma$. 
\end{defi}
The function $\theta$ is called the \textit{density} of the varifold. For every bounded Borel function $\vfi:G_1(\R^2)\ra \R$ it holds
$$ \int_{G_1(\R^2)}\vfi\,dV_{\Gamma,\theta}=\int_\Gamma \vfi(x,P(x))\theta(x)\,d\hf(x).$$
Making use of these concepts, the couples $(\mue,P_\es)$ introduced in Section \ref{sec:mfpairs} can then be regarded as the rectifiable varifolds $V_{\gme,\es}=(id \times P_\es)_\#\mue\in RM(G_1(\o))$, associated to $\gme$, with constant density $\theta(x)=\es$. 

We introduce now the generalized mean curvature vector, in the sense of Allard. Define the first variation of a varifold $V$ as 
$$ \delta V(\eta):=\int_{G_1(\o)} \div\nolimits_P\,\eta(x)\,dV(x,P),\quad \forall\,\eta\in C^1_c(\o;\R^2)$$
where $\div_P\,\eta$ is the tangential divergence of the vector field $\eta$ with respect to $P$. If
$$ \sup \left\{ \delta V(\eta)\ : \ \eta\in C^1_c(\o;\R^2),\ \|\eta\|_{\infty}\leq 1\right\}<+\infty$$
then there exists a unique $H\in L^1_{loc}(\o,\mu;\R^2)$, called \textit{generalized mean curvature} and a unique vectorial Radon measure $\sigma$ such that
				$$ \delta V(\eta) = - \int_\o H(x)\cdot \eta(x)\,d\mu(x)- \int_\o \langle\eta(x),d\sigma\rangle,\quad \forall\,\eta\in C^1(\o;\R^2).$$
These varifolds are called \textit{varifolds with locally finite first variation}, or \textit{Allard's varifolds}. As a consequence, for every $\es>0$, $V_{\gme,\es}$ is an Allard's varifold and owing to a result by Brakke (see \cite{Brakke78}) the generalized mean curvature is almost everywhere orthogonal to the approximated tangent plane, i.e. 	
\begin{equation}\label{eq:curvperp}
	P_\es(x)H_\es(x)=0\h \mbox{for $\hf$-a.e. }x \in \gme.
\end{equation}
Finally we can state a varifold analogue of Theorem \ref{thetheorem}.\ref{thetheorem:part1}:
\begin{theo}\label{theo:varif}
Let the hypothesis of Theorem \ref{thetheorem} hold, define the sequence of 1-rectifiable varifolds $V_{\gme,\es}:=(id \times P_\es)_\# \es|\nabla \ue|$. There exists a unique $P\in \mathcal{K}_0$ and a subsequence of indexes $\{\es_j\}_{j\in \N}$  such that
\begin{align*}
		V_{\gme,\es} \ra V:=(id \times P)_\# \frac12\mathcal{L}^2 \h &\mbox{as varifolds on }G_1(\o)\\
		\lim_{\es_j \ra 0} \delta V_{\gme,\es} (A) = \delta V(A) \h & \forall\ \mbox{open set }A \subset \o.
\end{align*}		
\end{theo}
Note that an application of Allard's Compactness Theorem gives the existence of a limit varifold $V$ such that
$$ \liminf_{\es_j \ra 0} |\delta V_{\gme,\es}| (A) \geq |\delta V|(A) \h \forall\ \mbox{open set }A \subset \subset \o.$$
We prove that the compactness enforced by $\G_\es$ is much stronger: instead of a lower bound we obtain a limit for the first variations and we have a precise characterization of the limit varifold which, note, is not 1-rectifiable, in contrast to the elements of the sequence.

\small
\bibliography{helf-lib,marcok}
\bibliographystyle{plain}
\drop{
}
\end{document}